\documentclass[
a4paper,reqno]{amsart}
\usepackage[T1]{fontenc}
\usepackage[english]{babel}
\usepackage{amssymb,bbm,enumerate}
\usepackage{color}
\usepackage[linktocpage=true,colorlinks=true, linkcolor=blue, citecolor=red, urlcolor=green]{hyperref}

\usepackage{float}
\restylefloat{table}

\usepackage{tikz}

\renewcommand{\phi}{\varphi}
\renewcommand{\ker}{\Ker}

\newcommand{\mc}[1]{\mathcal{#1}}
\newcommand{\mf}[1]{\mathfrak{#1}}
\newcommand{\mb}[1]{\mathbb{#1}}

\newcommand{\id}{\mathbbm{1}}

\DeclareMathOperator{\Mat}{Mat}

\DeclareMathOperator{\End}{End}
\DeclareMathOperator{\diag}{diag}
\DeclareMathOperator{\tr}{tr}

\DeclareMathOperator{\ad}{ad}

\DeclareMathOperator{\Ker}{Ker}
\DeclareMathOperator{\Span}{Span}

\DeclareMathOperator{\sign}{sign}

\DeclareMathOperator{\gr}{gr}
\DeclareMathOperator{\rdet}{rdet}
\DeclareMathOperator{\cdet}{cdet}

\theoremstyle{plain}
\newtheorem{theorem}{Theorem}[section]
\newtheorem{lemma}[theorem]{Lemma}
\newtheorem{proposition}[theorem]{Proposition}
\newtheorem{corollary}[theorem]{Corollary}
\newtheorem{conjecture}[theorem]{Conjecture}

\theoremstyle{definition}
\newtheorem{definition}[theorem]{Definition}
\newtheorem{example}[theorem]{Example}

\theoremstyle{remark}
\newtheorem{remark}[theorem]{Remark}

\numberwithin{equation}{section}

\definecolor{light}{gray}{.9}

\begin{document}

\title{Finite $W$-algebras for $\mf{gl}_N$}

\author{Alberto De Sole}

\address{Dipartimento di Matematica, Sapienza Universit\`a di Roma,
P.le Aldo Moro 2, 00185 Rome, Italy}
\email{desole@mat.uniroma1.it}
\urladdr{www1.mat.uniroma1.it/\$$\sim$\$desole}

\author{Victor G. Kac}
\address{Dept of Mathematics, MIT,
77 Massachusetts Avenue, Cambridge, MA 02139, USA}
\email{kac@math.mit.edu}

\author{Daniele Valeri}
\address{Yau Mathematical Sciences Center, Tsinghua University, 100084 Beijing, China}
\email{daniele@math.tsinghua.edu.cn}


%

\maketitle

\begin{center}
\emph{To Dima Kazhdan, with admiration, for his 70-th birthday.}
\end{center}


\begin{abstract}
We study the quantum finite $W$-algebras $W(\mf{gl}_N,f)$, associated to the Lie algebra $\mf{gl}_N$,
and its arbitrary nilpotent element $f$.
We construct for such an algebra an $r_1\times r_1$ matrix $L(z)$ of Yangian type,
where $r_1$ is the number of maximal parts of the partition corresponding to $f$.
The matrix $L(z)$ is the quantum finite analogue of the operator of Adler type
which we introduced in the classical affine setup.
As in the latter case, the matrix $L(z)$ is obtained as a generalized quasideterminant.
It should encode the whole structure of $W(\mf{gl}_N,f)$,
including explicit formulas for generators and the commutation relations
among them.
We describe in all detail the examples of principal, rectangular and minimal nilpotent elements.
\end{abstract}

\tableofcontents

\section{Introduction}\label{sec:1}

As we all learned in our first linear algebra course, the characteristic polynomial
of an $N\times N$ matrix $E$,
\begin{equation}\label{intro:eq1}
P_E(z)=|z\id_N+E|\,,
\end{equation}
is invariant under the action of the group $G=GL_N(\mb F)$
by conjugation.
This can be restated in a more complete form as follows.
Consider the Lie algebra $\mf{gl}_N$ of $G$.
Let $e_{ij}\in\mf{gl}_N$, $i,j=1,\dots,N$, be the basis of elementary matrices,
and let $x_{ij}\in\mf{gl}_N^*$ be the dual basis.
Consider the matrix $E=(x_{ij})_{i,j=1}^N\in\Mat_{N\times N}\mf{gl}_N^*$
and its characteristic polynomial \eqref{intro:eq1}.
Then the coefficients of $P_E(z)$ are the free generators
of the algebra of invariants 
$\mb C[\mf{gl}_N]^{G}=S(\mf{gl}_N^*)^{G}$.

We identify $\mf{gl}_N$ with $\mf{gl}_N^*$ via the trace form $(A|B)=\tr(AB)$.
Then the element of the dual basis $x_{ij}$ corresponds to the elementary matrix $e_{ji}$.
Under this identification, the matrix $z\id_N+E$ becomes
\begin{equation}\label{eq:1}
z\id_N+E=\sum_{i,j=1}^N(e_{ji}+\delta_{ij}z)E_{ij}\in\Mat_{N\times N}\mf{gl}_N
\,.
\end{equation}
Here and further, in order to avoid confusion,
we denote by $e_{ij}$ the elementary matrices in $\mf{gl}_N$
(considered as commuting variables of the algebra $S(\mf{gl}_N$)
and by $E_{ij}$ the elementary matrices in $\Mat_{N\times N}\mb F$.
Then the above statement becomes the following.
Letting
$$
P_E(z)=|z\id_N+E|=
z^N+s_1z^{N-1}+\dots+s_N\,,
$$
we obtain
\begin{equation}\label{intro:eq2}
S(\mf{gl}_N)^{G}=\mb F[s_1,\dots,s_n]
\,,
\end{equation}
the polynomial algebra in the $N$ variables $s_1,\dots,s_N\in S(\mf{gl}_N)$.

Next, we can quantize the above result.
The quantization of $S(\mf{gl}_N)$ is the universal enveloping algebra $U(\mf{gl}_N)$,
and the quantization of the subalgebra of invariants $S(\mf{gl}_N)^{G}$
is the center of the universal enveloping algebra $Z(U(\mf{gl}_N))$.
Again, $Z(U(\mf{gl}_N))$ is the polynomial algebra $\mb F[z_1,\dots,z_N]$.
To construct the generators $z_1,\dots,z_N$,
one needs to consider the ``shifted'' matrix \eqref{eq:1}
by a diagonal matrix $D$:
\begin{equation}\label{intro:eq3}
z\id+E+D=\sum_{i,j=1}^N(e_{ji}+\delta_{ij}(z-(i-1))E_{ij}\in\Mat_{N\times N}\mf{gl}_N
\,,
\end{equation}
and take the row determinant:
\begin{equation}\label{intro:eq4}
\rdet(z\id+E+D)
=z^N+z_1z^{N-1}+\dots+z_N
\in Z(U(\mf{gl}_N))
\,.
\end{equation}
This result is known as the Capelli identity \cite{Cap1902}.
The shift matrix $D=\diag(0,-1,\dots,-N+1)$
is a purely quantum effect.
There is a similar formula involving the column determinant,
but with a different shift matrix.

\medskip

Next, we recall the definition of the finite $W$-algebra $W(\mf g,f)$
associated to a reductive Lie algebra $\mf g$ and its nilpotent element $f\in\mf g$.
By Morozov-Jacobson Theorem, the element $f$ can be included in an $\mf{sl}_2$-triple $\{f,h,e\}$.
Let $\mf g=\bigoplus_{j\in\frac12\mb Z}\mf g_j$
be the eigenspace decomposition with respect to $\frac12\ad h$.
The $W$-algebra $W(\mf g,f)$ can be defined as
\begin{equation}\label{intro:eq5}
W(\mf g,f)
=
\big(U(\mf{g})/I\big)^{\ad\mf g_{\geq\frac12}}
\,,
\end{equation}
where $I\subset U(\mf g)$ is the left ideal generated by the elements 
$m-(f|m)$, $m\in\mf g_{\geq1}$.
Though, the product of cosets of $I$ in $U(\mf g)/I$ is not well defined,
it is well defined on the subspace of $\ad\mf g_{\geq\frac12}$-invariants,
making $W(\mf g,f)$ an associative algebra.

Let $f^{pr}\in\mf g$ be the principal nilpotent element of $\mf g$.
In this case we have $\mf g_{\geq\frac12}=\mf g_{\geq1}=\mf n_+$, 
the subalgebra of strictly upper triangular matrices,
and $\mf g_{\leq\frac12}=\mf g_{\leq0}=\mf b_-$,
the subalgebra of lower triangular matrices.
Let $\pi_{-}:\,\mf g\twoheadrightarrow\mf b_-$ be the projection
with kernel $\mf n_+$.
By the PBW Theorem we have
$U(\mf g)\simeq U(\mf b_-)\otimes U(\mf n_+)$.
Consider the linear map 
\begin{equation}\label{intro:eq6}
\rho:\,U(\mf g)
\twoheadrightarrow 
U(\mf b_-)
\,,
\end{equation}
defined as follows: if $g=\sum_iA_iB_i\in U(\mf g)$ corresponds 
to $\sum_iA_i\otimes B_i\in U(\mf b_-)\otimes U(\mf n_+)$,
then $\rho(g)=\sum_i\chi(B_i)A_i\in U(\mf b_-)$,
where $\chi:\, U(\mf n_+)\twoheadrightarrow\mb F$
is the homomorphism induced by the character $\chi(a)=(f|a)$, $a\in\mf n_+$.
For example, for $a\in\mf g$, we have $\rho(a)=\pi_-(a)+\chi(a)$.
Note that $\ker\rho=I$, so $\rho$ induces 
a vector space isomorphism $U(\mf g)/I\simeq U(\mf g_{\leq0})$.
By the Kazhdan-Kostant Theorem \cite{Kos78},
the map $\rho$ restricts to an isomorphism of the center of $U(\mf g)$
to the principal finite $W$-algebra $W(\mf g,f^{pr})\subset U(\mf g)/I\simeq U(\mf b_-)$:
\begin{equation}\label{intro:eq7}
\rho:\,Z(U(\mf g))\,\stackrel{\sim}{\longrightarrow}W(\mf g,f^{pr})
\,.
\end{equation}

\medskip

As a consequence of the Capelli identity \eqref{intro:eq4}
and the Kazhdan-Kostant Theorem \eqref{intro:eq7},
we get a set of generators for the principal finite $W$-algebra of $\mf{gl}_N$:
$$
W(\mf{gl}_N,f^{\text{pr}})
=
\mb F[w_1,\dots,w_N]
\,,
$$
whose generating polynomial 
$L^{\text{pr}}(z)=z^N+w_1z^{N-1}+\dots+w_N$
can be expressed in terms of the row determinant
\begin{equation}\label{intro:eq8}
L^{\text{pr}}(z)=\rho\big(\rdet(z\id+E+D)\big)
\,.
\end{equation}
Even though the map $\rho$ is not an algebra homomorphism,
it is easy to check that $\rho$, in the RHS of \eqref{intro:eq8},
can be brought inside the row determinant.
This is stated (and proved) in Proposition \ref{intro:prop}.
As a consequence,
the generating polynomial $L^{\text{pr}}(z)$
of the generators $w_1,\dots,w_N$ 
of the principal $W$-algebra $W(\mf{gl}_N,f^{\text{pr}})$,
is obtained by
\begin{equation}\label{intro:eq13}
\begin{array}{l}
\displaystyle{
\vphantom{\Big(}
L^{\text{pr}}(z)
=
\rdet(z\id_N+F+\pi_-E+D)
} \\
\displaystyle{
\vphantom{\Big(}
=\rdet
\left(\begin{array}{ccccc}
e_{11}+z&e_{21}&e_{31}&\dots&e_{N1} \\
1&e_{22}+z-1&e_{32}&\dots&e_{N2} \\
0&1&\ddots&&\vdots \\
\vdots&\dots&\ddots&\ddots&\vdots \\
0&&&1&e_{NN}+z-N+1
\end{array}\right)
\,,}
\end{array}
\end{equation}
where 
\begin{equation}\label{intro:eq15}
F=\sum_{i=1}^{N-1}E_{i+1,i}\,\in\Mat_{N\times N}\mb F
\,,
\end{equation}
is the matrix corresponding to the nilpotent $f^{\text{pr}}\in\mf{gl}_N$.
Note that, while in \eqref{intro:eq8} the row determinant cannot be replaced by the column determinant,
the RHS of \eqref{intro:eq13} is unchanged if we take the column determinant in place of the row determinant
(cf. Proposition \ref{intro:prop2}).

Formula \eqref{intro:eq13} is well known \cite{BK06}.
In order to find a correct generalization for an arbitrary nilpotent $f\in\mf{gl}_N$,
we observe that \eqref{intro:eq13} can be expressed in terms of quasideterminants,
using Proposition \ref{intro:prop2}:
\begin{equation}\label{intro:eq14}
L^{\text{pr}}(z)
=
|z\id_N+F+\pi_-E+D|_{1N}
\,.
\end{equation}
The polynomial $L^{\text{pr}}(z)$ in \eqref{intro:eq14} is generalized as follows.
First, we replace 
the matrix \eqref{intro:eq15} by the matrix $F\in\Mat_{N\times N}\mb F$
corresponding to the given nilpotent element $f$.
Second, we replace $\pi_-$ by $\pi_{\leq\frac12}$,
the projection $\mf g\twoheadrightarrow\mf g_{\leq\frac12}$ with kernel $\mf g_{\geq1}$.
Third, we replace the shift matrix $D$ by an appropriate diagonal matrix over $\mb F$,
depending on the nilpotent element $f$, described in Section \ref{sec:4.2}.
Finally, we replace the $(1N)$-quasideterminant in \eqref{intro:eq14}
by an appropriate generalized $(I_1J_1)$-quasideterminant,
where $I_1$ and $J_1$ are the matrices \eqref{eq:factor1}.
Hence, we get the $r_1\times r_1$ matrix (where $r_1$ is the number of Jordan blocks of $F$):
\begin{equation}\label{intro:eq16}
L(z)
=
|z\id_N+F+\pi_{\leq\frac12}E+D|_{I_1J_1}
\,,
\end{equation}
whose entries are Laurent series in $z^{-1}$.

Our first main Theorem \ref{thm:main1}
says that the coefficients of this Laurent series lie in the $W$-algebra $W(\mf{gl}_N,f)$.
Note that the same formula \eqref{intro:eq16},
with $z$ replaced by the derivative $\partial$,
appeared in our work on classical affine $W$-algebras \cite{DSKV16b}.
The new ingredient here is the shift matrix $D$, which is a purely quantum effect.
In \cite{DSKV16b} we also prove that the classical affine analogue of the matrix \eqref{intro:eq16}
satisfies what we call the ``Adler identity'',
from which we derive that this operator satisfies a hierarchy of Lax type equations,
and therefore gives rise to an integrable system of bi-Hamiltonian PDE.
At the same time the Adler identity allowed us to construct the generators
of the classical affine $W$-algebra associated to $f\in\mf{gl}_N$
and to compute $\lambda$-brackets between them.
The second main result of the present paper, Theorem \ref{thm:main2},
says that the matrix \eqref{intro:eq16}
satisfies the quantum finite analogue of the Adler identity,
which we call the Yangian identity:
\begin{equation}\label{intro:eq17}
(z-w)[L_{ij}(z), L_{hk}(w)]
= 
L_{hj}(w)L_{ik}(z)- L_{hj}(z)L_{ik}(w)
\,.
\end{equation}
This identity is well known in the theory of quantum groups,
since it defines the generating series of the generators of the Yangian of $\mf{gl}_{r_1}$
\cite{Dr86,Mol07}.
As a corollary of Theorem \ref{thm:main2}
we get an algebra homomorphism from the Yangian $Y(\mf{gl}_{r_1})$
to the $W$-algebra $W(\mf{gl}_N,f)$.
Brundan and Kleshchev in \cite{BK06}
also found a close connection between the Yangians 
and the quantum finite $W$-algebras associated to $\mf{gl}_N$,
but the precise relation of their approach to our approach
is unclear and is still under investigation \cite{Fed16}.

Even though the expression \eqref{intro:eq16} for the matrix $L(z)$
could be easily guessed (except for the shift matrix $D$)
from the classical affine analogue in \cite{DSKV16b},
the proof of our two main results
is more difficult and completely different from the classical case.
The main reason being that, unlike in the classical case, 
the map $\rho$ generalizing \eqref{intro:eq6} (cf. Section \ref{sec:3.2})
is not an algebra homomorphism (even for the principal nilpotent).

First, in order to use arguments similar to those in \cite{DSKV16b}
for the proofs of Theorems \ref{thm:main1} and \ref{thm:main2},
we rewrite the matrix \eqref{intro:eq16}
(or rather its image via the map $\rho$), in a different way.
This is the content of the Main Lemma \ref{lem:main},
which can be equivalently stated by the following identity:
\begin{equation}\label{intro:eq18}
\rho(L(z))
=
\rho(|z\id_N+E|_{I_1J_1})
\,.
\end{equation}

The proof of equation \eqref{intro:eq18}
is the result of a long (and beautiful) computation,
which is performed in Sections \ref{step2} -- \ref{step5}.
There is also a conceptual difficulty in giving the right interpretation
to the quasideterminant in the RHS of equation \eqref{intro:eq18}.
Indeed, the entries of the quasideterminant $|z\id_N+E|_{I_1J_1}$
do not lie in the algebra $U(\mf g)((z^{-1}))$,
as they involve fractions.
On the other hand,
the map $\rho$ is not defined on the whole skewfield of fractions $\mc K$ of $U(\mf g)((z^{-1}))$.
Therefore, the problem is to find a sufficiently large subring of $\mc K$
containing all the entries of $|z\id_N+E|_{I_1J_1}$,
on which $\rho$ is defined.
To construct such ring,
we consider the completed Rees algebra $\mc RU(\mf g)$
associated to the Kazhdan filtration,
and its localization $\mc R_\infty U(\mf g)$, defined in Section \ref{sec:5.4}.
An open problem is whether this localization is Ore,
which would simplify our arguments, cf. Remark \ref{0410:rem}.

In \cite{DSKV16b}
we derived from the classical affine analogues of the two main theorems mentioned above
an algorithm for finding an explicit set of free generators
of the classical affine $W$-algebra
and $\lambda$-brackets between them.
The quantum finite analogue of these results
is proposed in Section \ref{sec:conjecture} of the present paper
only as a conjecture.
In Section \ref{sec:Examples} we test this conjecture for the rectangular and minimal nilpotents,
and we use Theorems \ref{thm:main1} and \ref{thm:main2}
to find the generators of $W(\mf{gl}_N,f)$ and their commutators
in these examples.
In particular, for the principal nilpotent element $f^{\text{pr}}$
we recover that $L(z)$ is the generating polynomial for the generators of $W(\mf{gl}_N,f^{\text{pr}})$,
discussed above.

\medskip

Throughout the paper, the base field $\mb F$ is an algebraically closed field of characteristic $0$,
and all vector spaces are considered over the field $\mb F$.

\subsubsection*{Aknowledgments} 

We are extremely grateful to Pavel Etingof for giving always interesting insights to all our questions.
We also thank Toby Stafford, Michel Van den Bergh and last, but not least, Andrea Maffei,
for correspondence.
The first author would like to acknowledge
the hospitality of MIT, where he was hosted during the spring semester of 2016.
The second author would like to acknowledge
the hospitality of the University of Rome La Sapienza
and the support of the University of Rome Tor Vergata,
during his visit in Rome in January 2016.
The third author is grateful to the University of Rome La Sapienza
for its hospitality during his visit in the months of December 2015 and January and February 2016.
The first author is supported by National FIRB grant RBFR12RA9W,
National PRIN grant 2012KNL88Y, and University grant C26A158K8A,
the second author is supported by an NSF grant,
and the third author is supported by an NSFC ``Research Fund for International
Young Scientists'' grant.

\section{Operators of Yangian type and
generalized quasideterminants
}\label{sec:2}

\subsection{
Operators of Yangian type
}\label{sec:2.1}

\begin{definition}\label{def:adler}
Let $U$ be an associative algebra. 
A matrix $A(z)\in\Mat_{M\times N}U((z^{-1}))$
is called an operator of \emph{Yangian type} for the associative algebra $U$
if, for every $i,h\in\{1,\dots,M\},\,j,k\times\{1,\dots,N\}$, 
the following \emph{Yangian identity} holds
\begin{equation}\label{eq:adler}
(z-w)[A_{ij}(z), A_{hk}(w)]
= 
A_{hj}(w)A_{ik}(z)- A_{hj}(z)A_{ik}(w)
\,.
\end{equation}
\end{definition}
Here and further the bracket stands for the usual commutator in the associative algebra $U$.
In particular, for $M=N=1$ equation \eqref{eq:adler} means that
$(z-w+1)[A_{11}(z), A_{11}(w)]=0$, 
hence it is equivalent to $[A_{11}(z),A_{11}(w)]=0$,
which means that all the coefficients of $A_{11}(z)$ commute.

\begin{remark}\label{rem:terminology}
Equation  \eqref{eq:adler}
is, up to an overall sign, the defining relation for the Yangian of $\mf{gl}_N$
(\cite{Dr86}, see also \cite{Mol07}), hence the name.
Its classical version is the same identity for a Poisson algebra,
where the bracket on the left is replaced by the Poisson bracket.
The ``chiralization'' of the latter
was introduced in the context of Poisson vertex algebras
under the name of Adler identity \cite{DSKV15,DSKV16a,DSKV16b}.
The ``chiralization'' of the Yangian identity,
in the context of vertex algebras,
has not been investigated yet.
We are planning to do it in a future publication.
\end{remark}
\begin{example}\label{ex:A}
Let $U=U(\mf{gl}_N)$.
Denote by $E_{ij}\in\Mat_{N\times N}\mb F$ the elementary matrix 
with $1$ in position $(ij)$ and $0$ everywhere else.
We shall also denote by $e_{ij}\in\mf{gl}_N$
the same matrix when viewed as an element of the associative algebra $U(\mf{gl}_N)$,
and by
\begin{equation}\label{eq:E}
E=\sum_{i,j=1}^Ne_{ji}E_{ij}
\,\in\Mat_{N\times N} U(\mf{gl}_N)
\,,
\end{equation}
the $N\times N$ matrix which, in position $(ij)$,
has entry $e_{ji}\in U(\mf{gl}_N)$.
Then, it is easily checked that the $N\times N$ matrix 
\begin{equation}\label{eq:A}
A(z)
=
z\id_N+E
\end{equation}
is an operator of Yangian type for $U(\mf{gl}_N)$.

We have the decomposition $\mf{gl}_N=\mf b_-\oplus\mf n_+$,
where $\mf b_-$ consists of lower triangular matrices,
and $\mf n_+$ of strictly upper triangular elements.
Let $\pi_-:\,\mf{gl}_N\twoheadrightarrow\mf b_-$
be the projection with kernel $\mf n_+$.
Then, by the PBW Theorem, $U(\mf{gl}_N)\simeq U(\mf b_-)\otimes U(\mf n_+)$.
Define the linear map $\rho:\, U(\mf{gl}_N)\to U(\mf b_-)$
by letting $\rho(b_-n_+)=\chi(n_+)b_-$,
for every $b_-\in U(\mf b_-)$ and $n_+\in U(\mf n_+)$,
where $\chi:\,U(\mf n_+)\to\mb F$ is the character defined on $\mf n_+$
by $\chi(e_{i,i+1})=1$, $\chi(e_{ij})=0$ if $j>i+1$.
(This map $\rho$ is a special case,
for $\mf g=\mf{gl}_N$, $f=\sum_{i=1}^{N-1}e_{i+1,i}$,
of the linear map $\rho$ introduced in Section \ref{sec:3.2}.)
Note that, applying the map $\rho$ to the entries of the matrix $E$ in \eqref{eq:E}, we get
\begin{equation}\label{intro:eq12}
\rho(E)=F+\pi_-E
\,,
\end{equation}
where $\pi_-E=\sum_{i\leq j}e_{ji}E_{ij}$, and 
$F=\left(\begin{array}{llll} 
0&\dots &\dots& 0 \\ 
1&\ddots&&\vdots \\ 
\vdots&\ddots&\ddots&\vdots \\ 
0&\dots&1&0 
\end{array}\right)$
is a single nilpotent Jordan block.
Even though $\rho$ is not an algebra homomorphism,
in $\rho(\rdet(E))$ we can pull $\rho$ inside the row determinant.
This is a consequence of the following:
\begin{proposition}\label{intro:prop}
If $E$ is the matrix \eqref{eq:E} and $D$ is any matrix
with entries in the field $\mb F$, we have
\begin{equation}\label{intro:eq9}
\rho(\rdet(E+D))=\rdet(\pi_-E+F+D)\,.
\end{equation}
\end{proposition}
\begin{proof}
First, we prove the claim for $D=0$.
Recall that the row determinant of the matrix $E$ is obtained by the expansion
$$
\rdet(E)=
\sum_{\sigma\in S_N}\sign(\sigma)e_{\sigma(1)1}e_{\sigma(2)2}\dots e_{\sigma(N)N}
\,.
$$
Hence, to prove equation \eqref{intro:eq9},
it is enough to prove that, for every permutation $\sigma\in S_N$,
we have
\begin{equation}\label{intro:eq10}
\begin{array}{l}
\displaystyle{
\vphantom{\Big(}
\rho(e_{\sigma(1)1}e_{\sigma(2)2}\dots e_{\sigma(N)N})
} \\
\displaystyle{
\vphantom{\Big(}
=
(\pi_-e_{\sigma(1)1})(\pi_-e_{\sigma(2)2}+\delta_{\sigma(2),1})\dots 
(\pi_-e_{\sigma(N)N}+\delta_{\sigma(N),N-1})
\,.}
\end{array}
\end{equation}
Recalling the definition of the map $\rho$,
in order to compute the LHS of \eqref{intro:eq10}
we need to permute the factors in $e_{\sigma(1)1}e_{\sigma(2)2}\dots e_{\sigma(N)N}$
(using the commutation relations of $U(\mf{gl}_N)$)
so that all factors from $\mf b_-$ are on the left and all factors from $\mf n_+$ are on the right.
On the other hand,
if we have two factors in the wrong order,
i.e. $e_{\sigma(i)i}\in\mf n_+$ and $e_{\sigma(j)j}\in\mf b_-$,
with $i<j$,
then we have $\sigma(i)<i<j\leq\sigma(j)$,
and therefore $e_{\sigma(i)i}$ and $e_{\sigma(j)j}$ commute in $U(\mf{gl}_N)$.
It follows that equation \eqref{intro:eq10} holds.
The proof for arbitrary constant matrix $D$ is the same, by letting $\widetilde{e}_{ij}=e_{ij}+\alpha_{ij}$,
where $\alpha_{ij}\in\mb F$ are the entries of the matrix $D$,
and repeating the same argument with $\widetilde{e}_{ij}$ in place of $e_{ij}$.
\end{proof}
\end{example}
Note that
equation \eqref{intro:eq9} does NOT hold if we replace the row determinant by the column determinant.
Indeed, the LHS of \eqref{intro:eq9} changes if we replace the row determinant by the column determinant.
For example, for $N=2$, we have $\rho(\cdet E)=\rho(\rdet E)+e_{11}-e_{22}$.
On the other hand, if $D$ is a constant diagonal matrix,
the RHS of \eqref{intro:eq9} is unchanged if we replace the row determinant
by the column determinant,
or even by the quasi-determinant in position $(1,N)$ (cf. Section \ref{sec:2.3}).
This is a consequence of the following fact:
\begin{proposition}\label{intro:prop2}
Let $A\in\Mat_{N\times N}R$ be a matrix 
with entries in a  unital associative algebra $R$,
of the following form:
$$
A=
\left(\begin{array}{llll} 
a_{11}&a_{12} &\dots& a_{1N} \\ 
1&a_{22}&\dots&a_{2N} \\ 
\vdots&\ddots&\ddots&\vdots \\ 
0&\dots&1&a_{NN} 
\end{array}\right)
\,.
$$
Then 
\begin{equation}\label{intro:eq11}
\rdet(A)=\cdet(A)=|A|_{1N}\,,
\end{equation}
where the quasideterminant $|A|_{1N}$ can be defined as (cf. Proposition \ref{0304:prop})
$$
|A|_{1N}:=
a_{1N}-
(a_{11}\,a_{12}\,\dots\,a_{1,N-1})
\left(\begin{array}{cccc} 
1&a_{22}&\dots&a_{2,N-1} \\ 
0&1&\ddots&\vdots \\
\vphantom{\Bigg(}
\vdots&\ddots&\ddots&a_{N\!-\!1\!,N\!-\!1\!} \\ 
0&\dots&0&1
\end{array}\right)^{-1}
\left(\begin{array}{c}a_{2N} \\ a_{3N} \\ \vdots \\ a_{NN}\end{array}\right)
\,.
$$
\end{proposition}
\begin{proof}
It is a simple linear algebra exercise.
\end{proof}

\subsection{
Generalized quasideterminants
}\label{sec:2.2}

We recall, following \cite{DSKV16a}, the definition of a generalized quasideterminant, 
cf. \cite{GGRW05}.
\begin{definition}\label{def:gen-quasidet}
Let $\mc U$ be a unital associative algebra.
Let $A\in\Mat_{N\times N}\mc U$,
$I\in\Mat_{N\times M}\mc U$, and $J\in\Mat_{M\times N}\mc U$,
for some $M\leq N$.
The $(I,J)$-\emph{quasideterminant} of $A$ is
\begin{equation}\label{eq:gen-quasidet}
|A|_{IJ}
=
(JA^{-1}I)^{-1}\,\in\Mat_{M\times M}\mc U
\,,
\end{equation}
assuming that the RHS exists, namely that $A$ is invertible in $\Mat_{N\times N}\mc U$
and that $JA^{-1}I$ is invertible in $\Mat_{M\times M}\mc U$.
\end{definition}

A special case of a (generalized) quasideterminant is the following.
Given a matrix $A=(a_{ij})_{i,j=1}^N\in\Mat_{N\times N}\mc U$ and two ordered subsets 
\begin{equation}\label{0304:eq1}
\mc I=\{i_1,\dots,i_m\}
\,,\,\,
\mc J=\{j_1,\dots,j_n\}
\,\,
\subset\{1,\dots,N\}
\,,
\end{equation}
we denote by $A_{\mc I\mc J}$ the $m\times n$ submatrix of $A$
obtained by taking rows from the set $\mc I$ and columns from the set $\mc J$:
\begin{equation}\label{0304:eq2}
A_{\mc I\mc J}
=
\left(\begin{array}{ccc}
a_{i_1j_1} & \dots & a_{i_1j_n} \\
\vdots&&\vdots \\
a_{i_mj_1} & \dots & a_{i_mj_n}
\end{array}\right)
=
\big(a_{i_\alpha j_\beta}\big)_{\substack{\alpha\in\{1,\dots,m\} \\ \beta\in\{1,\dots,n\}}}
\,.
\end{equation}
To the subsets $\mc I$ and $\mc J$ we attach 
the following two matrices:
\begin{equation}\label{0304:eq3}
I
=
\big(\delta_{i_{\alpha},i}\big)_{\substack{\alpha\in\{1,\dots,m\} \\ i\in\{1,\dots,N\}}}
\in\Mat_{m\times N}\mb F
\,\,,\,\,\,\,
J
=
\big(\delta_{j,j_\beta}\big)_{\substack{j\in\{1,\dots,N\} \\ \beta\in\{1,\dots,n\}}}
\in\Mat_{N\times n}\mb F
\,.
\end{equation}
and their transpose matrices
\begin{equation}\label{0304:eq4}
I_1
=
\big(\delta_{i,i_{\alpha}}\big)_{\substack{i\in\{1,\dots,N\} \\ \alpha\in\{1,\dots,m\}}}
\in\Mat_{N\times m}\mb F
\,\,,\,\,\,\,
J_1
=
\big(\delta_{j_\beta,j}\big)_{\substack{\beta\in\{1,\dots,n\} \\ j\in\{1,\dots,N\}}}
\in\Mat_{n\times N}\mb F
\,.
\end{equation}
Then, we have the following formula for the submatrix \eqref{0304:eq2}:
\begin{equation}\label{0304:eq5}
A_{\mc I\mc J}
=
IAJ
\,\,,\,\,\,\,
A_{\mc J\mc I}=J_1AI_1
\,.
\end{equation}
Let now $m=n$,
and consider the $(I_1,J_1)$-quasideterminant of $A$ (assuming it exists),
where $I_1$ and $J_1$ are as in \eqref{0304:eq4}.
We have, by the definition \eqref{eq:gen-quasidet} 
and the second formula in \eqref{0304:eq5},
\begin{equation}\label{0304:eq5b}
|A|_{I_1J_1}
=
((A^{-1})_{\mc J\mc I})^{-1}
\,\in\Mat_{n\times n}\mc U\,.
\end{equation}
The following result provides an alternative formula
for computing quasideterminant in this special case
\begin{proposition}\label{0304:prop}
Assume that the matrix $A\in\Mat_{N\times N}\mc U$ is invertible.
Let $\mc I,\mc J\subset\{1,\dots,N\}$ be subsets of the same cardinality
$|\mc I|=|\mc J|=n$,
and let $I_1\in\Mat_{N\times n}\mb F$ and $J_1\in\Mat_{n\times N}\mb F$
be as in \eqref{0304:eq4}.
Let also $\mc I^c,\mc J^c$ be the complements of $\mc I$ and $\mc J$ in $\{1,\dots,N\}$.
Then,
the quasideterminant $|A|_{I_1J_1}$ exists
if and only if 
the submatrix $A_{\mc I^c\mc J^c}\in\Mat_{(N-n)\times(N-n)}\mc U$
is invertible.
In this case, the quasideterminant $|A|_{I_1J_1}$ is given 
by the following formula
\begin{equation}\label{0304:eq6}
|A|_{I_1J_1}
=
A_{\mc I\mc J}
-
A_{\mc I\mc J^c}(A_{\mc I^c\mc J^c})^{-1}A_{\mc I^c\mc J}
\,.
\end{equation}
\end{proposition}
\begin{proof}
(The ``if'' part is proved in \cite[Prop.4.2]{DSKV16a},
but we reproduce it here.)
After reordering the rows and the columns of $A$, 
we can write the resulting matrix in block form as:
$$
\tilde A=\left(\begin{array}{ll} 
A_{\mc I\mc J} & A_{\mc I\mc J^c} \\
A_{\mc I^c\mc J} & A_{\mc I^c\mc J^c}
\end{array}\right)
\,.
$$
Since $A$ is invertible by assumption,
$\tilde A$ is invertible as well, and
the inverse matrix $\tilde A^{-1}$ has the block form
$$
\tilde A^{-1}=\left(\begin{array}{ll} 
(A^{-1})_{\mc J\mc I} & (A^{-1})_{\mc J\mc I^c} \\
(A^{-1})_{\mc J^c\mc I} & (A^{-1})_{\mc J^c\mc I^c}
\end{array}\right)
\,.
$$
The equation $\tilde A\tilde A^{-1}=\id$ turns into the following four equations:
\begin{equation}\label{0309:eq1}
\begin{array}{l}
\vphantom{\Big(}
\displaystyle{
A_{\mc I\mc J}(A^{-1})_{\mc J\mc I}+A_{\mc I\mc J^c}(A^{-1})_{\mc J^c\mc I}=\id_{|\mc I|}
\,,} \\
\vphantom{\Big(}
\displaystyle{
A_{\mc I\mc J}(A^{-1})_{\mc J\mc I^c}+A_{\mc I\mc J^c}(A^{-1})_{\mc J^c\mc I^c}=0
\,,} \\
\vphantom{\Big(}
\displaystyle{
A_{\mc I^c\mc J}(A^{-1})_{\mc J\mc I}+A_{\mc I^c\mc J^c}(A^{-1})_{\mc J^c\mc I}=0
\,,} \\
\vphantom{\Big(}
\displaystyle{
A_{\mc I^c\mc J}(A^{-1})_{\mc J\mc I^c}+A_{\mc I^c\mc J^c}(A^{-1})_{\mc J^c\mc I^c}=\id_{N-|\mc I|}
\,.}
\end{array}
\end{equation}
Suppose that $A_{\mc I^c\mc J^c}$ is invertible.
Then we can use the first and the third equations in \eqref{0309:eq1}
to solve for $(A^{-1})_{\mc J\mc I}$, as follows.
From the third equation we get
$$
(A^{-1})_{\mc J^c\mc I}
=-
(A_{\mc I^c\mc J^c})^{-1}A_{\mc I^c\mc J}(A^{-1})_{\mc J\mc I}
\,,
$$
and substituting in the first equation we get
$$
\big(A_{\mc I\mc J}
-A_{\mc I\mc J^c}(A_{\mc I^c\mc J^c})^{-1}A_{\mc I^c\mc J}\big)
(A^{-1})_{\mc J\mc I}
=\id_{|\mc I|}
\,.
$$
It follows that $(A^{-1})_{\mc J\mc I}$ is invertible
and its inverse coincides with the RHS of \eqref{0304:eq6}.
This proves the ``if'' part, in view of \eqref{0304:eq5b}.
Conversely, suppose that $|A|_{I_1J_1}$ exists,
i.e., by \eqref{0304:eq5b}, that $(A^{-1})_{\mc J\mc I}$ is invertible.
Then we can use the third and fourth equations in \eqref{0309:eq1}
to solve for $A_{\mc I^c\mc J^c}$.
From the third equation we get
$$
A_{\mc I^c\mc J}=-A_{\mc I^c\mc J^c}(A^{-1})_{\mc J^c\mc I}((A^{-1})_{\mc J\mc I})^{-1}
\,.
$$
Substituting in the fourth equation we get
$$
A_{\mc I^c\mc J^c}
\big(-
(A^{-1})_{\mc J^c\mc I}((A^{-1})_{\mc J\mc I})^{-1}(A^{-1})_{\mc J\mc I^c}
+(A^{-1})_{\mc J^c\mc I^c}\big)
=\id_{N-|\mc I|}
\,.
$$
In particular, $A_{\mc I^c\mc J^c}$ is invertible, as claimed.
\end{proof}

\subsection{
Some properties of operators of Yangian type
}\label{sec:2.3}

We want to prove that every (generalized) quasideterminant of an operator
of Yangian type is again of Yangian type.
For this,
we shall need the following lemma.
\begin{lemma}
Let $\mc U$ be a unital associative algebra and 
let $A\in\Mat_{N\times N}\mc U$ be invertible. 
For every $a\in\mc U$
and $i,j\in\{1,\dots,N\}$, we have
\begin{equation}\label{20160402:eq1}
[a,(A^{-1})_{ij}]
=-\sum_{h,k=1}^N(A^{-1})_{ih}[a,A_{hk}](A^{-1})_{kj}
\,.
\end{equation}
\end{lemma}
\begin{proof}
It follows from the fact that $\ad(a)$ is a derivation of $\mc U$.
\end{proof}

\begin{proposition}\label{prop:properties-adler}
Let $U$ be a unital associative algebra
and 
suppose that $A(z)\in\Mat_{M\times N}U((z^{-1}))$ is
an operator of Yangian type.
\begin{enumerate}[(a)]
\item
Let $Z\subset U$ be the center of $U$,
and let $J\in\Mat_{P\times M}Z$, $I\in\Mat_{N\times Q}Z$.
Then
$JA(z)I\in\Mat_{P\times Q}U((z^{-1}))$ is an operator of Yangian type.
\item
Assume that $M\!=\!N$ and that $A(z)$ is invertible
in $\Mat_{N\times N}U((z^{-1}))$.
Then the inverse matrix $A^{-1}(z)$ is an operator of Yangian type
with respect to the opposite product of the algebra $U$.
\end{enumerate}
\end{proposition}
\begin{proof}
Part (a) is straightforward. 
We prove (b).
First, we use \eqref{20160402:eq1} for the left factor, to get
\begin{equation}\label{20160402:eq3a}
\begin{array}{l}
\vphantom{\Big(}
\displaystyle{
(z-w)[(A^{-1})_{ij}(z),(A^{-1})_{hk}(w)]^\text{op}=-(z-w)[(A^{-1})_{ij}(z),(A^{-1})_{hk}(w)]
} \\
\vphantom{\Big(}
\displaystyle{
=+\sum_{a,b=1}^N
(z-w)(A^{-1})_{ia}(z)[A_{ab}(z),(A^{-1})_{hk}(w)](A^{-1})_{bj}(z)
\,.}
\end{array}
\end{equation}
Next, we use again \eqref{20160402:eq1} for the right factor,
to rewrite the RHS of \eqref{20160402:eq3a} as
\begin{equation}\label{20160402:eq3b}
-\sum_{a,b,c,d=1}^N
(z-w)(A^{-1})_{ia}(z)(A^{-1})_{hc}(w)[A_{ab}(z),A_{cd}(w)](A^{-1})_{dk}(w)(A^{-1})_{bj}(z)
\,.
\end{equation}
Applying the Yangian identity \eqref{eq:adler}, formula \eqref{20160402:eq3b} becomes
\begin{equation*}
\begin{array}{l}
\displaystyle{
-\!\!\!\!\!
\sum_{a,b,c,d=1}^N
\!\!\!\!\!
(A^{-1})_{ia}\!(z)(A^{-1})_{hc}\!(w)
\big(
A_{cb}(w)A_{ad}(z)\!-\!A_{cb}(z)A_{ad}(w)
\big)
(A^{-1})_{dk}\!(w)(A^{-1})_{bj}\!(z)
} \\
\displaystyle{
-\sum_{b,d=1}^N
\delta_{id}\delta_{hb}(A^{-1})_{dk}(w)(A^{-1})_{bj}(z)
+\sum_{a,c=1}^N
(A^{-1})_{ia}(z)(A^{-1})_{hc}(w)\delta_{a,k}\delta_{c,j}
} \\
\vphantom{\Big(}
\displaystyle{
=-(A^{-1})_{ik}(w)(A^{-1})_{hj}(z)
+(A^{-1})_{ik}(z)(A^{-1})_{hj}(w)
} \\
\vphantom{\Big(}
\displaystyle{
=
(A^{-1})_{hj}(w)\cdot^{\text{op}}(A^{-1})_{ik}(z)
-(A^{-1})_{hj}(z)\cdot^{\text{op}}(A^{-1})_{ik}(w)
\,,}
\end{array}
\end{equation*}
where $a\cdot^{\text{op}}b=ba$ denotes the opposite product of $U$.
This concludes the proof.
\end{proof}
\begin{theorem}\label{thm:quasidet-adler}
Let $U$ be a unital associative algebra
and let $Z\subset U$ be its center.
Let $A(z)\in\Mat_{N\times N}U((z^{-1}))$ be an operator of Yangian type.
Then, for every $I\in\Mat_{N\times M}Z$
and $J\in\Mat_{M\times N}Z$ with $M\leq N$,
the generalized quasideterminant $|A(z)|_{IJ}$,
provided that it exists,
is an operator of Yangian type.
\end{theorem}
\begin{proof}
It is an obvious consequence of Proposition \ref{prop:properties-adler}.
\end{proof}

In Section \ref{sec:4.6} we will need the following lemma.
\begin{lemma}\label{0303:lem5}
Let $U$ be a unital associative algebra.
Let $A(z)\in\Mat_{N\times N}U((z^{-1}))$ be
an operator of Yangian type
and assume that $A(z)$ is invertible
in the algebra $\Mat_{N\times N}U((z^{-1}))$.
Then, the following formula holds
\begin{equation}\label{0303:eq17}
\begin{array}{l}
\vphantom{\Big(}
\displaystyle{
(z-w)\big[A_{ij}(z), (A^{-1})_{hk}(w)\big]
} \\
\vphantom{\Big(}
\displaystyle{
= 
-\delta_{hj}\sum_tA_{it}(z)(A^{-1})_{tk}(w)
+\delta_{ik}\sum_t(A^{-1})_{ht}(w)A_{tj}(z)
\,.}
\end{array}
\end{equation}
\end{lemma}
\begin{proof}
By formula \eqref{20160402:eq1}, we have
\begin{equation}\label{0303:eq18}
\begin{array}{l}
\vphantom{\Big(}
\displaystyle{
(z-w)\big[A_{ij}(z), (A^{-1})_{hk}(w)\big]
} \\
\vphantom{\Big(}
\displaystyle{
=
-(z-w)\sum_{s,t=1}^N(A^{-1})_{hs}(w)\big[A_{ij}(z),A_{st}(w)\big](A^{-1})_{tk}(w)
\,.}
\end{array}
\end{equation}
We then use the Yangian identity \eqref{eq:adler}, to rewrite the RHS of \eqref{0303:eq18}
as
\begin{equation}\label{0303:eq19}
\begin{array}{l}
\vphantom{\Big(}
\displaystyle{
-\sum_{s,t=1}^N(A^{-1})_{hs}(w)
\big(
A_{sj}(w)A_{it}(z)- A_{sj}(z)A_{it}(w)
\big)
(A^{-1})_{tk}(w)
} \\
\vphantom{\Big(}
\displaystyle{
=-\delta_{hj}
\sum_{t=1}^N
A_{it}(z)(A^{-1})_{tk}(w)
+\delta_{ik}
\sum_{s=1}^N
(A^{-1})_{hs}(w)
A_{sj}(z)
\,.}
\end{array}
\end{equation}
\end{proof}

\section{Finite \texorpdfstring{$W$}{W}-algebras
}\label{sec:3}

\subsection{
Definition of the $W$-algebra
}\label{sec:3.1}

Let $\mf g$ be a reductive Lie algebra with a non-degenerate symmetric 
invariant bilinear form $(\cdot\,|\,\cdot)$.
Let $f\in\mf g$ be a nilpotent element which,
by the Jacobson-Morozov Theorem, can be included in an $\mf{sl}_2$-triple
$\{f,2x,e\}\subset\mf g$.
We have the corresponding $\ad x$-eigenspace decomposition
\begin{equation}\label{eq:grading}
\mf g=\bigoplus_{k\in\frac{1}{2}\mb Z}\mf g_{k}
\,\,\text{ where }\,\,
\mf g_k=\big\{a\in\mf g\,\big|\,[x,a]=ka\big\}
\,,
\end{equation}
so that $f\in\mf g_{-1}$, $x\in\mf g_{0}$ and $e\in\mf g_{1}$.
We shall denote, for $j\in\frac12\mb Z$, $\mf g_{\geq j}=\oplus_{k\geq j}\mf g_k$,
and similarly $\mf g_{\leq j}$.

A key role in the theory of $W$-algebras is played by the left ideal
\begin{equation}\label{0225:eq4}
I=
U(\mf g)\big\langle m-(f|m) \big\rangle_{m\in\mf g_{\geq1}}
\subset U(\mf g)
\,,
\end{equation}
and the corresponding left $\mf g$-module 
\begin{equation}\label{eq:M}
M=U(\mf g)/I\,.
\end{equation}
We shall denote by $\bar 1\in M$ the image of $1\in U(\mf g)$
in the quotient space.
Note that, by definition, $g\bar1=0$ if and only if $g\in I$.
\begin{lemma}\label{0303:lem4}
\begin{enumerate}[(a)]
\item
$U(\mf g)IU(\mf g_{\geq\frac12})\subset I$.
\item
$\mf g$ acts on the module $M$ by left multiplication,
and $\mf g_{\geq\frac12}$ acts on $M$ both by left and by right multiplication
(hence, also via adjoint action).
\end{enumerate}
\end{lemma}
\begin{proof}
Clearly, $U(\mf g)I\subset I$, so it suffices to prove, for (a), 
that $IU(\mf g_{\geq\frac12})\subset I$,
namely that $I$ is a right module over $\mf g_{\geq\frac12}$.
Let $h=g(m-(f|m))\in I$, 
with $g\in U(\mf g)$ and $m\in\mf g_{\geq1}$,
and let $a\in\mf g_{\geq\frac12}$.
We have
$$
ha=g(m-(f|m))a
=ga(m-(f|m))+g[m,a]
\,.
$$
The first summand in the RHS obviously lies in $I$.
The second summand also lies in $I$ 
since $[m,a]\in\mf g_{\geq\frac32}$, which implies $(f|[m,a])=0$.
Part (b) is a restatement of (a) in terms of the module $M$.
\end{proof}
Consider the subspace
\begin{equation}\label{0225:eq3}
\widetilde{W}
:=
\big\{
w\in U(\mf g)\,\big|\,
[a,w]\bar 1=0\,\text{ in }\, M\,,\,\,\text{for all } a\in\mf g_{\geq\frac12}
\big\}
\,\subset U(\mf g)\,.
\end{equation}
\begin{lemma}\phantomsection\label{0225:prop1a}
\begin{enumerate}[(a)]
\item
$I\subset\widetilde{W}$.
\item
For $h\in I$ and $w\in\widetilde{W}$, we have $hw\in I$.
\item
$\widetilde{W}$ 
is a subalgebra of $U(\mf g)$.
\item
$I$ is a (proper) two-sided ideal of $\widetilde{W}$.
\end{enumerate}
\end{lemma}
\begin{proof}
By Lemma \ref{0303:lem4},
$I$ is preserved by both the left and right multiplication by elements of $\mf g_{\geq\frac12}$,
hence by their adjoint action. It follows that $[a,I]\bar1=0$, proving (a).
For (b), let $h=A(m-(f|m))\in I$, as above, and $w\in\widetilde{W}$. 
We have
\begin{equation}\label{0303:eq2}
hw=A(m-(f|m))w=
Aw(m-(f|m))+A[m,w]\,.
\end{equation}
The first summand in the RHS of \eqref{0303:eq2} obviously lies in $I$.
The second summand also lies in $I$ by definition of $\widetilde{W}$
(since $m\in\mf g_{\geq1}$).
This proves part (b).
If $w_1,w_2\in\widetilde{W}$ and $a\in\mf g_{\geq\frac12}$,
then
\begin{equation}\label{0225:eq7}
[a,w_1w_2]\bar 1=[a,w_1]w_2\bar 1+w_1[a,w_2]\bar 1
\,.
\end{equation}
Since, by assumption, $w_2\in\widetilde{W}$, we have $[a,w_2]\bar1=0$.
On the other hand, we also have $[a,w_1]\in I$,
so, by (b), $[a,w_1]w_2\in I$, and therefore $[a,w_1]w_2\bar 1=0$.
This proves claim (c).
Recall that $I$ is a left ideal of $U(\mf g)$, hence of $\widetilde{W}$,
and by (b) it is also a right ideal of $\widetilde{W}$,
proving (d).
%
\end{proof}
\begin{proposition}\label{0225:prop1}
The quotient
\begin{equation}\label{0225:eq8}
W(\mf g,f)
=
M^{\ad\mf g_{\geq\frac12}}
=
\widetilde{W}/I
\end{equation}
has a natural structure of a unital associative algebra,
induced by that of $U(\mf g)$.
\end{proposition}
\begin{proof}
It follows immediately from Lemma \ref{0225:prop1a}(c) and (d)
\end{proof}
\begin{definition}\label{def:Walg}
The \emph{finite} $W$-\emph{algebra}
associated to the Lie algebra $\mf g$ and its nilpotent element $f$
is the algebra $W(\mf g,f)$ defined in \eqref{0225:eq8}.
\end{definition}
\begin{remark}\label{rem:good-grading}
By a result of Kostant
all $\mf{sl}_2$-triples containing $f$ are conjugate in $\mf g$.
It follows that 
the $W$-algebra depends only on the nilpotent element $f$ (in fact only of its adjoint orbit).
In fact, one can define the $W$-algebra by replacing the Dynkin grading \eqref{eq:grading}
by an arbitrary good grading \cite{EK05}.
All the resulting $W$-algebras turn out to be isomorphic \cite{BG07}.
\end{remark}

\subsection{
The map $\rho$
}\label{sec:3.2}

Consider the Lie subalgebras $\mf g_{\leq0},\,\mf g_{\geq\frac12}\,\subset\mf g$
and the corresponding universal enveloping algebras 
$U(\mf g_{\leq0}),\,U(\mf g_{\geq\frac12})\,\subset U(\mf g)$.
By the PBW Theorem we have a canonical linear map (not an algebra isomorphism)
\begin{equation}\label{0224:eq2}
U(\mf g)
\,\stackrel{\sim}{\longrightarrow}\,
U(\mf g_{\leq 0})\otimes U(\mf g_{\geq\frac12})
\,.
\end{equation}
The ideal $I$ in \eqref{0225:eq4} is mapped, via the bijection \eqref{0224:eq2} to
$I\stackrel{\sim}{\rightarrow}U(\mf g_{\leq 0})\otimes I_{\geq\frac12}$,
where
$$
I_{\geq\frac12}=
U(\mf g_{\geq\frac12})\big\langle m-(f|m) \big\rangle_{m\in\mf g_{\geq1}}
\subset U(\mf g_{\geq\frac12})
\,.
$$
It is immediate to check, 
with the same computations as those used for the proof of Lemma \ref{0303:lem4},
that $I_{\geq\frac12}\subset U(\mf g_{\geq\frac12})$ is a two-sided ideal,
and the quotient
\begin{equation}\label{0224:eq3}
F(\mf g_{\frac12})
=
U(\mf g_{\geq\frac12})
\big/ I_{\geq\frac12}
\end{equation}
is isomorphic to the Weyl algebra of $\mf g_{\frac12}$
with respect to the non-degenerate skewsymmetric bilinear form
\begin{equation}\label{neutral}
\langle a|b\rangle\,=\,(f|[a,b])
\,\,,\,\,\,\,
a,b\in\mf g_{\frac12}
\,.
\end{equation}
Hence, the bijection \eqref{0224:eq2} induces the bijection
\begin{equation}\label{0224:eq4}
M=U(\mf g)/I\,\stackrel{\sim}{\longrightarrow}U:=U(\mf g_{\leq 0})\otimes F(\mf g_{\frac12})
\,\,,\,\,\,\,
g\bar1\mapsto\rho(g)\,.
\end{equation}
This defines the surjective map $\rho:\,U(\mf g)\twoheadrightarrow U$,
whose kernel is $\ker\rho=I$.

Note that the map $\rho$ is a linear transformation
from the associative algebra $U(\mf g)$ to the associative algebra
$U:=U(\mf g_{\leq0})\otimes F(\mf g_{\frac12})$,
but it is NOT an algebra homomorphism.
We can write an explicit formula for its action:
for $A\in U(\mf g_{\leq0})$, $q\in U(\mf g_{\geq\frac12})$, and $u\in U(\mf g_{\geq1})$,
we have
\begin{equation}\label{0225:eq2}
\rho(Aqu)
=
\chi(u)A\otimes\pi_{\frac12}(q)
\in U=U(\mf g_{\leq0})\otimes F(\mf g_{\frac12})
\,,
\end{equation}
where $\chi:\,U(\mf g_{\geq1})\twoheadrightarrow \mb F$
if the character defined by $\chi(m)=(f|m)$ for $m\in\mf g_{\geq1}$,
and $\pi_{\frac12}:\,U(\mf g_{\geq\frac12})\twoheadrightarrow F(\mf g_{\frac12})$
is the natural quotient map.

\subsection{
The $W$-algebra as a subalgebra of $U^\circ$.
}\label{sec:3.3}

By definition, the $W$-algebra $W(\mf g,f)$ is a subspace
of the module $M=U(\mf g)/I$,
which can be identified, via \eqref{0224:eq4}, with the algebra $U$.
It is natural to ask whether $W(\mf g,f)$ can be identified, via \eqref{0224:eq4}, 
with a subalgebra of $U$.
We shall prove, in Proposition \ref{thm:rho} below,
that this is the case, 
provided that we take the opposite product in the Weyl algebra $F(\mf g_{\frac12})$.

\begin{lemma}\phantomsection\label{0302:lem1}
For $w\in\widetilde{W}$, $g\in U(\mf g)$, $q\in U(\mf g_{\geq\frac12})$, 
and $u\in U(\mf g_{\geq1}\!)$, we have 
$\rho(gquw)=\chi(u)\rho(gw)(1\otimes \pi_{\frac12}(q))$.
\end{lemma}
\begin{proof}
By definition of $\widetilde{W}$, we have $aw-wa\in I$ for every $a\in\mf g_{\geq\frac12}$,
and, by Lemma \ref{0303:lem4}, 
$I$ is preserved by the right multiplication by elements of $\mf g_{\geq\frac12}$.
Hence, $\rho(gquw)=\rho(gwqu)$.
The claim then follows by equation \eqref{0225:eq2}.
\end{proof}
\begin{proposition}\label{thm:rho}
The linear map $\rho$ restricts to an algebra homomorphism
\begin{equation}\label{0224:eq5}
\rho:\,
\widetilde{W}
\,\longrightarrow\,
U^\circ:=U(\mf g_{\leq 0})\otimes F^{\text{op}}(\mf g_{\frac12})
\,,
\end{equation}
where $F^{\text{op}}(\mf g_{\frac12})$ denotes the Weyl algebra $F(\mf g_{\frac12})$
with the opposite product.
Hence, 
the bijective linear map \eqref{0224:eq4}
restricts to an algebra embedding
$W(\mf g,f)\hookrightarrow U^\circ$.
\end{proposition}
\begin{proof}
Let $w_1,w_2\in\widetilde{W}$.
By the PBW Theorem, we can expand them as
\begin{equation}\label{0225:eq15}
w_1
=
\sum
Ap
\,\,,\,\,\,\,
w_2
=
\sum
Bq
\,,
\end{equation}
where $A,B\in U(\mf g_{\leq0})$, and $p,q\in U(\mf g_{\geq\frac12})$.
By \eqref{0225:eq2}, we have
$$
\rho(w_1)
=
\sum
A\otimes\pi_{\frac12}(p)
\,\,,\,\,\,\,
\rho(w_2)
=
\sum
B\otimes\pi_{\frac12}(q)
\,.
$$
Hence, their product 
in the algebra $U=U(\mf g_{\leq0})\otimes F^{\text{op}}(\mf g_{\frac12})$
(which we denote by $\circ$) is
\begin{equation}\label{0225:eq9}
\begin{array}{c}
\vphantom{\Big(}
\displaystyle{
\rho(w_1)\circ\rho(w_2)
=
\sum
AB\otimes\pi_{\frac12}(qp)
\,.}
\end{array}
\end{equation}
Expanding $w_1$ as in \eqref{0225:eq15}
and applying Lemma \ref{0302:lem1}, we get
$$
\rho(w_1w_2)
=
\sum
\rho(Apw_2)
=
\sum
\rho(Aw_2)(1\otimes\pi_{\frac12}(p))
\,.
$$
Next, we expand $w_2$ as in \eqref{0225:eq15} and we
use formula \eqref{0225:eq2}, to get
\begin{equation}\label{0225:eq17}
\rho(w_1w_2)
=
\sum
\rho(ABq)(1\otimes\pi_{\frac12}(p))
=
\sum
AB\otimes\pi_{\frac12}(q)\pi_{\frac12}(p)
=
\sum
AB\otimes\pi_{\frac12}(qp)
\,.
\end{equation}
Comparing \eqref{0225:eq9} and \eqref{0225:eq17},
we get $\rho(w_1w_2)=\rho(w_1)\circ\rho(w_2)$, as claimed.
\end{proof}
\begin{remark}
There are analogues of Proposition \ref{thm:rho}
also for the classical and quantum affine $W$-algebras.
In the classical affine case,
the Poisson vertex algebra $\mc W(\mf g,f)$ is, 
by construction, contained in $\mc V(\mf g_{\leq\frac12})$,
the algebra of differential polynomials over $\mf g_{\leq\frac12}$ (cf. \cite[Eq.(3.10)]{DSKV13})
(while in the quantum case $U(\mf g_{\leq\frac12})$ does not make sense,
since $\mf g_{\leq\frac12}$ is not a Lie algebra).
Moreover, the associative algebra embedding 
$W(\mf g,f)\hookrightarrow U^\circ$ provided by Proposition \ref{thm:rho}
corresponds, in the classical affine case, to the Poisson vertex algebra
embedding $\mc W(\mf g,f)\hookrightarrow\mc V(\mf g_{\leq0})\otimes\mc F(\mf g_{\frac12})$
described in \cite[Sec.8]{DSKVjems}.
In the quantum affine case,
the analogue of Proposition \ref{thm:rho} is contained in \cite{DSK06}.
Indeed,
by Theorem 5.9 of that paper, the cohomology 
$W_k(\mf g,f)=H(V(R),d)$ is concentrated at charge zero,
hence by Theorem 5.7 it is a subalgebra
of the tensor product of the universal affine vertex algebra
over $J_{\mf g_{\leq0}}$
(which is the same as the affine vertex algebra over $\mf g_{\leq0}$,
but with a shifted level)
and the vertex algebra of free fermions over $\mf g_{\frac12}$
with the skewsymmetric bilinear form \eqref{neutral}.
\end{remark}

\begin{remark}\label{rem:premet}
Proposition \ref{thm:rho} is due to Premet, \cite[Prop.2.2]{Pre07}.
We included our proof for completeness.
\end{remark}

\section{Operator of Yangian type for the \texorpdfstring{$W$}{W}-algebra \texorpdfstring{$W(\mf{gl}_N,f)$}{W(gl\_N,f)}
}\label{sec:4}

\subsection{Setup and notation}\label{sec:4.1}

We review here the notation introduced in \cite[Sec.4]{DSKV16b}.

A nilpotent orbit in $\mf{gl}_N$ is associated 
to a partition, $p=(p_1,\dots,p_r)$ of $N$, where $p_1\geq\dots\geq p_r>0$,
$\sum_ip_i=N$.
We let $r_1$ be 
the multiplicity of $p_1$ in the partition $p$.

Let $V$ be the $N$-dimensional vector space over $\mb F$ with basis 
$\{e_{ih}\}_{(i,h)\in\mc T}$,
where $\mc T$ is the following index set (of cardinality $N$)
\begin{equation}\label{eq:J}
\mc T=\big\{(i,h)\in\mb Z_{\geq0}^2\,\big|\,1\leq i\leq r,\,1\leq h\leq p_i\big\}
\,.
\end{equation}
The Lie algebra $\mf g=\mf{gl}(V)\simeq\mf{gl}_N$ has a basis 
consisting of the elementary matrices $e_{(ih),(jk)}$, $(ih),(jk)\in\mc T$
(as before, we denote by $E_{(ih),(jk)}$ the elementary matrix
and by $e_{(ih),(jk)}$ the same matrix, when viewed as elements of the Lie algebra $\mf g$).
The nilpotent element $f\in\mf g$ associated to the partition $p$
is the ``shift'' operator
\begin{equation}\label{eq:f}
f=\sum_{(ih)\in\mc T\,|\,h<p_i}e_{(i,h+1),(ih)}\,\in\mf g
\,.
\end{equation}
(As before, in order to keep distinct elements from $\mf{gl}_N$ from
matrices in $\Mat_{N\times N}\mb F$,
we shall also denote, 
$F=\sum_{(ih)\in\mc T\,|\,h<p_i}E_{(i,h+1),(ih)}\,\in\Mat_{N\times N}\mb F$.)
If we order the indices $(ih)$ lexicographically,
$f$ has Jordan form with nilpotent Jordan blocks of sizes $p_1,\dots,p_r$.

The Dynkin grading \eqref{eq:grading} for the nilpotent element $f$
is defined by the adjoint action of the diagonal matrix 
$x=\sum_{(ih)\in\mc T}
\frac12(p_i+1-2h)
e_{(ih),(ih)}$:
\begin{equation}\label{eq:adx}
(\ad x)e_{(ih),(jk)}=\Big(\frac12(p_i-p_j)-(h-k)\Big)e_{(ih),(jk)}
\,.
\end{equation}
The depth of this grading is $d=p_1-1$,
and the $\ad x$-eigenspace of degree $d$ is
\begin{equation}\label{eq:gd}
\mf g_d
=
\Span_{\mb F}
\big\{
E_{(i1),(jp_1)}\,\big|\,
i,j=1,\dots,r_1
\big\}
\,.
\end{equation}

In $\mf g_d$ there is a ``canonical'' element of maximal rank $r_1$,
$S_1=\sum_{i=1}^{r_1}E_{(i1),(ip_1)}$.
By \cite[Prop.4.1]{DSKV16b}, 
the matrix $S_1$ factors as $S_1=I_1J_1$, where
\begin{equation}\label{eq:factor1}
I_1=\sum_{i=1}^{r_1}
E_{(i1),i}
\,\in\Mat_{N\times r_1}\mb F
\,\,,\,\,\,\,
J_1=\sum_{i=1}^{r_1}
E_{i,(ip_1)}
\,\in\Mat_{r_1\times N}\mb F
\,.
\end{equation}

\subsection{The ``shift'' matrix $D$}
\label{sec:4.2}

The following diagonal matrix will play an important role
in the next sections:
\begin{equation}\label{0226:eq5}
D=\sum_{(ih)\in\mc T}d_{(ih)}E_{(ih),(ih)}
\,,
\end{equation}
where
\begin{equation}\label{0226:eq6}
d_{(ih)}=-\dim(\mf g_{\geq1}e_{(ih)})
\,.
\end{equation}
Note that, by \eqref{eq:adx}, the elementary matrix $e_{(jk),(ih)}$ lies in $\mf g_{\geq1}$
precisely when $\frac12(p_j-p_i)-(k-h)\geq1$.
Hence, the numbers \eqref{0226:eq6} are given by:
$$
d_{(ih)}
=-\#\Big\{
(jk)\in\mc T
\,\Big|\,
\frac12(p_j-p_i)-(k-h)\geq1
\Big\}
\,.
$$
These numbers have the following simple combinatorial meaning.
Consider the symmetric (with respect to the $y$-axis) pyramid,
with boxes indexed by $(ih)\in\mc T$,
with $i$ and $h$ being respectively the row index (starting from the bottom)
and the column index (starting from the right) (cf. \cite[Fig.1]{DSKV16b}).
The $x$-coordinate of the center of the box labeled $(ih)$ is 
\begin{equation}\label{eq:xih}
x{(ih)}=\frac12(p_i+1-2h)
\end{equation}
(which is the $x$-eigenvalue on $e_{(ih)}\in V$).
Hence, the difference between the $x$-coordinates of the boxes $(jk)$ and $(ih)$ is
$$
x{(jk)}-x{(ih)}=\frac12(p_j-p_i)-(k-h)
\,.
$$
Therefore, the number $-d_{(ih)}$
is the number of boxes in the pyramid which are entirely to the right of the box labeled $(ih)$.

For example, for the partition $(3,2,1)$ of $6$, we have the pyramid in Figure \ref{fig:pyramid},

\begin{figure}[h]
\setlength{\unitlength}{0.14in}
\centering
\begin{picture}(30,9)

\put(13,7){\framebox(2,2){(31)}}

\put(12,5){\framebox(2,2){(22)}}
\put(14,5){\framebox(2,2){(21)}}

\put(11,3){\framebox(2,2){(13)}}
\put(13,3){\framebox(2,2){(12)}}
\put(15,3){\framebox(2,2){(11)}}

\put(10,2){\vector(1,0){8}}
\put(18,1){$x$}

\put(12,1.6){\line(0,1){0.8}}
\put(13,1.8){\line(0,1){0.4}}
\put(14,1.6){\line(0,1){0.8}}
\put(15,1.8){\line(0,1){0.4}}
\put(16,1.6){\line(0,1){0.8}}

\put(13.8,0.6){0}
\put(15.8,0.6){1}
\put(11.6,0.6){-1}

\put(14.8,-0.3){\tiny{$\frac12$}}
\put(12.6,-0.3){\tiny{$-\frac12$}}

\end{picture}
\caption{} 
\label{fig:pyramid}
\end{figure}

\noindent
and, if we order the index set $\mc T$ lexicographically, 
the diagonal matrix $D$ is the following
$$
D=
\diag(0,-1,-4,0,-2,-1)
\,.
$$

\subsection{The matrix $L(z)$}
\label{sec:4.3}

Recall the matrix $A(z)=z\id_N+E$ defined by \eqref{eq:A},
which is an operator of Yangian type for $U(\mf{gl}_N)$.
We want to find an analogous operator for the $W$-algebra $W=W(\mf g,f)$,
for $\mf g=\mf{gl}_N$ and its nilpotent element $f$.
More precisely, we will construct a matrix 
$L(z)\in\Mat_{r_1\times r_1}W((z^{-1}))$,
where $r_1$ is the multiplicity of $p_1$ in the partition $p$,
which is an operator of Yangian type 
for the algebra $W(\mf g,f)$.

The matrix $L(z)$ is constructed as follows.
%
%
Consider the matrix
\begin{equation}\label{eq:rhoA}
z\id_N+F+\pi_{\leq\frac12}E+D
\,\in\Mat_{N\times N}U(\mf g)[z]
\,,
\end{equation}
where 
$F$ is as in \eqref{eq:f} (rather, as in the parenthetical observation after \eqref{eq:f}),
$D$ is the ``shift'' matrix \eqref{0226:eq5},
$E$ is the matrix \eqref{eq:E},
and 
$$
\pi_{\leq\frac12}E
=
\sum_{\substack{ (ih),(jk)\in\mc T \\ (x{(jk)}-x{(ih)}\leq\frac12) }}
e_{(jk)(ih)}E_{(ih)(jk)}
\in\Mat_{N\times N}U(\mf g)
\,.
$$
We shall prove that,
if we view \eqref{eq:rhoA}
as a matrix over $U(\mf g)((z^{-1}))$,
then its $(I_1,J_1)$ quasideterminant exists,
where $I_1$ and $J_1$ are the matrices \eqref{eq:factor1}.
The matrix $L(z)$ is defined as the image in the quotient module $M=U(\mf g)/I$
of this quasideterminant:
\begin{equation}\label{eq:La}
L(z)
=
\widetilde{L}(z)\bar 1
\,\in\Mat_{r_1\times r_1} M((z^{-1}))
\,\,\text{ where }\,\,
\widetilde{L}(z)
=
|z\id_N+F+\pi_{\leq\frac12}E+D|_{I_1J_1}
\,.
\end{equation}

The main result of the paper is that
the entries of the coefficients of $L(z)$ actually lie in the $W$-algebra 
$W(\mf g,f)\subset M$,
and the matrix $L(z)$ satisfies the Yangian identity \eqref{eq:adler}
(with respect to the associative product of the $W$-algebra $W(\mf g,f)$).
This is stated in Theorems \ref{thm:main1} and \ref{thm:main2} below.
Before stating them, we prove, 
in Section \ref{sec:4.4},
that the generalized quasideterminant defining $\widetilde{L}(z)$ exists.

\subsection{$L(z)$ exists}
\label{sec:4.4}

\begin{proposition}\label{thm:L1}
\begin{enumerate}[(a)]
\item
The matrix $z\id_N+F+\pi_{\leq\frac12}E+D$ is invertible in the algebra 
$\Mat_{N\times N}U(\mf g)((z^{-1}))$.
\item
The matrix $J_1(z\id_N+F+\pi_{\leq\frac12}E+D)^{-1}I_1$
is a Laurent series in $z^{-1}$
of degree (=the largest power of $z$ with non-zero coefficient) 
equal to $-p_1$,
and leading coefficient $(-1)^{p_1-1}\id_{r_1}$.
In particular, it is invertible in $\Mat_{r_1\times r_1}U(\mf g)((z^{-1}))$.
\item
Consequently,
the quasideterminant defining $\widetilde{L}(z)$ (cf. \eqref{eq:La}) exists
in the algebra $\Mat_{r_1\times r_1}U(\mf g)((z^{-1}))$.
\end{enumerate}
\end{proposition}
\begin{proof}
The proof is similar to the proof of \cite[Thm.4.2]{DSKV16b}.
We review here the argument for completeness.
The matrix $z\id_N+F+\pi_{\leq\frac12}E+D$
is of order one with leading coefficient $\id_N$.
Hence it is invertible in the algebra $\Mat_{N\times N}U(\mf g)((z^{-1}))$,
and its inverse can be computed by geometric series expansion:
\begin{equation}\label{eq:thm1-pr1}
(z\id_N+F+\pi_{\leq\frac12}E+D)^{-1}
=
\sum_{\ell=0}^\infty (-1)^\ell z^{-\ell-1}
(F+\pi_{\leq\frac12}E+D)^\ell
\,.
\end{equation}
This proves part (a). 
Next, we prove part (b).
Recalling the definition \eqref{eq:factor1} of the matrices $I_1$ and $J_1$,
we have, by \eqref{eq:thm1-pr1}, for $i,j\in\{1,\dots,r_1\}$,
\begin{equation}\label{eq:thm1-pr2}
\begin{array}{l}
\vphantom{\Big(}
\displaystyle{
\big(J_1(z\id_N+F+\pi_{\leq\frac12}E+D)^{-1}I_1\big)_{ij}
=
\sum_{\ell=0}^\infty (-1)^\ell z^{-\ell-1}
\!\!\!\!\!\!\!\!\!\!\!\!
\sum
_{
\substack{
(i_0h_0),(i_1h_1),\dots,(i_\ell h_\ell)\in\mc T \\
(i_0h_0)=(ip_1),\,(i_\ell h_\ell)=(j1)
}} 
} \\
\vphantom{\Big(}
\displaystyle{
\big(\pi_{\leq\frac12}e_{(i_1h_1),(i_0h_0)}+\delta_{i_0i_1}\delta_{h_0,h_1+1}
+\delta_{i_0i_1}\delta_{h_0h_1}d_{i_1h_1}\big)
} \\
\vphantom{\Big(}
\displaystyle{
\big(\pi_{\leq\frac12}e_{(i_2h_2),(i_1h_1)}+\delta_{i_1i_2}\delta_{h_1,h_2+1}
+\delta_{i_1i_2}\delta_{h_1h_2}d_{i_2h_2}\big)
} \\
\vphantom{\Big(}
\displaystyle{
\dots
\big(\pi_{\leq\frac12}e_{(i_\ell h_\ell),(i_{\ell-1}h_{\ell-1})}+\delta_{i_{\ell-1}i_\ell}\delta_{h_{\ell-1},h_\ell+1}
+\delta_{i_{\ell-1}i_\ell}\delta_{h_{\ell-1}h_\ell}d_{i_{\ell}h_{\ell}}\big)
\,.}
\end{array}
\end{equation}
Let $x_\alpha:=x{(i_\alpha h_\alpha)}=\frac12(p_{i_\alpha}+1-2h_\alpha)\in\frac12\mb Z$, $\alpha=0,1,\dots,\ell$.
In particular,
\begin{equation}\label{eq:thm1-pr4}
x_0=-\frac12(p_1-1)=-\frac12 d 
\,\,,\,\,\,\, 
x_\ell=\frac12(p_1-1)=\frac12 d
\,.
\end{equation}
Note that the summand in the RHS of \eqref{eq:thm1-pr2} vanishes unless
\begin{equation}\label{eq:thm1-pr3}
x_1-x_0\leq1
\,\,,\,\,\,\,
x_2-x_1\leq1
\,,\,\,\dots\,\,,\,
x_\ell-x_{\ell-1}\leq1 
\,.
\end{equation}
Moreover, 
\begin{equation}\label{eq:thm1-pr5}
\text{ if } 
x_\alpha=x_{\alpha-1}+1
\,\,,\,\,\,\,\text{ then }\,
i_\alpha=i_{\alpha-1}
\,\text{ and }\,
h_\alpha=h_{\alpha-1}-1
\,.
\end{equation}
Clearly, from \eqref{eq:thm1-pr4} and \eqref{eq:thm1-pr3}
we get that necessarily $\ell\geq p_1-1$.
Moreover, by \eqref{eq:thm1-pr5}, if $\ell=p_1-1$ 
then necessarily 
$$
i_0=i_1=\dots=i_{p_1-1}
\,\text{ and }\,
h_0=p_1,\, h_1=p_1-1,\, h_{p_1-1}=1
\,,
$$
and, in this case,
$$
\pi_{\leq\frac12}e_{(i_\alpha h_\alpha),(i_{\alpha-1}h_{\alpha-1})}
+\delta_{i_{\alpha-1}i_\alpha}\delta_{h_{\alpha-1},h_\alpha+1}
+\delta_{i_{\alpha-1}i_\alpha}\delta_{h_{\alpha-1}h_\alpha}d_{i_\alpha h_\alpha}
=1
\,,
$$
for all $\alpha=1,\dots,\ell$.
It follows that the Laurent series $\big(J_1(z\id_N+F+\pi_{\leq\frac12}E+D)^{-1}I_1\big)_{ij}$
has degree $\leq-p_1$,
and the coefficient of $z^{-p_1}$ is $(-1)^{p_1-1}\delta_{ij}$.
Claim (b) follows.
Part (c) is an obvious consequence of (a) and (b).
\end{proof}

\subsection{The Main Theorems}
\label{sec:main-thm}

We state here our main results:

\begin{theorem}\label{thm:main1}
The matrix $L(z)$ defined in \eqref{eq:La} lies in $\Mat_{r_1\times r_1}W(\mf g,f)((z^{-1}))$.
\end{theorem}

\begin{theorem}\label{thm:main2}
$L(z)$ is an operator of Yangian type for the algebra $W(\mf g,f)$.
\end{theorem}
%

We shall prove Theorems \ref{thm:main1} and \ref{thm:main2} in Sections \ref{sec:4.6}
and \ref{sec:4.7} respectively.
Both proofs will rely on the Main Lemma \ref{lem:main},
which will be stated and proved in Section \ref{sec:main-lemma}.
In order to state (and prove) Lemma \ref{lem:main}, though,
we shall need to introduce the 
Kazhdan filtration of $U(\mf g)$ \cite{Kos78},
the corresponding (completed) Rees algebra $\mc R U(\mf g)$,
and we shall need to extend the action of $U(\mf g)$
on the module $M=U(\mf g)/I$
to an action of the algebra $\mc RU(\mf g)$
(and its extension $\mc R_\infty U(\mf g)$)
on the corresponding (completed) Rees module $\mc RM$.
All of this will be done in the next Section \ref{sec:ore},
for an arbitrary $W$-algebra $W(\mf g,f)$.

\section{Preliminaries: the Kazhdan filtration of \texorpdfstring{$U(\mf g)$}{U(g)},
the Rees algebra \texorpdfstring{$\mc RU(\mf g)$}{RU(g)}
and its extension $\mc R_\infty U(\mf g)$,
and the Rees module \texorpdfstring{$\mc RM$}{RM}
}
\label{sec:ore}

\subsection{The Kazhdan filtration of $U(\mf g)$}
\label{sec:5.1}

Associated to the grading \eqref{eq:grading} of $\mf g$,
we have the so called \emph{Kazhdan filtration} of $U(\mf g)$, defined as follows:
$$
\dots
\subset F_{-\frac12}U(\mf g)
\subset F_0U(\mf g)
\subset F_{\frac12}U(\mf g)
\subset F_1 U(\mf g)
\subset\dots\,\,\subset U(\mf g)
\,,
$$
where
\begin{equation}\label{0312:eq1}
F_\Delta U(\mf g)
=
\sum_{\Delta_1+\dots+\Delta_s\leq\Delta}
\mf g_{1-\Delta_1}\dots\mf g_{1-\Delta_s}
\,\,,\,\,\,\,
\Delta\in\frac12\mb Z
\,.
\end{equation}
In other words, $F_\Delta U(\mf g)$ is the increasing filtration defined letting the degree,
called the \emph{conformal weight}, of $\mf g_j$ equal to $\Delta=1-j$.
This is clearly an algebra filtration, in the sense that 
$F_{\Delta_1}U(\mf g)\cdot F_{\Delta_2}U(\mf g)\subset F_{\Delta_1+\Delta_2}U(\mf g)$,
and we have
\begin{equation}\label{0314:eq2}
[F_{\Delta_1}U(\mf g),F_{\Delta_2}U(\mf g)]
\subset F_{\Delta_1+\Delta_2-1}U(\mf g)
\,.
\end{equation}
It follows that the associated graded $\gr U(\mf g)$ is a graded Poisson algebra,
isomorphic to $S(\mf g)$ endowed with the Kirillov-Kostant Poisson bracket,
graded by the conformal weight.
Note that, by the definition of the Kazhdan filtration, we obviously have
$U(\mf g_{\geq1})\subset F_0U(\mf g)$.

Since $m-(f|m)$ is homogeneous in conformal weight,
the Kazhdan filtration induces the increasing filtration of the left ideal $I$
given by
\begin{equation}\label{0312:eq2}
F_\Delta I
:=F_\Delta U(\mf g)\cap I
=
\sum_{j\geq1}
(F_{\Delta+j-1}U(\mf g))
\big\{m-(f|m)\,\big|\,m\in\mf g_j\big\}
\,.
\end{equation}
\begin{lemma}\label{0316:lem1}
\begin{enumerate}[(a)]
\item
$F_\Delta U(\mf g)=F_\Delta I$
for every $\Delta<0$.
\item
$F_0U(\mf g)=\mb F\oplus F_0I$.
\end{enumerate}
\end{lemma}
\begin{proof}
Let $\Delta\leq0$.
We only need to prove that 
$F_\Delta U(\mf g)\subset F_\Delta I$ for $\Delta<0$,
and $F_0U(\mf g)\subset \mb F+F_0I$.
By \eqref{0312:eq1} and the PBW Theorem, we have
\begin{equation}\label{0316:eq2}
F_\Delta U(\mf g)
=
\sum_{\ell\in\mb Z_{\geq0}}
\sum_{
\substack{\Delta_1\geq\dots\geq\Delta_\ell\in\frac12\mb Z \\
\Delta_1+\dots+\Delta_\ell\leq\Delta}
}
\mf g_{1-\Delta_1}\dots\mf g_{1-\Delta_\ell}
\,.
\end{equation}
The summand with $\ell=0$ is present only for $\Delta=0$,
and in this case it is $\mb F$.
For $\ell\geq1$, there are only two possibilities:
$\Delta_\ell<0$,
or $\Delta=0=\Delta_1=\dots=\Delta_\ell$.
If $\Delta_\ell<0$,
we have $g_{1-\Delta_\ell}\subset\mf g_{\geq\frac32}\subset I$.
Since $I$ is a left module over $U(\mf g)$,
the corresponding summand in \eqref{0316:eq2} is contained in $I\cap F_\Delta U(\mf g)=F_\Delta I$.
If $\Delta=0=\Delta_1=\dots=\Delta_\ell$,
the corresponding summand in \eqref{0316:eq2}
is contained in $\mb F+F_0I$,
since $\mf g_1\subset\mb F+F_0I$.
\end{proof}
By Lemma \ref{0316:lem1}(b), $F_0I$ is a maximal two-sided ideal of the algebra $F_0U(\mf g)$,
and the corresponding factor algebra is $F_0U(\mf g)/F_0I\simeq\mb F$.
Hence, we have a canonical quotient map
\begin{equation}\label{eq:epsilon}
\epsilon_0:\,F_0U(\mf g)=\mb F\oplus\mb F_0I\,\twoheadrightarrow\,\mb F
\,.
\end{equation}
\begin{lemma}\label{0410:lem6}
The map \eqref{eq:epsilon} can be computed by the following formula:
\begin{equation}\label{0410:eq10}
\epsilon_0\big(\sum a_1\dots a_\ell\big) = \sum (f|a_1)\dots (f|a_\ell) 
\,.
\end{equation}
\end{lemma}
\begin{proof}
Let $a_1\in\mf g_{1-\Delta_1},\dots,a_\ell\in\mf g_{1-\Delta_\ell}$.
By assumption, $\Delta_1+\dots+\Delta_\ell\leq0$.
If there is some $\Delta_i>0$,
then there is some $\Delta_j<0$,
which means that $a_j\in\mf g_{1-\Delta_j}\subset\mf g_{\geq\frac32}\subset I$.
Since in $F_0U(\mf g)$ everything commutes modulo $F_0I$,
we get that, in this case, the corresponding monomial $a_1\dots a_\ell$ lies in $F_0I$.
On the other hand, $(f|a_j)=0$, so the monomial $a_1\dots a_\ell$
does not contribute to the RHS of \eqref{0410:eq10} either.
In the case when $\Delta_1,\dots,\Delta_\ell\leq0$,
we have $a_1,\dots,a_\ell\in\mf g_{\geq1}$.
Hence, by the definition of the left ideal $I$, we obtain
$a_1\dots a_\ell\equiv(f|a_1)\dots(f|a_\ell)\mod I$.
\end{proof}
We can consider the induced filtration of the quotient module $M=U(\mf g)/I$,
$$
F_\Delta M=F_\Delta U(\mf g)/F_\Delta I
\,.
$$
\begin{lemma}\label{0316:lem1b}
For $\Delta\leq0$, we have $F_\Delta M=\delta_{\Delta,0}\mb F$.
\end{lemma}
\begin{proof}
It is a restatement of Lemma \ref{0316:lem1}.
\end{proof}
Finally, we consider the restriction of the Kazhdan filtration to the subalgebra $\widetilde{W}\subset U(\mf g)$,
and the $W$-algebra $W(\mf g,f)=\widetilde{W}/I\subset M$:
\begin{equation}\label{0410:eq3}
F_\Delta \widetilde{W}
=
\widetilde{W}\cap F_\Delta U(\mf g)
\,\,,\,\,\,\,
F_\Delta W(\mf g,f)
=
F_\Delta \widetilde{W}/F_\Delta I
\,.
\end{equation}
The associated graded is the Poisson algebra of functions on the Slodowy slice \cite{GG02}.

\subsection{The (completed) Rees algebra $\mc RU(\mf g)$}
\label{sec:5.2}

The (completed) Rees algebra $\mc RU(\mf g)$
is defined as the subalgebra of $U(\mf g)((z^{-\frac12}))$
consisting of the Laurent series in $z^{-\frac12}$
with the property that the coefficient of $z^{\Delta}$ lies in $F_{-\Delta}U(\mf g)$
for all $\Delta\in\frac12\mb Z$:
\begin{equation}\label{eq:rees-alg}
\mc RU(\mf g)
=\Big\{
\sum_{\frac12\mb Z\ni n\leq N}a_nz^n\,\Big|\,a_n\in F_{-n}U(\mf g)\,\text{ for all }\, n\Big\}
\,\subset U(\mf g)((z^{-\frac12}))
\,.
\end{equation}
In the usual Rees algebra one takes only finite sums in \eqref{eq:rees-alg}.
For the sake of brevity we will use the term Rees algebra for its completion \eqref{eq:rees-alg}.
With a more suggestive formula, we can write
\begin{equation}\label{eq:reesb}
\mc RU(\mf g)
=
\widehat{\sum}_{n\in\frac12\mb Z}
z^nF_{-n}U(\mf g)
\,,
\end{equation}
where the completion is defined by allowing 
series with infinitely many negative integer powers of $z^{\frac12}$.

The Rees algebra has the following meaning:
if we extend the Kazhdan filtration to $U(\mf g)[z^{\pm\frac12}]$
by assigning the conformal weight $1$ to $z$,
then, the usual Rees algebra 
$\overline{\mc R}U(\mf g)=\sum_{n\in\frac12\mb Z}z^nF_{-n}U(\mf g)$
coincides with $F_0(U(\mf g)[z^{\pm\frac12}])$,
and $\mc RU(\mf g)$ is obtained as its completion.

Note that, since $1\in F_0U(\mf g)\subset F_{\frac12}U(\mf g)$,
we have that $z^{-\frac12}$ is a central element of the Rees algebra $\mc RU(\mf g)$.
Hence, multiplication by $z^{-\frac12}$
defines an injective endomorphism of $\mc RU(\mf g)$,
$$
z^{-\frac12}\,:\,\,
\mc RU(\mf g)
\hookrightarrow
\mc RU(\mf g)
\,,
$$
commuting with multiplication by elements of $\mc RU(\mf g)$.
On the other hand, $z^{\frac12}$ (or $z$) cannot be seen as an element of $\mc RU(\mf g)$,
since $1\not\in F_{-\frac12}U(\mf g)$.
Multiplication by $z^{-\frac12}$ defines a decreasing filtration of the Rees algebra
\begin{equation}\label{0409:eq1}
\dots
\subset z^{-\frac32}\mc RU(\mf g)
\subset z^{-1}\mc RU(\mf g)
\subset z^{-\frac12}\mc RU(\mf g)
\subset \mc RU(\mf g)
\,.
\end{equation}
In analogy with \eqref{eq:reesb}, we have
$z^{-\frac12}\mc RU(\mf g)=\widehat{\sum}_{n\in\frac12\mb Z}z^nF_{-n-\frac12}U(\mf g)$.
Hence,
\begin{equation}\label{0409:eq2}
\mc RU(\mf g)/(z^{-\frac12})
=
\widehat{\sum}_{n\in\frac12\mb Z}z^{-n}\gr_{n}U(\mf g)
\simeq\widehat{\gr}U(\mf g)
\,,
\end{equation}
is the completion (allowing infinite sums with arbitrarily large conformal weights) 
of the associated graded of $U(\mf g)$ with respect to the Kazhdan filtration.
\begin{remark}
The usual Rees algebra $\overline{\mc R}U(\mf g)$
intertwines between the associated graded algebra $\gr U(\mf g)$ and the algebra $U(\mf g)$ itself.
Indeed, we have $\overline{\mc R}U(\mf g)/(z^{-\frac12})\simeq\gr U(\mf g)$ (cf. \eqref{0409:eq2}),
and $\overline{\mc R}U(\mf g)/(z^{-\frac12}-1)\simeq U(\mf g)$
(via the map $\sum_nz^na_n\mapsto \sum_na_n$).
\end{remark}

Recall the algebra homomorphism 
$\epsilon_0:\,F_0U(\mf g)\twoheadrightarrow\mb F$ defined in \eqref{eq:epsilon}.
We extend it to a surjective linear map $\epsilon:\,\mc RU(\mf g)\twoheadrightarrow\mb F$, given by
\begin{equation}\label{0406:eq1}
a(z)=\sum_{\frac12\mb Z\ni n\leq N}a_nz^n\mapsto \epsilon(a(z)):=\epsilon_0(a_0)
\,.
\end{equation}
\begin{lemma}
The map $\epsilon:\,\mc RU(\mf g)\twoheadrightarrow\mb F$ defined by \eqref{0406:eq1}
is an algebra homomorphism.
\end{lemma}
\begin{proof}
Let $a(z)=\sum_{n\leq N}a_nz^n,\,b(z)=\sum_{n\leq N}b_nz^n\,\in\mc RU(\mf g)$.
We have
$$
\epsilon(a(z)b(z))
=\sum_{-N\leq n\leq N}\epsilon_0(a_{-n}b_n)
\,.
$$
For $n>0$, we have $b_n\in F_{-n}U(\mf g)=F_{-n}I$,
hence $a_{-n}b_n\in F_0I$, and $\epsilon_0(a_{-n}b_n)=0$.
For $n<0$, for the same reason as above we have $b_na_{-n}\in F_0I$,
but we also have $[a_{-n},b_n]\in F_{-1}U(\mf g)=F_{-1}I\subset F_0I$.
Hence, $\epsilon_0(a_{-n}b_n)=\epsilon_0(b_na_{-n})+\epsilon_0([a_{-n},b_n])=0$.
In conclusion,
$$
\epsilon(a(z)b(z))
=\epsilon_0(a_0b_0)
=\epsilon_0(a_0)\epsilon(b_0)
=\epsilon(a(z))\epsilon(b(z))
\,.
$$
\end{proof}

\subsection{The Rees module $\mc RM$}
\label{sec:5.3}

Consider the left ideal $I\subset U(\mf g)$, with the induced filtration \eqref{0312:eq2}.
We can consider the corresponding left ideal $\mc RI$ of the Rees algebra $\mc RU(\mf g)$, defined,
with the same notation as in \eqref{eq:reesb}, as
\begin{equation}\label{eq:reesI}
\mc RI
=
\widehat{\sum}_{n\in\frac12\mb Z}
z^nF_{-n}I
\,\subset I((z^{-\frac12}))
\,.
\end{equation}
Taking the quotient of the algebra $\mc RU(\mf g)$ by its left ideal $\mc RI$ 
we get the corresponding Rees module
\begin{equation}\label{eq:reesM}
\mc RM
=
\mc RU(\mf g)/\mc RI
=
\widehat{\sum}_{n\in\frac12\mb Z_{\leq0}}
z^nF_{-n}M
\,\subset M[[z^{-\frac12}]]
\,.
\end{equation}

By construction, the Rees module $\mc RM$ is a left cyclic module over the Rees algebra $\mc RU(\mf g)$,
generated by $\bar1\in\mc RM$.
By Lemma \ref{0316:lem1b},
its elements are Taylor series in $z^{-\frac12}$
with constant term in the base field $\mb F$.
In other words, we have 
\begin{equation}\label{0411:eq1}
\mc RM
=
\mb F\bar1\oplus\mc R_-M
\,\text{ where }\,
\mc R_-M
=
\widehat{\sum}_{n\leq-\frac12}
z^nF_{-n}M
\,\subset z^{-\frac12}M[[z^{-\frac12}]]
\,.
\end{equation}
\begin{lemma}\label{0410:lem1}
\begin{enumerate}[(a)]
\item
Multiplication by $z^{-\frac12}$ defines an endomorphism of $\mc RM$ commuting with the action 
of $\mc RU(\mf g)$.
\item
The annihilator of the cyclic element $\bar 1\in\mc RM$ is the left ideal $\mc RI\subset\mc RU(\mf g)$.
\item
$\mc RI\cdot\mc RM\subset z^{-1}\mc RM$.
\item
$\mc R_-M$ is a submodule of the $\mc RU(\mf g)$-module $\mc RM$.
\item
The action of $\mc RU(\mf g)$ 
on the quotient module $\mc RM/\mc R_-M=\mb F\bar1$
is given by the homomorphism $\epsilon:\,\mc RU(\mf g)\to\mb F$
defined in \eqref{0406:eq1}:
$$
a(z)\bar 1\equiv \epsilon(a(z))\bar 1\mod\mc R_-M
\,.
$$
\end{enumerate}
\end{lemma}
\begin{proof}
Claim (a) is obvious, since $z^{-\frac12}\in\mc RU(\mf g)$ is central.
Claim (b) holds by definition.
Let us prove claim (c).
Let $i(z)=\sum_{n\leq N}i_nz^n\in\mc RI$, where $i_n\in F_{-n}I$,
and let $\bar g(z)=g(z)\bar 1\in\mc RM$, 
where $g(z)=\sum_{n\leq 0}g_nz^n\in\mc RU(\mf g)$.
We have
$$
i(z)g(z)\bar 1
=
\sum_{n\leq N} 
\sum_{n-N\leq m\leq0}
z^ni_{n-m}g_m\bar 1
\,.
$$
Since $z^ng_mi_{n-m}\bar 1=0$,
and since $[i_{n-m},g_m]\in F_{-n-1}U(\mf g)$,
we conclude that
$i(z)g(z)\bar 1\in z^{-1}\mc RM$, as claimed.
Next, we prove claim (d).
An element $g(z)=\sum_{n\leq N}g_nz^n\in\mc RU(\mf g)$
decomposes as $g(z)=g_+(z)+g_0+g_-(z)$,
where
\begin{equation}\label{0411:eq2}
\begin{array}{l}
\displaystyle{
\vphantom{\Big(}
g_+(z)=\sum_{\frac12\leq n\leq N}g_nz^n\,\,\in\sum_{n=1}^Nz^nF_{-n}I\subset\mc RI
\,,} \\
\displaystyle{
\vphantom{\Big(}
g_0=k+i_0\,\,\in F_0U(\mf g)=\mb F\oplus F_0I
\,\,\,\,
k\in\mb F,\,i_0\in F_0I
\,,} \\
\displaystyle{
\vphantom{\Big(}
g_-(z)=\sum_{n\leq-\frac12}g_nz^n\,\,\in\widehat{\sum}_{n\leq-\frac12}z^nF_{-n}U(\mf g)
\,.}
\end{array}
\end{equation}
Clearly, $g_-(z)\cdot\mc RM\subset\mc R_-M$.
On the other hand, $g_+(z)+i_0\in\mc RI$,
hence, by claim (c), we have $(g_+(z)+i_0)\cdot\mc RM\subset z^{-1}\mc RM\subset\mc R_-M$.
In conclusion, $g(z)\cdot\mc R_-M\subset\mc R_-M$, as claimed.
Moreover, by the above observations, we also have 
$(g(z)-k)=(g_+(z)+i_0+g_-(z))\cdot\mc RM\subset\mc R_-M$.
Hence,
$g(z)\bar1\equiv k\bar1\mod\mc R_-M$.
Since, by definition, $\epsilon(g(z))=k$, this proves claim (e).
\end{proof}
\begin{proposition}\label{0410:prop1}
An element $g(z)\in\mc RU(\mf g)$ acts as an invertible endomorphism of $\mc RM$
if and only if $\epsilon(g(z))\neq0$.
\end{proposition}
\begin{proof}
First, if $g(z)$ acts as an invertible endomorphism of $\mc RM$,
then it must act as an invertible endomorphism of the quotient module
$\mc RM/\mc R_-M=\mb F\bar1$.
Hence, by Lemma \ref{0410:lem1}(e), we have $\epsilon(g(z))\neq0$.
We are left to prove the ``if'' part.
Let $g(z)$ decompose as in \eqref{0411:eq2},
and we assume that $\epsilon(g(z))=k\neq0$.
Clearly, the element $k+g_-(z)$ is invertible in the Rees algebra $\mc RU(\mf g)$,
and its inverse can be computed by geometric series expansion:
\begin{equation}\label{0410:eq12}
h(z):=(k+g_-(z))^{-1}
=\sum_{n=0}^\infty (-1)^nk^{-n-1} g_-(z)^n
\,.
\end{equation}
Indeed, by construction, $g_-(z)^n$ lies in $z^{-\frac{n}2}U(\mf g)[[z^{-\frac12}]]$,
and therefore the series \eqref{0410:eq12} is well defined.
Then we have
\begin{equation}\label{0410:eq13}
g(z)=
(k+g_-(z))
(1+h(z)(g_+(z)+i_0))
\,.
\end{equation}
Recall that $g_+(z)+i_0\in\mc RI$, which is a left ideal of the Rees algebra $\mc RU(\mf g)$.
Hence,
\begin{equation}\label{0411:eq3}
i(z):=h(z)(g_+(z)+i_0)\,\in\mc RI\,,
\end{equation}
and we are left to prove, by \eqref{0410:eq13}, that $1+i(z)$ acts as an invertible endomorphism
of $\mc RM$.
By Lemma \ref{0410:lem1},
for every $\bar g(z)\in\mc RM$,
we have $i(z)^n\bar g(z)\in z^{-n}\mc RM$.
Hence, the geometric series expansion
\begin{equation}\label{0411:eq4}
(1+i(z))^{-1}\bar g(z)
=
\sum_{n=0}^\infty (-1)^ni(z)^n\bar g(z)
\,,
\end{equation}
is defined in $\mc RM$.
\end{proof}

\subsection{The localized Rees algebra $\mc R_\infty U(\mf g)$}
\label{sec:5.4}

By Proposition \ref{0410:prop1}, an element $g(z)\in\mc RU(\mf g)$, with $\epsilon(g(z))\neq0$,
acts as an invertible endomorphism of the Rees module $\mc RM$.
But, in general, the inverse of $g(z)$ does not necessarily exist in the Rees algebra,
since the geometric series expansion $\sum_{n=0}^\infty (-1)^ni(z)^n$ in \eqref{0411:eq4}
may involve infinitely many positive powers of $z$.
We shall construct an algebra extension $\mc R_\infty U(\mf g)$ of $\mc RU(\mf g)$, 
still acting on the module $\mc RM$, on which we extend the map \eqref{0406:eq1}
to an algebra homomorphism 
$\epsilon:\,\mc R_\infty U(\mf g)\to\mb F$,
with the fundamental property that an element $\alpha(z)\in\mc R_-U(\mf g)$
is invertible if and only if $\epsilon(\alpha(z))\neq0$.

The algebra $\mc R_\infty U(\mf g)$ will be obtained by a limiting procedure.
We start by constructing the algebra $\mc R_1U(\mf g)$.
It is defined as the subalgebra of a skewfield extension of $\mc RU(\mf g)$
(contained in the skewfield of fractions of $U((z^{-\frac12}))$)
generated by $\mc RU(\mf g)$ and the inverse of the elements $a(z)\in\mc RU(\mf g)$
such that $\epsilon(a(z))\neq0$.
In other words, elements of $\mc R_1U(\mf g)$ are finite sums of the form
\begin{equation}\label{0410:eq1}
\alpha(z)
=
\sum a_1(z)b_1(z)^{-1}\dots a_\ell(z)b_\ell(z)^{-1}
\,\in\mc R_1U(\mf g)
\,,
\end{equation}
where 
$a_1(z),\dots,a_\ell(z),b_1(z),\dots,b_\ell(z)\,\in\mc RU(\mf g)$,
and $\epsilon(b_1(z))\cdot\dots\cdot\epsilon(b_\ell(z))\neq0$.
Note that, since $\epsilon(b_i(z))\neq0$ for all $i$,
we can naturally extend the map \eqref{0406:eq1}
to an algebra homomorphism $\epsilon:\,\mc R_1U(\mf g)\to\mb F$,
given, on an element of the form \eqref{0410:eq1}, by
\begin{equation}\label{0411:eq5}
\epsilon(\alpha(z))
=
\sum \frac{\epsilon{a_1(z)}\dots\epsilon(a_\ell(z))}{\epsilon(b_1(z))\dots\epsilon(b_\ell(z))}
\,\in\mb F
\,.
\end{equation}
Note also that,
as a consequence of Proposition \ref{0410:prop1},
the action of the Rees algebra $\mc RU(\mf g)$ on $\mc RM$
naturally extends to an action of $\mc R_1U(\mf g)$ on $\mc RM$.
\begin{proposition}\label{0411:lem1}
The algebra $\mc R_1U(\mf g)$ has the following properties:
\begin{enumerate}[(i)]
\item
$z^{-1}$ is central in $\mc R_1U(\mf g)$;
\item
the subspace $\mc R_-M\subset\mc RM$ is preserved by the action of $\mc R_1U(\mf g)$;
\item
for every $\alpha(z)\in\mc R_1U(\mf g)$, we have
$\alpha(z)\bar1\equiv\epsilon(\alpha(z))\bar1\,\mod\mc R_-M$,
where $\epsilon$ is the map \eqref{0411:eq5};
\item
for every $\alpha(z)\in\mc R_1U(\mf g)$ and every integer $N\geq0$,
there exist $\alpha_N(z)\in\mc RU(\mf g)$ such that
$(\alpha(z)-\alpha_N(z))\cdot\mc RM\subset z^{-N-1}\mc RM$;
\item
an element $\alpha(z)\in\mc R_1U(\mf g)$ acts as an invertible endomorphism of $\mc RM$
if and only if $\epsilon(\alpha(z))\neq0$.
\end{enumerate}
\end{proposition}
\begin{proof}
Claim (i) is obvious.
By Lemma \ref{0410:lem1}, $\mc R_-M$ is preserved by the action of $\mc RU(\mf g)$.
If $b(z)\in\mc RU(\mf g)$ is such that $\epsilon(b(z))\neq0$,
then, by Proposition \ref{0410:prop1}, it acts as an invertible endomorphism of $\mc RM$,
preserving the subspace $\mc R_-M$, which is of finite (actually 1) codimension.
Then, it is a simple linear algebra exercise to show that the inverse 
$b(z)^{-1}$ preserves $\mc R_-M$ as well.
As a consequence, every element \eqref{0410:eq1} of $\mc R_1U(\mf g)$
preserves $\mc R_-M$, proving (ii).

By Lemma \ref{0410:lem1}(e), we have $a(z)\bar 1\equiv\epsilon(a(z))\bar1\mod\mc R_-M$
for every $a(z)\in\mc RU(\mf g)$,
and by (ii) we also have $b(z)^{-1}\bar 1\equiv\frac{\bar1}{\epsilon(b(z))}\mod\mc R_-M$
for every $b(z)\in\mc RU(\mf g)$ such that $\epsilon(b(z))\neq0$.
Hence, claim (iii) follows immediately from (ii).

Next, we prove claim (iv).
First note that, if $b(z)\in\mc RM$ is such that $k=\epsilon(b(z))\neq0$,
then, decomposing it as we did in \eqref{0410:eq13},
we can write its inverse as
$$
b(z)^{-1}=(1+i(z))^{-1}h(z)\,,
$$
for some $i(z)\in\mc RI$, and $h(z)\in\mc RU(\mf g)$.
Hence, the generic element \eqref{0410:eq1} of $\mc R_1U(\mf g)$
can be assumed to have the form
\begin{equation}\label{0410:eq1b}
\alpha(z)
=
\sum g_1(z)(1+i_1(z))^{-1}\dots g_\ell(z)(1+i_\ell(z)^{-1})g_{\ell+1}(z)
\,,
\end{equation}
for some $g_1(z),\dots,g_{\ell+1}(z)\in\mc RU(\mf g)$,
and $i_1(z),\dots,i_\ell(z)\in\mc RI$.
Then, given the integer $N\geq0$,
we construct $\alpha_N(z)\in\mc RU(\mf g)$ 
by expanding each inverse $(1+i_s(z))^{-1}$ appearing in the formal expression
\eqref{0410:eq1b} by geometric series up to the $N$-th power:
\begin{equation}\label{0410:eq2}
\alpha_N(z)
=
\sum 
\sum_{n_1,\dots,n_\ell=0}^N
g_1(z)(-i_1(z))^{n_1}\dots g_\ell(z)(-i_\ell(z))^{n_\ell}g_{\ell+1}(z)
\,\in\mc RU(\mf g)
\,.
\end{equation}
It follows immediately from Lemma \ref{0410:lem1}(c)
(and the definition of the action of $(1+i(z))^{-1}$ on $\mc RM$,
by geometric series expansion, cf. \eqref{0411:eq4})
that $(\alpha(z)-\alpha_N(z))\cdot\mc RM\subset z^{-N-1}\mc RM$,
proving (iv).

Finally, we prove claim (v).
First, we prove the ``only if'' part.
Let $\alpha(z)\in\mc R_1U(\mf g)$ be such that $\epsilon(\alpha(z))=0$.
By (ii) we have $\alpha(z)\cdot R_-M\subset\mc R_-M$,
and by (iii) we also have $\alpha(z)\bar1\in\mc R_-M$.
Hence, $\alpha(z)\cdot\mc RM\subset\mc R_-M$,
so the action of $\alpha(z)$ on $\mc RM$ is not surjective,
and therefore not invertible.
We are left to prove the ``if'' part.
Suppose that $\epsilon(\alpha(z))\neq0$.
By (iv), there is a sequence of elements $\alpha_0(z),\alpha_1(z),\alpha_2(z),\dots\in\mc RU(\mf g)$
with the property that
\begin{equation}\label{0412:eq1a}
(\alpha(z)-\alpha_N(z))\cdot\mc RM\subset z^{-N-1}\mc RM
\,.
\end{equation}
Note that, then, since $\alpha_n-\alpha_N=(\alpha-\alpha_N)-(\alpha-\alpha_n)$,
we also have
\begin{equation}\label{0412:eq1}
(\alpha_n(z)-\alpha_N(z))\cdot\mc RM\,\subset\, z^{-N-1}\mc RM
\,\text{ for every }\,
n\geq N\,.
\end{equation}
Note also that, since $z^{-N-1}\mc RM\subset\mc R_-M$,
it follows by (iii) that
$$
0\neq\epsilon(\alpha(z))=\epsilon(\alpha_N(z))
\,\text{ for every }\, N\geq0
\,.
$$
Hence, by Proposition \ref{0410:prop1},
each $\alpha_N(z)$ acts as an invertible endomorphism of $\mc RM$,
and, by the construction of $\mc R_1U(\mf g)$, it is invertible in $\mc R_1U(\mf g)$.
Let us denote by $\beta_N(z)\in\mc R_1U(\mf g)$ its inverse:
\begin{equation}\label{0412:eq3}
\beta_N(z)\alpha_N(z)=\alpha_N(z)\beta_N(z)=1\,\text{ in }\,\mc R_1U(\mf g)
\,.
\end{equation}
Since $z^{-1}$ commutes with $\alpha_N(z)$, it also commutes with $\beta_N(z)$.
Moreover, by the obvious identity
$$
\beta_n(z)-\beta_N(z)
=
\beta_n(z)(\alpha_N(z)-\alpha_n(z))\beta_N(z)
\,,
$$
we have, as a consequence of \eqref{0412:eq1}, that
$$
(\beta_n(z)-\beta_N(z))\cdot\mc RM\,\subset\, z^{-N-1}\mc RM
\,\text{ for every }\,
n\geq N\,.
$$
Hence, there is a well defined limit
\begin{equation}\label{0412:eq2}
\beta(z):=\lim_{N\to\infty}\beta_N(z)\,\in\,\End(\mc RM)
\,,
\end{equation}
and by construction 
\begin{equation}\label{0412:eq1b}
(\beta(z)-\beta_N(z))\cdot\mc RM\subset z^{-N-1}\mc RM
\,.
\end{equation}
We claim that $\beta(z)$ is the inverse of $\alpha(z)$ in $\End(\mc RM)$.
Indeed, for every $N\geq0$ we have the following identity in $\End(\mc RM)$:
$$
\alpha(z)\beta(z)
=
\alpha_N(z)\beta_N(z)
+\alpha_N(z)(\beta(z)-\beta_N(z))+(\alpha(z)-\alpha_N(z))\beta(z)
\,.
$$
Hence, by \eqref{0412:eq1a}, \eqref{0412:eq3} and \eqref{0412:eq1b},
we get
\begin{equation}\label{0412:eq4}
(\alpha(z)\beta(z)-\id_{\mc RM})(\mc RM)\subset z^{-N-1}\mc RM
\,.
\end{equation}
Since \eqref{0412:eq4} holds for every $N\geq0$ and since obviously, $\cap_Nz^{-N-1}\mc RM=0$,
we get that $\alpha(z)\beta(z)=\id_{\mc RM}$.
Similarly for $\beta(z)\alpha(z)$.
\end{proof}
Suppose, by induction on $n\geq2$, that we constructed an algebra $\mc R_{n-1}U(\mf g)$
satisfying the properties (i)-(v) of Proposition \eqref{0410:prop1}.
We then define the algebra extension $\mc R_nU(\mf g)$ of $\mc R_{n-1}U(\mf g)$
as the subalgebra of the skewfield of fractions of $\mc RU(\mf g)$
generated by $\mc R_{n-1}U(\mf g)$ and the inverses of the elements $a(z)\in\mc R_{n-1}U(\mf g)$
such that $\epsilon(a(z))\neq0$.
In other words, elements of $\mc R_nU(\mf g)$ are formal expressions of the form \eqref{0410:eq1}
with $a_1(z),\dots,a_\ell(z),b_1(z),\dots,b_\ell(z)\in\mc R_{n-1}U(\mf g)$,
and $\epsilon(b_1(z))\neq0,\dots,\epsilon(b_\ell(z))\neq0$.
We extend $\epsilon$ 
to an algebra homomorphism $\epsilon:\,\mc R_nU(\mf g)\to\mb F$,
given by \eqref{0411:eq5}.
Moreover, as a consequence of the property (iv) of $\mc R_{n-1}U(\mf g)$,
the action of $\mc R_{n-1}U(\mf g)$ on $\mc RM$
naturally extends to an action of $\mc R_nU(\mf g)$ on $\mc RM$.
\begin{proposition}\label{0411:lem2}
Properties (i)-(v) of Proposition \eqref{0411:lem1} hold for $\mc R_nU(\mf g)$.
\end{proposition}
\begin{proof}
The proofs of properties (i), (ii), (iii)
are the same as the corresponding proofs of (i), (ii) and (iii) in Proposition \ref{0411:lem1},
with $\mc RU(\mf g)$ and $\mc R_1U(\mf g)$ replaced 
by $\mc R_{n-1}U(\mf g)$ and $\mc R_nU(\mf g)$ respectively.
Property (iv) requires a proof.
Let $\alpha(z)\in\mc R_nU(\mf g)$ be as in \eqref{0410:eq1},
with $a_1(z),\dots,a_\ell(z),b_1(z),\dots,b_\ell(z)\in\mc R_{n-1}U(\mf g)$,
and $\epsilon(b_1(z)),\dots,\epsilon(b_\ell(z))\neq0$.
By the inductive assumption, there exist
elements $b_{1,N}(z),\dots,b_{\ell,N}(z)\in\mc RU(\mf g)$
such that $(b_i(z)-b_{i,N}(z))\cdot\mc RM\subset z^{-N-1}\mc RM$.
Since $z^{-N-1}\mc RM\subset\mc R_-M$,
we have $\epsilon(b_{i,N}(z))=\epsilon(b_i(z))\neq0$,
so each $b_{i,N}(z)$ is invertible in $\mc R_1U(\mf g)\subset\mc R_{n-1}U(\mf g)$.
We let
\begin{equation}\label{0412:eq5}
\tilde{\alpha}_N(z)
=
\sum a_1(z)b_{1,N}(z)^{-1}\dots a_\ell(z)b_{\ell,N}(z)^{-1}
\,\in\mc R_{n-1}U(\mf g)
\,.
\end{equation}
By construction, 
\begin{equation}\label{0412:eq6}
(\alpha(z)-\tilde{\alpha}_N(z))\cdot\mc RM\subset z^{-N-1}\mc RM
\,.
\end{equation}
But then, again by the inductive assumption, there exists $\alpha_N(z)\in\mc RU(\mf g)$ 
such that
\begin{equation}\label{0412:eq7}
(\tilde{\alpha}_N(z)-\alpha_N(z))\cdot\mc RM\subset z^{-N-1}\mc RM
\,.
\end{equation}
Combining \eqref{0412:eq6} and \eqref{0412:eq7},
we get $(\alpha(z)-\alpha_N(z))\cdot\mc RM\subset z^{-N-1}\mc RM$,
as desired.
The proof of property (v)
is the same as the corresponding proofs in Proposition \ref{0411:lem1},
with $\mc R_1U(\mf g)$ replaced by $\mc R_nU(\mf g)$.
\end{proof}
We constructed an increasing sequence of algebras
$$
\mc RU(\mf g)\subset\mc R_1U(\mf g)\subset\mc R_2U(\mf g)\subset\dots
\,,
$$
each acting on the Rees module $\mc RM$,
and each satisfying the properties (i)-(v) of Proposition \ref{0411:lem1}.
We then define the localized Rees algebra $\mc R_\infty U(\mf g)$ as their union.
By construction, $\mc R_\infty U(\mf g)$ naturally acts on the Rees module $\mc RM$,
we have a canonical homomorphism $\epsilon:\,\mc R_\infty U(\mf g)\to\mb F$
such that 
\begin{equation}\label{0411:eq6}
\alpha(z)\bar1\equiv\epsilon(\alpha(z))\bar1\mod\mc R_-M
\,\text{ for every }\,
\alpha(z)\in\mc R_\infty U(\mf g)
\,,
\end{equation}
and we have
\begin{corollary}\label{0411:cor}
For en element $\alpha(z)\in\mc R_\infty U(\mf g)$,
the following conditions are equivalent:
\begin{enumerate}[1)]
\item
$\alpha(z)$ is invertible in $\mc R_\infty U(\mf g)$;
\item
$\alpha(z)$ acts as an invertible endomorphism of $\mc RM$;
\item
$\epsilon(\alpha(z))\neq0$.
\end{enumerate}
Moreover, for every $\alpha(z)\in\mc R_\infty U(\mf g)$ and every integer $N\geq0$,
there exists $\alpha_N(z)\in\mc RU(\mf g)$ such that
\begin{equation}\label{0413:eq2}
(\alpha(z)-\alpha_N(z))\cdot\mc RM\subset z^{-N-1}\mc RM
\,.
\end{equation}
\end{corollary}
\begin{proof}
Clearly 1) implies 2).
If $\alpha(z)\in\mc R_\infty U(\mf g)$,
then $\alpha(z)\in\mc R_nU(\mf g)$ for some $n\geq0$.
Hence, by Proposition \ref{0411:lem2}(v),
it acts as an invertible endomorphism of $\mc RM$ if and only if $\epsilon(\alpha(z))\neq0$,
and in this case $\alpha(z)$ is invertible in $\mc R_{n+1}U(\mf g)\subset\mc R_\infty U(\mf g)$.
This proves that 2) and 3) are equivalent, and they imply 1).
The last statement of the corollary is clear, 
thanks to property (iv) of Proposition \ref{0411:lem2}.
\end{proof}
The above result can be generalized to the matrix case:
\begin{corollary}\label{0410:prop3}
A square matrix $A(z)\in\Mat_{N\times N}\mc R_\infty U(\mf g)$
is invertible if and only if $\epsilon(A(z))\in\Mat_{N\times N}\mb F$ is non-degenerate.
\end{corollary}
\begin{proof}
If $A(z)$ is invertible in $\Mat_{N\times N}\mc R_\infty U(\mf g)$,
then, $\epsilon:\,\mc R_\infty U(\mf g)\to\mb F$ is a homomorphism,
the matrix $\epsilon(A(z))\,\in\Mat_{N\times N}\mb F$ is automatically invertible, 
$\epsilon(A(z)^{-1})$ being its inverse.
This proves the ``only if'' part.
For the converse,
let $A(z)\in\Mat_{N\times N}\mc R_\infty U(\mf g)$ be such that 
$\bar A:=\epsilon(A(z))\in\Mat_{N\times N}\mb F$ is an invertible matrix.
Replacing $A(z)$ with $\bar A^{-1}A(z)$,
we reduce to the case when $\epsilon(A(z))=\id_N$.
In this case, it is easy to construct the inverse of the matrix $A(z)$ by the usual Gauss elimination.
\end{proof}
\begin{remark}\label{0410:rem}
It would be interesting to know whether
the multiplicative set $S=\mc RU(\mf g)\backslash\ker\epsilon$ were an Ore subset 
of the Rees algebra $\mc RU(\mf g)$.
If this were the case, the inductive construction above would not be needed,
as it would be 
$$
\mc R_\infty U(\mf g)=\mc R_1U(\mf g)=\mc RU(\mf g)S^{-1}=\{a(z)b(z)^{-1}\,|\,\epsilon(b(z))\neq0\}
\,.
$$
In this case, many arguments in the present paper would simplify.
\end{remark}

\subsection{The Rees $W$-algebra $\mc RW(\mf g,f)$}
\label{sec:5.5}

We can consider the Rees algebra of the $W$-algebra $W(\mf g,f)\subset M$,
defined as
\begin{equation}\label{eq:reesW}
\mc RW(\mf g,f)
=
\widehat{\sum}_{n\in\frac12\mb Z_{\leq0}}
z^nF_{-n}W(\mf g,f)
\,\subset \mc RM
\,.
\end{equation}
It is a subalgebra of $W(\mf g,f)[[z^{-\frac12}]]$.
Equivalently (cf. \eqref{0225:eq3}, \eqref{0225:eq8} and \eqref{0410:eq3}), it is
$$
\mc RW(\mf g,f)
=
\mc R\widetilde W/\mc RI
\,,
$$
where
\begin{equation}\label{eq:reesM2}
\begin{array}{l}
\displaystyle{
\vphantom{\Big(}
\mc R\widetilde{W}
=
\widehat{\sum}_{n\in\frac12\mb Z_{\leq0}}
z^nF_{-n}\widetilde{W}
} \\
\displaystyle{
\vphantom{\Big(}
=
\big\{
g(z)\in\mc RU(\mf g)
\,\big|\,
[a,g(z)]\bar 1=0
\,\text{ for all }\,a\in\mf g_{\geq\frac12}
\big\}
\,\subset\mc RU(\mf g)
\,.}
\end{array}
\end{equation}
Note that $\mf g_{\geq\frac12}\subset F_0U(\mf g)\subset\mc RU(\mf g)$,
hence the adjoint action of $a\in\mf g_{\geq\frac12}$ on $\mc RU(\mf g)$ is well defined.
The ``Rees algebra'' analogue of Lemma \ref{0225:prop1a} holds:
\begin{lemma}\phantomsection\label{0225:prop1a-rees}
$\mc R\widetilde{W}\subset \mc RU(\mf g)$ is a subalgebra of the Rees algebra $\mc RU(\mf g)$,
and $\mc RI\subset\mc R\widetilde{W}$ is its two-sided ideal.
\end{lemma}
\begin{proof}
It is an obvious consequence of Lemma \ref{0225:prop1a}.
\end{proof}
We can also consider the extension of the algebra $\mc R\widetilde{W}$ inside $\mc R_\infty\widetilde{W}$,
defined as
\begin{equation}\label{0413:eq1}
\mc R_\infty\widetilde{W}
:=
\big\{
\alpha(z)\in\mc R_\infty U(\mf g)
\,\big|\,
[a,\alpha(z)]\bar 1=0
\,\text{ for all }\,a\in\mf g_{\geq\frac12}
\big\}
\,\subset\mc R_\infty U(\mf g)
\,.
\end{equation}
Clearly, $\mc R\widetilde{W}\subset\mc R_\infty\widetilde{W}\subset\mc R_\infty U(\mf g)$.
Hence, we have
$\mc RW(\mf g,f)=\mc R\widetilde{W}\bar1\subset\mc R_\infty\widetilde{W}\bar1
=:\mc R_\infty W(\mf g,f)\subset\mc RM$,
and it is natural to ask what is $\mc R_\infty W(\mf g,f)$.
The following result says, in fact, that $\mc RW(\mf g,f)=\mc R_\infty W(\mf g,f)$.
\begin{proposition}\label{0413:prop1}
Let $\alpha(z)\in\mc R_\infty U(\mf g)$ and $g(z)\in\mc RU(\mf g)$ be such that
$\alpha(z)\bar1=g(z)\bar1=:w(z)\in\mc RM$.
The following conditions are equivalent:
\begin{enumerate}[(a)]
\item
$[a,\alpha(z)]\bar1=0$ for all $a\in\mf g_{\frac12}$ (i.e., $\alpha(z)\in\mc R_\infty\widetilde{W}$);
\item
$[a,g(z)]\bar1=0$ for all $a\in\mf g_{\frac12}$ (i.e., $g(z)\in\mc R\widetilde{W}$);
\item
$w(z)\in\mc RW(\mf g,f)$.
\end{enumerate}
\end{proposition}
\begin{proof}
First, note that conditions (b) and (c) are equivalent by definition of the Rees $W$-algebra $\mc RW(\mf g,f)$.
Let $N\geq0$, and let $\alpha_N(z)\in\mc RU(\mf g)$ satisfy condition \eqref{0413:eq2}.
Then,
$$
(g(z)-\alpha_N(z))\bar1=(\alpha(z)-\alpha_N(z))\bar1\,\in z^{-N-1}\mc RM
=z^{-N-1}\mc RU(\mf g)\bar1
\,,
$$
which means that
\begin{equation}\label{0413:eq3}
g(z)-\alpha_N(z)
=
i_N(z)+z^{-N-1}r_N(z)\,\in\mc RI+z^{-N-1}\mc RU(\mf g)
\subset\mc RU(\mf g)
\,.
\end{equation}
But then, we have, for $a\in\mf g_{\geq\frac12}$,
\begin{equation}\label{0413:eq4}
[a,\alpha(z)]\bar1
=
a\alpha(z)\bar1-\alpha(z)a\bar1
=
[a,g(z)]\bar1
+(g(z)-\alpha_N(z))a\bar1
+(\alpha_N(z)-\alpha(z))a\bar1
\,.
\end{equation}
By \eqref{0413:eq2} the last term in the RHS of \eqref{0413:eq4} lies in $z^{-N-1}\mc RM$.
Moreover, by \eqref{0413:eq3}, and by the fact that $\mc RI$ is preserves by right multiplication
by elements in $\mf g_{\geq\frac12}$ (cf. Lemma \ref{0303:lem4}(a)),
also the second term in the RHS of \eqref{0413:eq4} lies in $z^{-N-1}\mc RM$.
Hence,  \eqref{0413:eq4} gives
$$
[a,\alpha(z)]\bar1
\equiv
[a,g(z)]\bar1
\,\mod z^{-N-1}\mc RM
\,.
$$
Since this holds for every $N\geq0$,
we conclude that $[a,\alpha(z)]\bar1=[a,g(z)]\bar1$.
The claim follows.
\end{proof}

\section{The Main Lemma}
\label{sec:main-lemma}

\subsection{Statement of the Main Lemma}
\label{sec:4.5}

Consider the matrix $E$ defined in \eqref{eq:E},
which, in the notation of Section \ref{sec:4.1}, is
$E=\sum_{(ih),(jk)\in\mc T}e_{(jk)(ih)}E_{(ih)(jk)}$.
The $((ih)(jk))$-entry is $e_{(jk)(ih)}$, which is an element of conformal weight
(cf. \eqref{eq:adx}, \eqref{eq:xih})
\begin{equation}\label{eq:Delta}
\Delta(e_{(jk)(ih)})
=1+x(ih)-x(jk)
=1+\frac12(p_i-p_j)-(h-k)
\,.
\end{equation}
Recalling the definition \eqref{eq:reesb} of the Rees algebra,
$z^{-\Delta(e_{(jk)(ih)})}e_{(jk)(ih)}$ lies in $\mc RU(\mf g)$.
We shall then consider the matrix
\begin{equation}\label{0410:eq4}
z^{-\Delta}E
:=
\sum_{(ih),(jk)\in\mc T}
z^{-\Delta(e_{(jk)(ih)})}
e_{(jk)(ih)}E_{(ih)(jk)}
\,\in\Mat_{N\times N}\mc RU(\mf g)\,.
\end{equation}
The Main Lemma, on which the proofs of Theorems \ref{thm:main1} and \ref{thm:main2} are based,
is the following:
\begin{lemma}\label{lem:main}
The quasideterminant $|\id_N+z^{-\Delta}E|_{I_1J_1}$
exists in $\Mat_{r_1\times r_1}\mc  R_\infty U(\mf g)$, and the following identity holds:
\begin{equation}\label{0229:eq3}
|\id_N+z^{-\Delta}E|_{I_1J_1}\bar 1=z^{-p_1}|z\id_N+F+\pi_{\leq\frac12}E+D|_{I_1J_1}\bar 1
\,\in\Mat_{r_1\times r_1}\mc  RM
\,.
\end{equation}
\end{lemma}
\begin{example}\label{0304:ex}
Before proving Lemma \ref{lem:main}, we show, in all detail, 
why equation \eqref{0229:eq3} holds for $N=2$.
(The reader can have fun checking it for $N=3$ and principal or minimal nilpotent $f$,
or in other examples.)
Consider the Lie algebra $\mf{gl}_2$
and its $\mf{sl}_2$-triple $\{e,h=2x,f\}$, with
$$
e=
\left(\begin{array}{ll} 0&1 \\ 0&0 \end{array}\right)=e_{12}
\,\,,\,\,\,\,
f=
\left(\begin{array}{ll} 0&0 \\ 1&0 \end{array}\right)=e_{21}
\,\,\text{ and }\,\,
x=\diag(\frac12,-\frac12)
\,.
$$
The corresponding partition is $2$, associated to the pyramid 
\framebox[3\width]{2}\framebox[3\width]{1},
hence the ``shift'' matrix \eqref{0226:eq5} is
\begin{equation}\label{0304:eq7}
D=\diag(0,-1)
\,.
\end{equation}
The $\frac12\mb Z$-grading \eqref{eq:grading} of $\mf g$ is
$\mf g=\mf g_{-1}\oplus\mf g_0\oplus\mf g_1$, where
$\mf g_{-1}=\mb Ff$,
$\mf g_0=\mb Fe_{11}\oplus\mb Fe_{22}$,
and 
$\mf g_1=\mb Fe$.
Hence, 
we have 
$$
z^{-\Delta}E=\left(\begin{array}{ll}
z^{-1}e_{11} & z^{-2}f \\
e & z^{-1}e_{22}
\end{array}\right)
\,,
$$
and the quasideterminant in the LHS of \eqref{0229:eq3} corresponds, in this example,
to the $(12)$-quasideterminant
\begin{equation}\label{0304:eq9}
|\id_2+z^{-\Delta}E|_{I_1J_1}
=
z^{-2}f
-(1+z^{-1}e_{11})e^{-1}(1+z^{-1}e_{22})
\,.
\end{equation}
Note that, since $(f|e)=1$, we have $e-1\in F_0I\subset\mc RI$.
Hence, $e\in 1+\mc RI$, and the RHS of \eqref{0304:eq9}
lies in the localized Rees algebra $\mc R_1U(\mf g)\subset\mc R_\infty U(\mf g)$,
as claimed in Lemma \ref{lem:main}.
Before applying \eqref{0304:eq9} to $\bar1\in\mc RM$,
we rewrite \eqref{0304:eq9} as a right fraction.
Since $[e,e_{22}]=e$, we have
$$
e^{-1}e_{22}
=
(e_{22}-1)e^{-1}
\,.
$$
Hence, \eqref{0304:eq9} can be rewritten as
$$
|\id_2+z^{-\Delta}E|_{I_1J_1}
=
z^{-2}f
-(1+z^{-1}e_{11})(1+z^{-1}(e_{22}-1))e^{-1}
\,.
$$
We can now apply this to the cyclic element $\bar 1$ of $\mc RM$, 
(and use the fact that $e\bar1=\bar1$), to get
\begin{equation}\label{0304:eq10}
|\id_2+z^{-\Delta}E|_{I_1J_1}\bar 1
=
z^{-2}f
-(1+z^{-1}e_{11})(1+z^{-1}(e_{22}-1))\bar1
\,.
\end{equation}
On the other hand, we have, by \eqref{0304:eq7},
$$
z\id_2+F+\pi_{\leq\frac12}E+D
=
\left(\begin{array}{cc}
e_{11}+z & f \\
1 & e_{22}+z-1
\end{array}\right)
\,.
$$
Computing its $(12)$-quasideterminant, we get
$$
|z\id_2+F+\pi_{\leq\frac12}E+D|_{I_1J_1}
=
f-(z+e_{11})(z+e_{22}-1)
\,.
$$
Hence, $z^{-2}|z\id_2+F+\pi_{\leq\frac12}E+D|_{I_1J_1}\bar1$ is equal to \eqref{0304:eq10},
as claimed by Lemma \ref{lem:main}.
\end{example}

The remainder of Section \ref{sec:main-lemma}
will be dedicated to the proof of Lemma \ref{lem:main}.

\subsection{Step 1: Existence of the quasideterminant $|\id_N+z^{-\Delta}E|_{I_1J_1}$}\label{sec:step1}

\begin{lemma}\label{0410:lem3}
The matrix $\id_N+z^{-\Delta}E$ is invertible in $\Mat_{N\times N}\mc RU(\mf g)$.
\end{lemma}
\begin{proof}
Consider the diagonal matrix (cf. \eqref{eq:xih} for notation)
\begin{equation}\label{eq:X}
X=\diag(z^{x(ih)})_{(ih)\in\mc T}
=\sum_{(ih)\in\mc T}z^{x(ih)}E_{(ih)(ih)}
\,.
\end{equation}
It is clearly an invertible element of the algebra $\Mat_{N\times N}U(\mf g)((z^{-\frac12}))$.
Moreover, recalling \eqref{0410:eq4}, we have the following identity
(in the algebra $\Mat_{N\times N}U(\mf g)((z^{-\frac12}))$):
\begin{equation}\label{0410:eq5}
\id_N+z^{-\Delta}E
=
z^{-1}X^{-1}(z\id_N+E)X
\,.
\end{equation}
On the other hand, the matrix $z\id_N+E$ is invertible,
by geometric series expansion, in $\Mat_{N\times N}U(\mf g)((z^{-\frac12}))$.
It follows by \eqref{0410:eq5} that also $\id_N+z^{-\Delta}E$
is invertible in the algebra $\Mat_{N\times N}U(\mf g)((z^{-\frac12}))$.
We need to prove that the entries of the inverse matrix $(\id_N+z^{-\Delta}E)^{-1}$
actually lie in the Rees algebra $\mc RU(\mf g)$.
The inverse matrix $(\id_N+z^{-\Delta}E)^{-1}$
is easily computed by \eqref{0410:eq5} and geometric series expansion:
\begin{equation}\label{0410:eq6}
\begin{array}{l}
\displaystyle{
\vphantom{\Big(}
(\id_N+z^{-\Delta}E)^{-1}
=
zX^{-1}(z\id_N+E)^{-1}X
=
\sum_{\ell=0}^\infty(-1)^\ell z^{-\ell}
X^{-1}E^\ell X
} \\
\displaystyle{
\vphantom{\Big(}
=
\sum_{\ell=0}^\infty
\sum_{(i_0h_0),\dots,(i_\ell h_\ell)\in\mc T}
(-1)^\ell
z^{-\ell-x(i_0h_0)+x(i_\ell h_\ell)}
} \\
\displaystyle{
\vphantom{\Big(}
\,\,\,\,\,\,\,\,\,\,\,\,\,\,\,\,\,\,\,\,\,\,\,\,\,\,\,
e_{(i_1h_1)(i_0h_0)}\dots e_{(i_\ell h_\ell)(i_{\ell-1}h_{\ell-1})}
E_{(i_0h_0)(i_\ell h_\ell)}
\,.}
\end{array}
\end{equation}
The monomial $e_{(i_1h_1)(i_0h_0)}\dots e_{(i_\ell h_\ell)(i_{\ell-1}h_{\ell-1})}$
has conformal weight
$$
\Delta=
\sum_{r=1}^\ell(1+x(i_{r-1}h_{r-1})-x(i_rh_r))=\ell+x(i_0h_0)-x(i_\ell h_\ell)
\,.
$$
Therefore, 
$z^{-\ell-x(i_0h_0)+x(i_\ell h_\ell)}
e_{(i_1h_1)(i_0h_0)}\dots e_{(i_\ell h_\ell)(i_{\ell-1}h_{\ell-1})}$
lies in the Rees algebra $\mc RU(\mf g)$.
\end{proof}
\begin{lemma}\label{0410:lem4}
Applying the homomorphism $\epsilon:\,\mc RU(\mf g)\to\mb F$ defined in \eqref{0406:eq1}
to the entries of the matrix $J_1(\id_N+z^{-\Delta}E)^{-1}I_1\in\Mat_{r_1\times r_1}\mc RU(\mf g)$,
we get
\begin{equation}\label{0410:eq7}
\epsilon\big(J_1(\id_N+z^{-\Delta}E)^{-1}I_1\big)
=
(-1)^{p_1-1}\id_{r_1}
\,.
\end{equation}
\end{lemma}
\begin{proof}
We already know, from Lemma \ref{0410:lem3},
that $J_1(\id_N+z^{-\Delta}E)^{-1}I_1$ admits an expansion
of the form $\sum_{n=-\infty}^{N}z^nA_n$, for some $N\in\mb Z$ and $A_n\in\Mat_{r_1\times r_1}F_{-n}U(\mf g)$.
Recalling the definition \eqref{eq:factor1} of the matrices $J_1$ and $I_1$,
we get from \eqref{0410:eq6} that the $(ij)$-entry ($1\leq i,j\leq r_1$) of 
the matrix $J_1(\id_N+z^{-\Delta}E)^{-1}I_1$ is
\begin{equation}\label{0410:eq8}
\begin{array}{l}
\displaystyle{
\vphantom{\Big(}
(J_1(\id_N+z^{-\Delta}E)^{-1}I_1)_{ij}
=
((\id_N+z^{-\Delta}E)^{-1})_{(ip_1)(j1)}
} \\
\displaystyle{
\vphantom{\Big(}
=
\sum_{\ell=0}^\infty
\sum_{\substack{(i_0h_0),\dots,(i_\ell h_\ell)\in\mc T\\ (i_0h_0)=(ip_1),\,(i_\ell h_\ell)=(j1)}}
(-1)^\ell
z^{-\ell-x(ip_1)+x(j1)}
e_{(i_1h_1)(i_0h_0)}\dots e_{(i_\ell h_\ell)(i_{\ell-1}h_{\ell-1})}
\,.}
\end{array}
\end{equation}
Note that, for $i,j\in\{1,\dots,r_1\}$, we have (cf. \eqref{eq:xih})
$x(j1)-x(ip_1)=p_1-1$.
Hence, the largest power of $z$ appearing in the expansion \eqref{0410:eq8} is $z^{p_1-1}$.
In order to prove \eqref{0410:eq7},
we need to compute the term with $\ell=p_1-1$ in \eqref{0410:eq8}.
It is
\begin{equation}\label{0410:eq9}
(A_0)_{ij}=
(-1)^{p_1-1}
\sum_{(i_1h_1),\dots,(i_{p_1-2} h_{p_1-2})\in\mc T}
e_{(i_1h_1)(ip_1)}e_{(i_2h_2)(i_1h_1)}\dots e_{(j1)(i_{p_1-2}h_{p_1-2})}
\,.
\end{equation}
We already know by Lemma \ref{0410:lem3} 
that the expression in the RHS of \eqref{0410:eq9} lies in $F_0U(\mf g)=\mb F\oplus F_0I$
(cf. Lemma \ref{0316:lem1}(b)).
To compute the value of $\epsilon$, we apply Lemma \ref{0410:lem6}:
\begin{equation}\label{0410:eq11}
\epsilon((A_0)_{ij})=
(-1)^{p_1-1}
\!\!\!\!\!\!\!\!\!\!\!\!\!\!\!\!\!\!
\sum_{(i_1h_1),\dots,(i_{p_1-2} h_{p_1-2})\in\mc T}
\!\!\!\!\!\!\!\!\!\!\!\!\!\!\!\!\!\!
(f|e_{(i_1h_1)(ip_1)})(f|e_{(i_2h_2)(i_1h_1)})\dots (f|e_{(j1)(i_{p_1-2}h_{p_1-2})})
\,.
\end{equation}
Recalling that $(f|e_{(jk)(ih)})=\delta_{ij}\delta_{h,k+1}$,
the only non vanishing term in the RHS of \eqref{0410:eq11}
is when $i=i_1=\dots=i_{p_1-2}=j$,
and $h_1=p_1-1$, $h_2=p_1-2$, $\dots$, $h_{p_1-2}=2$.
Hence, \eqref{0410:eq11} gives $\epsilon((A_0)_{ij})=(-1)^{p_1-1}\delta_{ij}$, 
completing the proof.
\end{proof}
\begin{lemma}\label{0410:lem5}
The matrix $J_1(\id_N+z^{-\Delta}E)^{-1}I_1\in\Mat_{r_1\times r_1}\mc RU(\mf g)$ 
is invertible in the matrix algebra $\Mat_{r_1\times r_1}\mc R_\infty U(\mf g)$.
\end{lemma}
\begin{proof}
It follows by Corollary \ref{0410:prop3} and Lemma \ref{0410:lem4}.
\end{proof}
\begin{proposition}\label{0410:prop4}
The quasideterminant 
$|\id_N+z^{-\Delta}E|_{I_1J_1}$ exists in the matrix algebra 
$\Mat_{r_1\times r_1}\mc R_\infty U(\mf g)$.
\end{proposition}
\begin{proof}
It is an immediate consequence of Lemmas \ref{0410:lem3} and \ref{0410:lem5}.
\end{proof}

\subsection{Step 2: Setup and notation}\label{step2}

The matrices $I_1$ and $J_1$ in \eqref{eq:factor1}
are as in \eqref{0304:eq4}, where $\mc I,\mc J$ are the following subsets of $\mc T$,
\begin{equation}\label{0330:eq1}
\mc I=\big\{(i1)\,\big|\,1\leq i\leq r_1\big\}
\,\,,\,\,\,\,
\mc J=\big\{(ip_1)\,\big|\,1\leq i\leq r_1\big\}
\,,
\end{equation}
of the same cardinality $|\mc I|=|\mc J|=r_1$.
The complementary subsets are
$$
\mc I^c=\mc I^c_1\sqcup\mc K
\,\,,\,\,\,\,
\mc J^c=\mc J^c_1\sqcup\mc K
\,,
$$
where
\begin{equation}\label{0330:eq2}
\mc I^c_1
=
\big\{(ih)\,\big|\,1\leq i\leq r_1,\,2\leq h\leq p_1\big\}
\,\,,\,\,\,\,
\mc J^c_1
=
\big\{(ih)\,\big|\,1\leq i\leq r_1,\,1\leq h\leq p_1-1\big\}
\,,
\end{equation}
and
\begin{equation}\label{0330:eq3}
\mc K
=
\big\{(ih)\,\big|\, r_1<i\leq r,\,1\leq h\leq p_i\big\}
\,.
\end{equation}
There is a natural bijection between 
the sets $\mc I_1^c$ and $\mc J_1^c$ (hence 
between $\mc I^c$ and $\mc J^c$),
given by
\begin{equation}\label{0330:eq4}
\mc I_1^c
\stackrel{\sim}{\longrightarrow}
\mc J_1^c
\,\,,\,\,\,\,
(ih)\mapsto(i,h-1)
\,.
\end{equation}

\subsection{Step 3: Preliminary computations}\label{step3}

We can compute the quasideterminants
in the LHS and the RHS of equation \eqref{0229:eq3} 
applying Proposition \ref{0304:prop}.
For the quasideterminant in the LHS, we have
\begin{equation}\label{0330:eq5}
\begin{array}{l}
\displaystyle{
\vphantom{\Big(}
|\id_N+z^{-\Delta}E|_{I_1J_1}
=
z^{-p_1}E_{\mc I\mc J}
} \\
\displaystyle{
\vphantom{\Big(}
-
(\id_N+z^{-\Delta}E)_{\mc I\mc J^c}((\id_N+z^{-\Delta}E)_{\mc I^c\mc J^c})^{-1}(z\id_N+z^{-\Delta}E)_{\mc I^c\mc J}
\,.}
\end{array}
\end{equation}
We used the fact that $\mc I\cap\mc J=\emptyset$, so that $(\id_N)_{\mc I\mc J}=0$,
and that $\Delta(e_{(jp_1)(i1)})=p_1$ for every $(i1)\in\mc I$ and $(jp_1)\in\mc J$,
so that $z^{-\Delta}E_{\mc I\mc J}=z^{-p_1}E_{\mc I\mc J}$.
Similarly, we compute the quasideterminants
in the RHS of equation \eqref{0229:eq3}.
Recalling the definition \eqref{eq:X} of the matrix $X$
and equation \eqref{0410:eq5}, we have
\begin{equation}\label{0413:eq6}
X^{-1}EX=z^{1-\Delta}E
\,\,,\,\,\,\,
X^{-1}FX=zF
\,\,,\,\,\,\,
X^{-1}\id_NX=\id_N
\,\,,\,\,\,\,
X^{-1}DX=D
\,,
\end{equation}
from which we get the following identity:
\begin{equation}\label{0412:eq8}
z\id_N+F+\pi_{\leq\frac12}E+D
=
zX(\id_N+F+z^{-\Delta}\pi_{\leq\frac12}E+z^{-1}D)X^{-1}
\,.
\end{equation}
Recalling the definitions \eqref{eq:factor1} of the matrices $I_1$ and $J_1$,
we have
\begin{equation}\label{0412:eq9}
J_1X=z^{-\frac{p_1-1}{2}}J_1
\,\text{ and }\,
X^{-1}I_1=z^{-\frac{p_1-1}{2}}I_1
\,.
\end{equation}
Hence, using the definition \eqref{eq:gen-quasidet} of quasideterminant
and equations \eqref{0412:eq8} and \eqref{0412:eq9},
we get
$$
\begin{array}{l}
\displaystyle{
\vphantom{\Big(}
|z\id_N+F+\pi_{\leq\frac12}E+D|_{I_1J_1}
=
(J_1(z\id_N+F+\pi_{\leq\frac12}E+D)^{-1}I_1)^{-1}
=} \\
\displaystyle{
\vphantom{\Big(}
z^{p_1}
|\id_N+F+z^{-\Delta}\pi_{\leq\frac12}E+z^{-1}D|_{I_1J_1}
\,.}
\end{array}
$$
Applying Proposition \ref{0304:prop}, we thus get
\begin{equation}\label{0330:eq6}
\begin{array}{l}
\displaystyle{
\vphantom{\Big(}
z^{-p_1}|z\id_N+F+\pi_{\leq\frac12}E+D|_{I_1J_1}
=
z^{-p_1}E_{\mc I\mc J}
-
(\id_N+z^{-\Delta}E)_{\mc I\mc J^c}
} \\
\displaystyle{
\vphantom{\Big(}
\times(\id_N+F+z^{-\Delta}\pi_{\leq\frac12}E+z^{-1}D)_{\mc I^c\mc J^c})^{-1}
(\id_N+z^{-\Delta}E+z^{-1}D)_{\mc I^c\mc J}
\,.}
\end{array}
\end{equation}
In the RHS of \eqref{0330:eq6} we used 
the facts that 
$D_{\mc I\mc T}=0$ (since $d_{(i1)}=0$ for all $(ih)\in\mc I$),
$F_{\mc I\mc T}=0$, $F_{\mc T\mc J}=0$,
and $\pi_{\leq\frac12}E_{\mc I\mc T}=E_{\mc I\mc T}$, 
$\pi_{\leq\frac12}E_{\mc T\mc J}=E_{\mc T\mc J}$.
Comparing equations \eqref{0330:eq5} and \eqref{0330:eq6},
we reduce equation \eqref{0229:eq3} to the following
equation in $\Mat_{\mc J^c\times\mc J}\mc RM$:
\begin{equation}\label{0330:eq7}
\begin{array}{l}
\vphantom{\Big(}
\displaystyle{
((\id_N+z^{-\Delta}E)_{\mc I^c\mc J^c})^{-1}(z\id_N+z^{-\Delta}E)_{\mc I^c\mc J}
\bar1
} \\
\displaystyle{
\vphantom{\Big(}
=(\id_N+F+z^{-\Delta}\pi_{\leq\frac12}E+z^{-1}D)_{\mc I^c\mc J^c})^{-1}
(\id_N+z^{-\Delta}E+z^{-1}D)_{\mc I^c\mc J}
\bar1
\,,}
\end{array}
\end{equation}
which we need to prove.
To simplify notation, we introduce the matrices
$A,B\in\Mat_{\mc I^c\times\mc J^c}\mc RU(\mf g)$,
and $v,w\in\Mat_{\mc I^c\times\mc J}\mc RU(\mf g)$,
defined as follows
\begin{equation}\label{0330:eq8}
\begin{array}{l}
\vphantom{\Big(}
\displaystyle{
A:=
(\id_N+F+z^{-\Delta}\pi_{\leq\frac12}E+z^{-1}D)_{\mc I^c\mc J^c}
} \\
\vphantom{\Big(}
\displaystyle{
\,\,\,\,\,\,
=
\!\!\!\!
\sum_{(ih)\in\mc I^c,\,(jk)\in\mc J^c}
\!\!\!\!\!
\big(z^{-\Delta}\pi_{\leq\frac12}e_{(jk)(ih)}
+\delta_{ij}\delta_{h,k+1}+(1+z^{-1}d_{(ih)})\delta_{ij}\delta_{hk}\big)
E_{(ih)(jk)}
,} \\
\vphantom{\Big(}
\displaystyle{
B:=(\id_N+z^{-\Delta}E)_{\mc I^c\mc J^c}-A
} \\
\vphantom{\Big(}
\displaystyle{
\,\,\,\,\,\,\,
=
\!\!\!\!\!
\sum_{(ih)\in\mc I^c,\,(jk)\in\mc J^c}
\!\!\!\!\!\!\!\!\!\!
z^{-\Delta}\big(\pi_{\geq1}e_{(jk)(ih)}-(f|e_{(jk)(ih)})\big)
E_{(ih)(jk)}
-z^{-1}
\!\!\!\!\!\!\!
\sum_{(ih)\in\mc I^c\cap\mc J^c}
\!\!\!\!
d_{(ih)}
E_{(ih)(ih)}
,} \\
\vphantom{\Big(}
\displaystyle{
v:=
(\id_N+z^{-\Delta}E+z^{-1}D)_{\mc I^c\mc J}
} \\
\vphantom{\Big(}
\displaystyle{
\,\,\,\,\,\,\,
=
\sum_{(ih)\in\mc I^c}
\sum_{j=1}^{r_1}
\big(z^{-\Delta}e_{(jp_1)(ih)}+(1+z^{-1}d_{(ih)})\delta_{ij}\delta_{hp_1}\big)
E_{(ih)j}
\,,} \\
\vphantom{\Big(}
\displaystyle{
w:=
(\id_N+z^{-\Delta}E)_{\mc I^c\mc J}-v
=
-z^{-1}D_{\mc I^c\mc J}
=
-z^{-1}\sum_{j=1}^{r_1}
d_{(jp_1)}
E_{(jp_1)j}
\,.} 
\end{array}
\end{equation}
Here and further, when no confusion may arise,
we denote for simplicity (cf. \eqref{eq:Delta})
$$
z^{-\Delta}e_{(jk)(ih)}
:=
z^{-\Delta(e_{(jk)(ih)})}e_{(jk)(ih)}
\,.
$$
Using the notation \eqref{0330:eq8},
equation \eqref{0330:eq7} can be rewritten as follows
\begin{equation}\label{0330:eq9}
(A+B)^{-1}(v+w)\bar1
=
A^{-1}v\bar1
\,\in\Mat_{\mc J^c\times\mc J}\mc RM
\,.
\end{equation}

\subsection{Step 4: the key computation}\label{step4}

For every $(ih),(jk)\in\mc T$, denote
\begin{equation}\label{key-X}
X_{(ih)(jk)}
=
z^{-\Delta(e_{(jk)(ih)})}(\pi_{\geq1}e_{(jk)(ih)}-(f|e_{(jk)(ih)}))
A^{-1}v\bar1
\,\in\Mat_{\mc J^c\times\mc J}\mc RM
\,.
\end{equation}
Clearly, $X_{(ih)(jk)}=0$ unless 
\begin{equation}\label{0330:eq-ass}
x{(jk)}-x{(ih)}\geq1
\,.
\end{equation}
\begin{lemma}\label{lem:key}
For every $(ih),(jk)\in\mc T$ satisfying condition \eqref{0330:eq-ass},
we have
\begin{equation}\label{key-eq}
\begin{array}{l}
\displaystyle{
\vphantom{\Big(}
X_{(ih)(jk)}
+z^{-1}\delta_{(ih)\in\mc I^c\cap\mc J^c}
\!\!\!
\sum_{\substack{(i_1h_1)\in\mc I^c \\ (x{(ih)}-x{(i_1h_1)}\leq\frac12)}}
\!\!\!
A^{-1}
E_{(i_1h_1)(ih)}
X_{(i_1h_1)(jk)}
} \\
\displaystyle{
\vphantom{\Big(}
-z^{-1}\delta_{(jk)\in\mc I^c\cap\mc J^c}
\!\!\!
\sum_{\substack{(j_1k_1)\in\mc J^c \\ (x{(j_1k_1)}-x{(jk)}\leq\frac12)}}
\!\!\!
A^{-1}
E_{(jk)(j_1k_1)}
X_{(ih)(j_1k_1)}
} \\
\displaystyle{
\vphantom{\Big(}
=
-z^{-1}\delta_{(ih)\in\mc I^c\cap\mc J^c}
E_{(jk)(ih)}
A^{-1}v\bar1
+
z^{-1}\delta_{(ih)\in\mc J}
E_{(jk)i}\bar1
} \\
\displaystyle{
\vphantom{\Big(}
+
z^{-2}(d_{(jk)}-d_{(ih)})
\delta_{(ih)\in\mc I^c\cap\mc J^c}
\delta_{(jk)\in\mc I^c\cap\mc J^c}
A^{-1}
E_{(jk)(ih)}
A^{-1}v\bar1
} \\
\displaystyle{
\vphantom{\Big(}
-
z^{-2}(d_{(jk)}-d_{(ip_1)})
\delta_{(ih)\in\mc J}
\delta_{(jk)\in\mc I^c\cap\mc J^c}
A^{-1}
E_{(jk)i}
\bar1
\,.}
\end{array}
\end{equation}
\end{lemma}
\begin{proof}
Note that, under assumption \eqref{0330:eq-ass},
we have, in particular, that $(ih)\in\mc I^c$ and $(jk)\in\mc J^c$.
Recalling definition \eqref{eq:reesM} of the module $\mc RM$
and definitions \eqref{0330:eq8} of the matrices $A$ and $v$,
we have
\begin{equation}\label{0330:eq10}
\begin{array}{l}
\displaystyle{
\vphantom{\Big(}
X_{(ih)(jk)}
=
-A^{-1}
[z^{-\Delta}e_{(jk)(ih)},A]
A^{-1}v\bar1
+
A^{-1}
[z^{-\Delta}e_{(jk)(ih)},v]\bar1
} \\
\displaystyle{
\vphantom{\Big(}
=
-
\sum_{\substack{(i_1h_1)\in\mc I^c,\,(j_1k_1)\in\mc J^c \\ (x{(j_1k_1)}-x{(i_1h_1)}\leq\frac12)}}
A^{-1}
[z^{-\Delta}e_{(jk)(ih)},z^{-\Delta}e_{(j_1k_1)(i_1h_1)}]
E_{(i_1h_1)(j_1k_1)}
A^{-1}v\bar1
} \\
\displaystyle{
\vphantom{\Big(}
+
\sum_{(i_1h_1)\in\mc I^c}
\sum_{j_1=1}^{r_1}
A^{-1}
[z^{-\Delta}e_{(jk)(ih)},z^{-\Delta}e_{(j_1p_1)(i_1h_1)}]
E_{(i_1h_1)j_1}\bar1
} \\
\displaystyle{
\vphantom{\Big(}
=
-\delta_{(ih)\in\mc I^c\cap\mc J^c}
\!\!\!
\sum_{\substack{(i_1h_1)\in\mc I^c \\ (x{(ih)}-x{(i_1h_1)}\leq\frac12)}}
\!\!\!
A^{-1}
(z^{-1-\Delta}e_{(jk)(i_1h_1)})
E_{(i_1h_1)(ih)}
A^{-1}v\bar1
} \\
\displaystyle{
\vphantom{\Big(}
+
\delta_{(jk)\in\mc I^c\cap\mc J^c}
\!\!\!
\sum_{\substack{(j_1k_1)\in\mc J^c \\ (x{(j_1k_1)}-x{(jk)}\leq\frac12)}}
\!\!\!
A^{-1}
(z^{-1-\Delta}e_{(j_1k_1)(ih)})
E_{(jk)(j_1k_1)}
A^{-1}v\bar1
} \\
\displaystyle{
\vphantom{\Big(}
+
\delta_{(ih)\in\mc J}
\sum_{(i_1h_1)\in\mc I^c}
A^{-1}
(z^{-1-\Delta}e_{(jk)(i_1h_1)})
E_{(i_1h_1)i}
\bar1
} \\
\displaystyle{
\vphantom{\Big(}
-
\delta_{(jk)\in\mc I^c\cap\mc J^c}
\sum_{j_1=1}^{r_1}
A^{-1}
(z^{-1-\Delta}e_{(j_1p_1)(ih)})
E_{(jk)j_1}
\bar1
\,.}
\end{array}
\end{equation}
We used the fact that $\Delta(a)+\Delta(b)=1+\Delta([a,b])$, for homogeneous elements $a,b\in\mf g$.
Note that, for $(ih)\in\mc I^c$ and $(jk)\in\mc J^c$, we have
\begin{equation}\label{0330:eq11}
\begin{array}{l}
\displaystyle{
\vphantom{\Big(}
z^{-\Delta}e_{(j_1k_1)(i_1h_1)}
=
A_{(i_1h_1)(j_1k_1)}
+z^{-\Delta}(\pi_{\geq1}e_{(j_1k_1)(i_1h_1)}-(f|e_{(j_1k_1)(i_1h_1)}))
} \\
\displaystyle{
\vphantom{\Big(}
-(1+z^{-1}d_{(i_1h_1)})\delta_{i_1j_1}\delta_{h_1k_1}
\,,}
\end{array}
\end{equation}
where $A_{(i_1h_1)(j_1k_1)}$ denotes the $((i_1h_1)(j_1k_1))$-entry of the matrix $A$,
while, for $1\leq j_1\leq r_1$ and $(i_1h_1)\in\mc I^c$, we have
\begin{equation}\label{0330:eq12}
z^{-\Delta}e_{(j_1p_1)(i_1h_1)}
=
v_{(i_1h_1)j_1}
-(1+z^{-1}d_{(i_1p_1)})\delta_{i_1j_1}\delta_{h_1p_1}
\,,
\end{equation}
where $v_{(i_1h_1)j_1}$ denotes the $((i_1h_1)j_1)$-entry of the matrix $v$.
Using identity \eqref{0330:eq11}, 
we can rewrite the first term in the RHS of \eqref{0330:eq10}
as
\begin{equation}\label{0330:eq13}
\begin{array}{l}
\displaystyle{
\vphantom{\Big(}
-z^{-1}\delta_{(ih)\in\mc I^c\cap\mc J^c}
E_{(jk)(ih)}
A^{-1}v\bar1
} \\
\displaystyle{
\vphantom{\Big(}
-z^{-1}\delta_{(ih)\in\mc I^c\cap\mc J^c}
\!\!\!
\sum_{\substack{(i_1h_1)\in\mc I^c \\ (x{(ih)}-x{(i_1h_1)}\leq\frac12)}}
\!\!\!
A^{-1}
E_{(i_1h_1)(ih)}
X_{(i_1h_1)(jk)}
} \\
\displaystyle{
\vphantom{\Big(}
+
z^{-1}(1+z^{-1}d_{(jk)})
\delta_{(ih)\in\mc I^c\cap\mc J^c}
\delta_{(jk)\in\mc I^c\cap\mc J^c}
A^{-1}
E_{(jk)(ih)}
A^{-1}v\bar1
\,.}
\end{array}
\end{equation}
For the first term of \eqref{0330:eq13} we used the fact that $A_{(i_1h_1)(jk)}=0$
unless $x{(jk)}-x{(i_1h_1)}\leq1$,
which, combined to the assumption \eqref{0330:eq-ass},
implies $x{(ih)}-x{(i_1h_1)}\leq\frac12$.
Hence, 
$\sum_{\substack{(i_1h_1)\in\mc I^c \\ (x{(ih)}-x{(i_1h_1)}\leq\frac12)}}
A_{(i_1h_1)(jk)}E_{(i_1h_1)(ih)}=AE_{(jk)(ih)}$.
For the second term of \eqref{0330:eq13} we used
the definition \eqref{key-X} of $X_{(ih)(jk)}$.
With similar computations, we can use \eqref{0330:eq11}
to rewrite the second term in the RHS of \eqref{0330:eq10}
as
\begin{equation}\label{0330:eq14}
\begin{array}{l}
\displaystyle{
\vphantom{\Big(}
+z^{-1}\delta_{(jk)\in\mc I^c\cap\mc J^c}
A^{-1}
E_{(jk)(ih)}v\bar1
} \\
\displaystyle{
\vphantom{\Big(}
+z^{-1}\delta_{(jk)\in\mc I^c\cap\mc J^c}
\!\!\!
\sum_{\substack{(j_1k_1)\in\mc J^c \\ (x{(j_1k_1)}-x{(jk)}\leq\frac12)}}
\!\!\!
A^{-1}
E_{(jk)(j_1k_1)}
X_{(ih)(j_1k_1)}
} \\
\displaystyle{
\vphantom{\Big(}
-z^{-1}(1+z^{-1}d_{(ih)})
\delta_{(ih)\in\mc I^c\cap\mc J^c}
\delta_{(jk)\in\mc I^c\cap\mc J^c}
A^{-1}
E_{(jk)(ih)}
A^{-1}v\bar1
\,,}
\end{array}
\end{equation}
and the third term in the RHS of \eqref{0330:eq10}
as
\begin{equation}\label{0330:eq15}
+
z^{-1}\delta_{(ih)\in\mc J}
E_{(jk)i}\bar1
-
z^{-1}(1+z^{-1}d_{(jk)})
\delta_{(ih)\in\mc J}
\delta_{(jk)\in\mc I^c\cap\mc J^c}
A^{-1}
E_{(jk)i}
\bar1
\,,
\end{equation}
while, using \eqref{0330:eq12}, we can rewrite the fourth term in the RHS of \eqref{0330:eq10}
as
\begin{equation}\label{0330:eq16}
-z^{-1}
\delta_{(jk)\in\mc I^c\cap\mc J^c}
A^{-1}
E_{(jk)(ih)}v\bar1
+
z^{-1}(1+z^{-1}d_{(ip_1)})
\delta_{(ih)\in\mc J}
\delta_{(jk)\in\mc I^c\cap\mc J^c}
A^{-1}
E_{(jk)i}
\bar1
\,.
\end{equation}
Combining equation \eqref{0330:eq10}
with \eqref{0330:eq13}, \eqref{0330:eq14}, \eqref{0330:eq15} and \eqref{0330:eq16},
we get equation \eqref{key-eq}.
\end{proof}
\begin{lemma}\label{lem:key2}
The unique solution of equation \eqref{key-eq} is:
\begin{equation}\label{key-sol}
X_{(ih)(jk)}
=
-z^{-1}\delta_{(ih)\in\mc I^c\cap\mc J^c}E_{(jk)(ih)}A^{-1}v\bar1
+z^{-1}\delta_{(ih)\in\mc J}E_{(jk)i}\bar1
\,,
\end{equation}
if $x{(jk)}-x{(ih)}\geq1$, and $X_{(ih)(jk)}=0$ otherwise.
\end{lemma}
\begin{proof}
Let us prove that \eqref{key-sol}
solves equation \eqref{key-eq}.
Note that the first term in the LHS of \eqref{key-eq} equals, by \eqref{key-sol},
the sum of the first two terms in the RHS of \eqref{key-eq}.
We next claim that the second term in the LHS of \eqref{key-eq} vanishes.
Indeed, by \eqref{key-sol},
$E_{(i_1h_1)(ih)}X_{(i_1h_1)(jk)}$ 
vanishes unless $(ih)=(jk)$.
But, in this case, the inequalities 
$x{(ih)}-x{(i_1h_1)}\leq\frac12$ and $x{(jk)}-x{(i_1h_1)}\geq1$
are incompatible.
Therefore, equation \eqref{key-eq} reduces to
\begin{equation}\label{0331:eq1}
\begin{array}{l}
\displaystyle{
\vphantom{\Big(}
-z
\!\!\!
\sum_{\substack{(j_1k_1)\in\mc J^c \\ (x{(j_1k_1)}-x{(jk)}\leq\frac12)}}
\!\!\!
E_{(jk)(j_1k_1)}
X_{(ih)(j_1k_1)}
} \\
\displaystyle{
\vphantom{\Big(}
=
\delta_{(ih)\in\mc I^c\cap\mc J^c}
(d_{(jk)}-d_{(ih)})
E_{(jk)(ih)}
A^{-1}v\bar1
-
\delta_{(ih)\in\mc J}
(d_{(jk)}-d_{(ip_1)})
E_{(jk)i}
\bar1\,,}
\end{array}
\end{equation}
for every $(jk)\in\mc I^c\cap\mc J^c$.
Equation \eqref{0331:eq1} is satisfied by \eqref{key-sol} since,
by the definition \eqref{0226:eq6} of the shift matrix $D$, we have
$$
d_{(jk)}-d_{(ih)}
=\Big\{
(j_1k_1)\in\mc T\,\Big|\,x{(j_1k_1)}-x{(ih)}\geq1\,,\,\,x{(j_1k_1)}-x{(jk)}\leq\frac12\Big\}
\,.
$$
The uniqueness of the solution of equation \eqref{key-eq} is clear. 
Indeed, equation \eqref{key-eq} has the matrix form
$(\id+z^{-1}M)X=Y$,
where $X$ is the column vector $(X_{(ih)(jk)})$, indexed by the pairs $((ih),(jk))$,
with entries in the vector space $V=\Mat_{\mc J^c\times\mc J}\mc RM$,
$Y$ is the analogous column vector defined by the RHS of \eqref{key-eq},
and $M$ is some matrix with entries in $\Mat_{\mc J^c\times\mc I^c}\mc RU(\mf g)$,
which is an algebra acting on the vector space $V$.
But then the matrix $\id+z^{-1}M$ can be inverted by geometric series expansion.
\end{proof}
\begin{corollary}\label{lem:key3}
We have
\begin{equation}\label{key-cor}
B A^{-1}v\bar1=w\bar1
\,.
\end{equation}
\end{corollary}
\begin{proof}
By the definitions \eqref{0330:eq8} of $B$ and \eqref{key-X} of $X_{(ih)(jk)}$,
we have
\begin{equation}\label{0331:eq2}
B A^{-1}v\bar1
=
\sum_{(ih)\in\mc I^c,\,(jk)\in\mc J^c}
E_{(ih)(jk)}
X_{(ih)(jk)}
-
z^{-1}\sum_{(ih)\in\mc I^c\cap\mc J^c}
d_{(ih)}
E_{(ih)(ih)}
A^{-1}v\bar1
\,.
\end{equation}
We can use equation \eqref{key-sol}
to rewrite the first term in the RHS of \eqref{0331:eq2} as
\begin{equation}\label{0331:eq3}
\begin{array}{l}
\displaystyle{
\vphantom{\Big(}
-z^{-1}\sum_{(ih)\in\mc I^c\cap\mc J^c}
\#\big\{(jk)\in\mc T\,\big|\,x{(jk)}-x{(ih)}\geq1\big\}
E_{(ih)(ih)}
A^{-1}v\bar1
} \\
\displaystyle{
\vphantom{\Big(}
+z^{-1}\sum_{i=1}^{r_1}
\#\big\{(jk)\in\mc T\,\big|\,x{(jk)}-x{(ih)}\geq1\big\}
E_{(jk)i}\bar1
\,.}
\end{array}
\end{equation}
Recalling the definition \eqref{0226:eq6} of the shift $d_{(ih)}$,
we see that the first term in \eqref{0331:eq3}
is opposite to the second term in the RHS of \eqref{0331:eq2},
while the first term in \eqref{0331:eq3} coincides with $w\bar1$ (cf. \eqref{0330:eq8}).
\end{proof}

\subsection{Step 5: proof of Equation \eqref{0330:eq9}}\label{step5}

First, consider the matrix $B\in\Mat_{\mc I^c\times\mc J^c}$
defined in \eqref{0330:eq8}.
Its entries are all in the kernel of the homomorphism $\epsilon:\,\mc RU(\mf g)\to\mb F$,
since $z^{\Delta}(e_{(jk)(ih)}-(f|e_{(jk)(ih)}))\in\mc RI\subset\ker\epsilon$,
and $z^{-1}d_{ih}\in z^{-1}\mc RU(\mf g)\subset\ker\epsilon$.
Therefore, $\epsilon(\id+BA^{-1})=\id$,
and Corollary \ref{0410:prop3} guarantees that 
the matrix $\id+BA^{-1}$ is invertible in $\Mat_{\mc I^c\times\mc J^c}\mc R_\infty U(\mf g)$.
By Corollary \ref{lem:key3}, we have
$$
(\id+B A^{-1})v\bar1
=(v+w)\bar1
\,.
$$
It follows that
$$
\begin{array}{l}
\displaystyle{
\vphantom{\Big(}
A^{-1}v\bar1
=
A^{-1}(\id+B A^{-1})^{-1}
(\id+B A^{-1})v\bar1
=
(A+B)^{-1}(v+w)\bar1
\,.}
\end{array}
$$
\begin{flushright}
\qedsymbol
\end{flushright}

\section{Proof of the Main Theorems}

\subsection{Proof of Theorem \ref{thm:main1}}
\label{sec:4.6}

The proof is similar to the proof of the analogous result 
for classical affine $W$-algebras, presented in \cite[Appendix]{DSKV16b}.
We need to prove that, for every $i_0,j_0\in\{1,\dots,r_1\}$, we have
$L_{i_0j_0}(w)\in W((w^{-1}))$.
By the definition \eqref{eq:La} of the matrix $L(w)$
and by the Main Lemma \ref{lem:main},
this is the same as to prove that
\begin{equation}\label{0413:eq5}
\widetilde{L}^\prime_{i_0j_0}(w)\bar1
\,\in\mc RW(\mf g,f)\,,
\,\text{ where }\,
\widetilde{L}^\prime(w)
:=
|\id_N+w^{-\Delta}E|_{I_1J_1}\,\in\Mat_{r_1\times r_1}\mc R_\infty U(\mf g)
\,.
\end{equation}
(Here, the Rees algebra is considered with respect to the variable $w$ in place of $z$.)
By Proposition \eqref{0413:prop1},
condition \eqref{0413:eq5} is equivalent to
\begin{equation}\label{0303:eq23}
[a,\widetilde{L}^\prime_{i_0j_0}(w)]\bar1=0
\,\,\text{ for all }\,\, a\in\mf g_{\geq\frac12}\,,
\end{equation}
which we are going to prove.
By the definition \eqref{eq:gen-quasidet}
of the quasideterminant,
$\widetilde L^\prime(w)=(J_1(\id_N+w^{-\Delta}E)^{-1}I_1)^{-1}$.
Hence, we can use formula \eqref{20160402:eq1}
to rewrite \eqref{0303:eq23} as
\begin{equation}\label{0303:eq24}
-\sum_{i,j=1}^{r_1}
\widetilde{L}_{i_0i}(w)
\big[a,(J_1(\id_N+w^{-\Delta}E)^{-1}I_1)_{ij}\big]
\widetilde{L}_{jj_0}(w)
\bar1
\,.
\end{equation}
We need to prove that \eqref{0303:eq24} vanishes for every $a\in\mf g_{\geq\frac12}$.
In order to do so, we shall prove that
\begin{equation}\label{0303:eq25}
\big[a,(J_1(\id_N+w^{-\Delta}E)^{-1}I_1)_{ij}\big]
=0
\,\,\text{ for all }\,\, a\in\mf g_{\geq\frac12}
\,.
\end{equation}
First, by \eqref{0413:eq6}, we have the identity
$$
J_1(\id_N+w^{-\Delta}E)^{-1}I_1
=w^{p_1}J_1(w\id_N+E)^{-1}I_1
\,.
$$
Hence, we use the definition \eqref{eq:factor1} of the matrices $I_1$ and $J_1$
to rewrite equation \eqref{0303:eq25} as
\begin{equation}\label{0303:eq26}
\big[a,((w\id_N+E)^{-1})_{(ip_1)(j1)}\big]
=0
\,,
\end{equation}
where now the identity takes place in $U(\mf g)((w^{-1}))$
(not anymore in the Rees algebra $\mc RU(\mf g)$.)
We now let $a$ be the basis element $e_{(\tilde j\tilde k)(\tilde i\tilde h)}$,
where $(\tilde i\tilde h),(\tilde j\tilde k)\in\mc T$
satisfy (cf. \eqref{eq:adx})
\begin{equation}\label{0303:eq26b}
\frac12(p_{\tilde j}-p_{\tilde i})-(\tilde k-\tilde h)\geq\frac12
\,.
\end{equation}
Recalling the definition \eqref{eq:E} of the matrix $E$, we have
$$
a:=e_{(\tilde j\tilde k)(\tilde i\tilde h)}
=
(z\id_N+E)_{(\tilde i\tilde h)(\tilde j\tilde k)}
\,.
$$
Hence, \eqref{0303:eq26} becomes
\begin{equation}\label{0303:eq27}
\big[
(z\id_N+E)_{(\tilde i\tilde h)(\tilde j\tilde k)}
,((w\id_N+E)^{-1})_{(ip_1)(j1)}
\big]
\,.
\end{equation}
Recalling, by Example \ref{ex:A}, that the matrix $z\id_N+E$ 
is an operator of Yangian type and it is invertible in 
$\Mat_{N\times N}U(\mf g)((z^{-1}))$,
we can use Lemma \ref{0303:lem5} to rewrite \eqref{0303:eq27} as
\begin{equation}\label{0303:eq28}
\begin{array}{l}
\vphantom{\Big(}
\displaystyle{
(z-w)^{-1}\sum_{\tau\in\mc T}\Big(
-\delta_{(ip_1)(\tilde j\tilde k)}
(z\id_N+E)_{(\tilde i\tilde h)\tau}((w\id_N+E)^{-1})_{\tau(j1)}
} \\
\vphantom{\Big(}
\displaystyle{
+\delta_{(\tilde i\tilde h)(j1)}
((w\id_N+E)^{-1})_{(ip_1)\tau}(z\id_N+E)_{\tau(\tilde j\tilde k)}
\Big)
\,.}
\end{array}
\end{equation}
To prove that \eqref{0303:eq28} vanishes,
we just observe that 
$\delta_{(ip_1)(\tilde j\tilde k)}=\delta_{(\tilde i\tilde h)(j1)}=0$.
Indeed, both equalities $(\tilde j\tilde k)=(ip_1)$
and $(\tilde i\tilde h)=(j1)$ contradict \eqref{0303:eq26b}.
\begin{flushright}
\qedsymbol
\end{flushright}

\subsection{Proof of Theorem \ref{thm:main2}}
\label{sec:4.7}

By the Main Lemma \ref{lem:main} and Theorem \ref{thm:main1}, we have
$$
L_{ij}(z)=z^{p_1}\widetilde{L}^\prime_{ij}(z)\bar1
\,\in W(\mf g,f)((z^{-1}))
$$
and, by Proposition \ref{0413:prop1},
$$
\widetilde{L}^\prime(z)
:=
|\id_N+z^{-\Delta}E|_{I_1J_1}
\,\in\Mat_{r_1\times r_1}\mc R_\infty\widetilde{W}
\,.
$$
Recall that the associative product of the $W$-algebra $W(\mf g,f)$
is induced by the product of $\widetilde{W}\subset U(\mf g)$.
Hence,
\begin{equation}\label{0303:eq14}
[L_{ij}(z),L_{hk}(w)]
=
z^{p_1}w^{p_1}\big[
(|\id_N+z^{-\Delta}E|_{I_1J_1})_{ij},
(|\id_N+w^{-\Delta}E|_{I_1J_1})_{hk}
\big]\bar1
\,.
\end{equation}
Recall, by Example \ref{ex:A},
that $z\id_N+E$ is an operator of Yangian type for the algebra $U(\mf g)$.
Hence, using the first equation in \eqref{0413:eq6},
we have 
$$
\id_N+z^{-\Delta}E
=
z^{-1}X^{-1}(z\id_N+E)X
\,,
$$
which, by Proposition \ref{prop:properties-adler}(a),
satisfies the Yangian identity \eqref{eq:adler} as well.
It then follows by Theorem \ref{thm:quasidet-adler}
that the quasideterminant $|\id_N+z^{-\Delta}E|_{I_1J_1}$
satisfies the Yangian identity \eqref{eq:adler} as well.
(Strictly speaking, Theorem \ref{thm:quasidet-adler}
cannot be applied, as stated, in this situation,
since the quasideterminant $|\id_N+z^{-\Delta}E|_{I_1J_1}$
does not have coefficients in $U(\mf g)((z^{-\frac12}))$,
but in $\mc R_\infty U(\mf g)$.
On the other hand, it is easy to check that the same proof of Theorem \ref{thm:quasidet-adler}
can be repeated in this situation to show that the matrix $|\id_N+z^{-\Delta}E|_{I_1J_1}$
satisfies the Yangian identity \eqref{eq:adler}.)
We thus get, from equation \eqref{0303:eq14},
\begin{equation}\label{0303:eq15}
\begin{array}{l}
\vphantom{\Big(}
\displaystyle{
(z-w)[L_{ij}(z),L_{hk}(w)]
= 
z^{p_1}w^{p_1}
\Big(
(|\id_N+w^{-\Delta}E|_{I_1J_1})_{hj}(|\id_N+z^{-\Delta}E|_{I_1J_1})_{ik}
} \\
\vphantom{\Big(}
\displaystyle{
- (|\id_N+z^{-\Delta}E|_{I_1J_1})_{hj}(|\id_N+w^{-\Delta}E|_{I_1J_1})_{ik}
\Big)\bar1
} \\
\vphantom{\Big(}
\displaystyle{
=
L_{hj}(w)\circ L_{ik}(z)
-L_{hj}(z)\circ L_{ik}(w)
\,.}
\end{array}
\end{equation}
\begin{flushright}
\qedsymbol
\end{flushright}

\section{Conjectural form of \texorpdfstring{$L(z)$}{L(z)}}
\label{sec:conjecture}

\subsection{PBW basis for $W(\mf g,f)$}
We recall the construction of a PBW basis for $W(\mf g,f)$ provided in \cite{Pre02}.
Fix a subspace $U\subset\mf g$ complementary to $[f,\mf g]$ and compatible with the grading \eqref{eq:grading}.
For example, we could take $U=\mf g^e$, the Slodowy slice. 
Since $\ad f:\,\mf g_{j}\to\mf g_{j-1}$ is surjective for $j\leq\frac12$, 
we have $\mf g_{\leq-\frac12}\subset[f,\mf g]$.
In particular, we have the direct sum decomposition
\begin{equation}\label{eq:U1}
\mf g_{\geq-\frac12}=[f,\mf g_{\geq\frac12}]\oplus U\,.
\end{equation}
Note that, by the non-degeneracy of $(\cdot\,|\,\cdot)$, the orthocomplement to $[f,\mf g]$
is $\mf g^f$, the centralizer of $f$ in $\mf g$.
Hence, the direct sum decomposition dual to \eqref{eq:U1}
has the form
\begin{equation}\label{eq:Uperp}
\mf g_{\leq\frac12}=U^\perp\oplus\mf g^f\,.
\end{equation}
As a consequence of \eqref{eq:Uperp}
we have the decomposition in a direct sum of subspaces
$$
S(\mf g_{\leq\frac12})=S(\mf g^f)\oplus S(\mf g_{\leq\frac12})U^\perp
\,,
$$
where $S(V)$ denotes the symmetric algebra over the vector space $V$.
Let $\pi_{\mf g^f}:S(\mf g_{\leq\frac12})\twoheadrightarrow S(\mf g^f)$
be the canonical quotient map, with kernel $S(\mf g_{\leq\frac12}) U^\perp$.

The next theorem gives a description of a PBW basis for $W(\mf g,f)$.
\begin{theorem}\cite[Theorem 4.6]{Pre02}
There exists a (non-unique) linear map
$$
w:\mf g^f\to W(\mf g,f)
$$
such that $w(x)\in F_\Delta W(\mf g,f)$ and $\pi_{\mf g^f}(\gr_\Delta w(x))=x$, for every $x\in\mf g_{1-\Delta}^f$.
Moreover, if $\{x_i\}_{i=1}^{t}$, $t=\dim\mf g^f$, is an ordered basis of $\mf g^f$ consisting of $\ad x$-eigenvectors
$x_i\in\mf g_{1-\Delta_i}^f$, then the
monomials
$$
\{w(x_{i_1})\dots w(x_{i_k})\mid k\in\mb Z_+\,, 1\leq i_1\leq\dots\leq i_k\leq t\,,
\Delta_1+\dots+\Delta_k\leq\Delta
\}
$$
form a basis of $F_\Delta W(\mf g,f)$, $\Delta\geq0$.
\end{theorem}

\subsection{Conjecture on the map $w$}

Let $\mf g=\mf{gl}_N$, let $f\in\mf g$ be a non-zero nilpotent element,
and let $p=p_1\geq\dots\geq p_r>0$ be the corresponding partition of $N$.
Recall the notation of Section \ref{sec:4.1}.
It is proved in \cite{DSKV16b} that $\mf g=[f,\mf g]\oplus U$, where
\begin{equation}\label{eq:U}
U=\Span\Big\{e_{(j1),(i,p_i-k)}
\,\,,\,\,\text{ where }\,\,
1\leq i,j\leq r
\,\,\text{ and }\,\,
0\leq k\leq\min\{p_i,p_j\}-1
\Big\}
\,,
\end{equation}
and the basis of $\mf g^f$ dual to the basis of $U$ in \eqref{eq:U} is:
\begin{equation}\label{eq:basisgf}
f_{ij;k}
:=
\sum_{h=0}^k
e_{(i,p_i+h-k),(j,h+1)}
\,\,,\,\,\,\,
1\leq i,j\leq r\,,\,\, 0\leq k\leq\min\{p_i,p_j\}-1
\,.
\end{equation}
In the sequel we give a conjectural way to define the corresponding free generators of $W(\mf g,f)$,
\begin{equation}\label{eq:basisW}
w_{ij;k}
:=
w(f_{ij;k})
\,\,,\,\,\,\,
1\leq i,j\leq r\,,\,\, 0\leq k\leq\min\{p_i,p_j\}-1
\,,
\end{equation}
satisfying the conditions of Premet's Theorem, following the analogous result in the classical affine setting, see \cite{DSKV16b}. For $i,j\in\{1,\dots,r\}$, let
\begin{equation}\label{eq:Wij}
W_{ij}(z)
=
\sum_{k=0}^{\min\{p_i,p_j\}-1} w_{ij;k}(-z)^k
\in W(\mf g,f)[z]
\,.
\end{equation}
Denote by $W(z)$ the $r\times r$ matrix with entries \eqref{eq:Wij}
(transposed):
$$
W(z)
=
\sum_{i,j=1}^r 
W_{ij}(z) E_{ji}
=
\left(\begin{array}{cccc}
W_{11}(z) & W_{21}(z) & \dots & W_{r1}(z) \\
W_{12}(z) &W_{22}(z) & \dots &W_{r2}(z) \\
\vdots & \vdots & \ddots & \vdots \\
W_{1r}(z) & W_{2r}(z) & \dots & W_{rr}(z)
\end{array}\right)
\,,
$$
and by $(-z)^p$ the diagonal $r\times r$ matrix
with diagonal entries $(-z)^{p_i}$, $i=1,\dots,r$:
$$
(-z)^p
=
\sum_{i=1}^r (-z)^{p_i}E_{ii}
=
\left(\begin{array}{ccc}
(-z)^{p_1} & & 0 \\
& \ddots& \\
0 & & (-z)^{p_r}
\end{array}\right)
\,.
$$
\begin{conjecture}\label{conj1}
There exists a unique set of generators $w_{ij;k}$, $1\leq i,j\leq r$, and $0\leq k\leq\min\{p_i,p_j\}-1$, 
of $W(\mf g,f)$,
for which the following identity holds
\begin{equation}\label{eq:L-explicit}
L(z)
:=
|z\id_N+F+\pi_{\leq\frac12}+D|_{I_1J_1}\bar1
=
|-(-z)^p+W(z)|_{I_{rr_1}J_{r_1r}}
\,,
\end{equation}
where $I_{rr_1}\in\Mat_{r\times r_1}\mb F$ and $J_{r_1r}\in\Mat_{r_1\times r}\mb F$
are as in \eqref{0304:eq4} corresponding to the subsets $\mc I=\mc J=\{1,\dots,r_1\}$.
In this case, 
the linear map $w:\mf g^f\to W(\mf g,f)$ defined by \eqref{eq:basisW} 
satisfies all the conditions of Premet's Theorem.
\end{conjecture}
In Section \ref{sec:Examples}
we check Conjecture \ref{conj1}  in some explicit examples.

We can write the matrix $W(z)$ in block form as
$W(z)
=
\left(\begin{array}{ll}
W_1(z) & W_2(z) \\
W_3(z) & W_4(z)
\end{array}\right)$,
where
\begin{align*}
\begin{split}
& W_1(z)
=
\big(
W_{ji}(z)
\big)_{1\leq i,j\leq r_1}
\,\,,\,\,\,\,
W_2(z)
=
\big(
W_{ji}(z)
\big)_{1\leq i\leq r_1<j\leq r}
\,,\\
& W_3(z)
=
\big(
W_{ji}(z)
\big)_{1\leq j\leq r_1<i\leq r}
\,\,,\,\,\,\,
W_4(z)
=
\big(
W_{ji}(z)
\big)_{r_1<i,j\leq r}
\,.
\end{split}
\end{align*}
Then, by Proposition \ref{0304:prop},
we can rewrite equation \eqref{eq:L-explicit}
as follows:
$$
L(z)
=
-\id_{r_1}(-z)^{p_1}+W_1(z)
-W_2(z)(-(-z)^{q}+W_4(z))^{-1}W_3(z)
\,,
$$
where $q=(p_{r_1+1}\geq\dots\geq p_r>0)$ is the partition of $N-r_1p_1$,
obtained by removing from the partition $p$ all the maximal parts.

\section{Examples}\label{sec:Examples}

\subsection{Example 1: principal nilpotent $f$ and the Yangian $Y_{1,N}$}\label{sec:Examples.1}

The principal nilpotent element 
$f^{\text{pr}}=\sum_{i=1}^{N-1}e_{i+1,i}\in\mf{gl}_N$
corresponds to the partition $p=N$.
For this partition $r=r_1=1$ and the shift matrix $D$ given by equation \eqref{0226:eq5} 
is $D=\sum_{i=1}^N(1-i)E_{ii}$.
In this case, by Proposition \ref{thm:rho},
we can identify $W(\mf g,f^{\text{pr}})$ with a subalgebra of $U(\mf g_{\leq0})$ 
(since the grading \eqref{eq:grading} is by integers).
Using Proposition \ref{0304:prop} and equation \eqref{eq:La} we have that $L(z)\in U(\mf g_{\leq0})((z^{-1}))$ is defined by
\begin{equation}\label{eq:principal-quasidet}
\begin{array}{l}
\displaystyle{
\vphantom{\Big(}
L(z)
=|z\id_N+\rho(E)+D|_{1N}
=
e_{N1}-
\left(\begin{array}{llll}
z+e_{11}&e_{21}&\dots&e_{N-1,1}
\end{array}\right) 
}\\
\displaystyle{
\times\left(\begin{array}{cccll}
1&z\!+\!e_{22}\!-\!1&e_{32}&\dots&e_{N-1,2} \\
0&1&\ddots&\ddots&\vdots \\
\vdots&\ddots&\ddots&\ddots&e_{N-1,N-2} \\
\vdots&&\ddots&\ddots&z\!+\!e_{N\!-\!1,N\!-\!1}\!+\!2\!-\!N \\
0&\dots&\dots&0&1
\end{array}\right)^{-1}
\left(\begin{array}{l}
e_{N2} \\ \vdots \\ e_{N,N-1} \\ z\!+\!e_{N\!N}\!+\!1\!-\!N
\end{array}\right)
\,.}
\end{array}
\end{equation}
We can expand the inverse matrix in the RHS of \eqref{eq:principal-quasidet}
in geometric series, to get the following more explicit formula for $L(z)$ (we use the shorthand
$\tilde e_{ij}=e_{ij}+\delta_{ij}(1-i)$):
\begin{equation}\label{eq:principal-gener-explicit}
\begin{split}
&L(z) 
=
e_{N1}+
\sum_{s=1}^{N-1}(-1)^s 
\sum_{2\leq h_1<\dots<h_s\leq N}
(\delta_{h_1-1,1}z+\tilde e_{h_1-1,1})
(\delta_{h_2-1,h_1}z+\tilde e_{h_2-1,h_1})
\dots \\
& \qquad\qquad\dots
(\delta_{h_s-1,h_{s-1}}z+\tilde e_{h_s-1,h_{s-1}})
(\delta_{N,h_s}z+\tilde e_{N,h_s})
\,.
\end{split}
\end{equation}
Note that the RHS of \eqref{eq:principal-gener-explicit} is a polynomial in $z$, hence (recall Proposition \ref{thm:L1}(b))
it uniquely defines elements $w_k\in W(\mf{gl}_N,f^\text{pr})\subset U(\mf g_{\leq0})$ such that
$$
L(z)
=
-(-z)^N
+\sum_{k=0}^{N-1}w_k(-z)^k
\,.
$$
For example, $w_{N-1}=\tr E=e_{11}+\dots+e_{NN}-\frac{N(N-1)}{2}$.
Hence in this case, 
Conjecture \ref{conj1} holds.
This solution agrees with the explicit generators obtained in \cite{BK06}.
By Theorem \ref{thm:main2} we have that $[L(z),L(w)]=0$, which implies that
all the $w_k$'s  commute.
Thus, $W(\mf{gl}_N,f^\text{pr})=\mb F[w_1,\dots,w_k]$. 
This shows, in particular, that $W(\mf{gl}_N,f^{\text{pr}})$ is isomorphic to the Yangian $Y_{1,N}$ (the Yangian
of $\mf{gl}_1$ of level $N$) introduced by Cherednik
\cite{Che87} as already noticed in \cite{RS99,BK06}.

\subsection{Example 2: rectangular nilpotent and the Yangian $Y_{r_1,p_1}$
}\label{sec:Examples.2}

Consider the partition $p=(p_1,\dots,p_1)$ of $N$,
consisting of $r_1$ equal parts of size $p_1$.
It corresponds to the so called rectangular nilpotent element $f^\text{rect}$.
In this case we have the shift matrix $D=\sum_{(ih)\in\mc T}r_1(1-h)E_{(ih)(ih)}$. 

We identify 
$$
\Mat_{N\times N}\mb F
\simeq
\Mat_{p_1\times p_1}\mb F\otimes\Mat_{r_1\times r_1}\mb F
\,,
$$
by mapping $E_{(ih),(jk)}\mapsto E_{hk}\otimes E_{ij}$.
Under this identification, we have
\begin{equation*}
\begin{split}
& \id_N\mapsto\id_{p_1}\otimes\id_{r_1}
\,\,,\,\,\,\,
D\mapsto \sum_{h=1}^{p_1}r_1(1-h)E_{hh}\otimes \id_{r_1}
\,,\\
&
\rho(E)\mapsto \sum_{k=1}^{p_1-1}E_{k+1,k}\otimes \id_{r_1}
+
\sum_{i,j=1}^{r_1}\sum_{1\leq h\leq k\leq p_1}e_{(jk),(ih)}E_{hk}\otimes E_{ij}
\,.
\end{split}
\end{equation*}
As for the principal nilpotent case,
we can identify $W(\mf g,f^{\text{rect}})$ with a subalgebra of $U(\mf g_{\leq0})$.
By the formula for quasideterminants \eqref{0304:eq6} given in Proposition \ref{0304:prop} and equation \eqref{eq:La}, we have
that $L(z)\in\Mat_{r_1\times r_1} U(\mf g_{\leq0})((z^{-1}))$ is defined by
(we use the shorthand $\tilde e_{(ih),(jk)}=e_{(ih),(jk)}+\delta_{(ih),(jk)}r_1(1-h)$)
\begin{equation}\label{eq:rectangular-quasidet}
\begin{split}
& L(z)
=
\Big|
(\id_{p_1}\otimes\id_{r_1})z
+\sum_{k=1}^{p_1-1}E_{k+1,k}\otimes\id_{r_1}
+\sum_{i,j=1}^{r_1}\sum_{1\leq h\leq k\leq p_1}\tilde e_{(jk),(ih)}E_{hk}\otimes E_{ij}
\Big|_{I_1J_1} \\
& =\!
\sum_{i,j=1}^{r_1}
e_{(jp_1),(i1)}E_{ij}
-
\Bigg(\!
\sum_{i,j=1}^{r_1}
\left(\begin{array}{llll}
\!\!\delta_{ij}z\!+\!e_{(j1),(i1)}&e_{(j2),(i1)}&\dots&e_{(j,p_1-1),(i1)}\!\!
\end{array}\right)\otimes E_{ij}
\!\Bigg)
\\
&
\times\Bigg(\!
\id_{p_1-1}\!\otimes\!\id_{r_1}+
\sum_{i,j=1}^{r_1}
\left(\begin{array}{rccl}
0&\delta_{ij}z\!\!+\!\!\tilde e_{(j2),(i2)}&\dots&e_{(j,p_1-1),(i2)} \\
\vdots&\ddots&\ddots&\vdots \\
\vdots&&\ddots&\delta_{ij}z\!\!+\!\!\tilde e_{(j,p_1-1),(i,p_1-1)} \\
0&\dots&\dots&0
\end{array}\right)\!\otimes\! E_{ij}
\!\Bigg)^{-1} \\
&
\times\Bigg(
\sum_{i,j=1}^{r_1}
\left(\begin{array}{l}
e_{(jp_1),(i2)} \\ \vdots \\ e_{(jp_1),(i,p_1-1)} \\ \delta_{ij}z+\tilde e_{(jp_1),(ip_1)}
\end{array}\right)
\otimes E_{ij}
\Bigg)
\,.
\end{split}
\end{equation}
As we did in Section \ref{sec:Examples.1},
we expand the inverse matrix in the RHS of \eqref{eq:rectangular-quasidet}
in geometric series, to get a more explicit formula for $L(z)$.
For every $1\leq i,j\leq r_1$, we have
\begin{equation}\label{eq:rectangular-gener-explicit}
\begin{split}
&L_{ij}(z)
=
e_{(jp_1),(i1)}+
\sum_{s=1}^{p_1-1}(-1)^s
\sum_{i_1,\dots,i_s=1}^{r_1}
\sum_{2\leq h_1<\dots<h_s\leq p_1} \\
& (\delta_{i_1,i}\delta_{h_1-1,1}z+\tilde e_{(i_1,h_1-1),(i1)})
(\delta_{i_2,i_1}\delta_{h_2-1,h_1}z+\tilde e_{(i_2,h_2-1),(i_1h_1)})
\dots \\
& \dots
(\delta_{i_s,i_{s-1}}\delta_{h_s-1,h_{s-1}}z+\tilde e_{(i_s,h_s-1),(i_{s-1},h_{s-1})})
(\delta_{i_s,j}\delta_{p_1,h_s}z+\tilde e_{(jp_1),(i_sh_s)})
\,.
\end{split}
\end{equation}
The RHS of \eqref{eq:rectangular-gener-explicit} is a polynomial in $z$, hence (recall Proposition \ref{thm:L1}(b)) 
it uniquely defines elements $w_{ji;k}\in W(\mf{gl}_N,f^{\text{rect}})\subset U(\mf g_{\leq0})$, 
$1\leq i,j\leq r_1$, $0\leq k\leq p_1-1$, such that
\begin{equation}\label{eq:L-rectangular}
L(z)
=
-\id_{r_1}(-z)^{p_1}
+\sum_{k=0}^{p_1-1}W_k(-z)^k
\,\,,\,\,\,\,
W_k=\big(w_{ji;k}\big)_{i,j=1}^{r_1}\in\Mat_{r_1\times r_1}\!\!\! W(\mf{gl}_N,f^{\text{rect}})
\,,
\end{equation}
in accordance to Conjecture \ref{conj1}.
Using equation \eqref{eq:L-rectangular} and Theorem \ref{thm:main2} we,
we can use the Yangian identity \eqref{eq:adler} to find explicit commutation relations among the generators
($1\leq a,b,c,d\leq r_1$ and $0\leq h,k\leq p_1$):
\begin{equation}\label{eq:bracket_rectangular}
[w_{ab;h},w_{cd;k}]=\sum_{n=0}^{\min\{p_1-1-h,k\}}\left(w_{ad;h+n+1}w_{bc;k-n}-w_{ad;k-n}w_{bc;h+n+1}\right)
\,,
\end{equation}
where in the RHS we let $w_{ji;p_1}=-\delta_{ij}$. Hence, $W(\mf{gl}_N,f^\text{rect})$ is isomorphic to the Yangian $Y_{r_1,p_1}$ (the Yangian of $\mf{gl}_{r_1}$ of level $p_1$) introduced by Cherednik
\cite{Che87} as already noticed in \cite{RS99,BK06}.

\subsection{Example 3: minimal nilpotent $f$}\label{sec:Examples.3}

The minimal nilpotent element $f$ in $\mf g=\mf{gl}_N$
is associated to the partition $p=(2,1,\dots,1)$.
In this case the entries of the shift matrix $D$ are $d_{(ih)}=-\delta_{(ih),{(12)}}$, for every $(i,h)\in\mc T$, and 
for our choice of $f$ equation \eqref{eq:La} becomes
\begin{equation}\label{eq:L-minimal-def}
L(z)
=
\Big|
\left(\begin{array}{ccc}
z+e_{(11),(11)} & e_{(12),(11)} & e_{+(11)} \\
1 & z+e_{(12),(12)}-1 & e_{+(12)} \\
e_{(11)+} & e_{(12)+} & z\id_{N-2}+E_{++}
\end{array}\right)
\Big|_{12}\bar1
\,,
\end{equation}
where ($k=1,2$)
\begin{equation*}
\begin{split}
& e_{+(1k)}
=
\left(\begin{array}{lll} e_{(21),(1k)} & \dots & e_{(r1),(1k)} \end{array}\right) 
\,,\\
& E_{++}
=
\left(\begin{array}{ccc} 
e_{(21),(21)} & \dots & e_{(r1),(21)} \\
\vdots & \ddots & \vdots \\
e_{(21),(r1)} & \dots & e_{(r1),(r1)} 
\end{array}\right)
\,\,,\,\,\,\,
e_{(1k)+}
=
\left(\begin{array}{l} e_{(1k),(21)} \\ \vdots \\ e_{(1k),(r1)} \end{array}\right)
\,.
\end{split}
\end{equation*}
We can compute the quasideterminant \eqref{eq:L-minimal-def} by the usual formula \eqref{0304:eq5} in Proposition
\ref{0304:prop}.
As a result we get, after a straightforward computation,
\begin{equation}\label{eq:L-minimal-gener}
\begin{array}{l}
\displaystyle{
\vphantom{\Big(}
L(z)
=
e_{(12),(11)}\bar1
-
(z+e_{(11),(11)})(z+e_{(12),(12)}-1)\bar1
} \\
\displaystyle{
\vphantom{\Big(}
-\left(
(z+e_{(11),(11)})e_{+(12)}-e_{+(11)}
\right)
\left(
z\id_{N-2}+E_{++}-e_{(11)+}e_{+(12)}
\right)^{-1}
} \\
\displaystyle{
\vphantom{\Big(}
\qquad\qquad\qquad\times
\left(
e_{(11)+}(z+e_{(12),(12)}-1)-e_{(12)+}
\right)\bar1
\,.}
\end{array}
\end{equation}
It is then not difficult to rewrite equation \eqref{eq:L-minimal-gener} as
\begin{equation}\label{eq:L1-minimal}
L(z)
=
-z^2-w_{11;1}z+w_{11;0}
- w_{+1}(z\id_{N-2}+W_{++})^{-1}w_{1+}
\,,
\end{equation}
where
\begin{align*}
& w_{+1}
=
\left(\begin{array}{lll} w_{21;0} & \dots & w_{r1;0} \end{array}\right) 
\,,\\
& W_{++}
=
\left(\begin{array}{ccc} 
w_{22;0} & \dots & w_{r2;0} \\
\vdots & \ddots & \vdots \\
w_{2r;0} & \dots & w_{rr;0} 
\end{array}\right)
\,\,,\,\,\,\,
w_{1+}
=
\left(\begin{array}{l} w_{12;0} \\ \vdots \\ w_{1r;0} \end{array}\right)
\,,
\end{align*}
and $w_{11;1}$, $w_{11;0}$ are defined as follows
(all elements lie in $M=U(\mf g)/I\simeq U(\mf g_{\leq0})\otimes F(\mf g_{\frac12})$):
\begin{equation}\label{eq:minimal-gener}
\begin{split}
w_{11;1}
&
=e_{(11),(11)}+e_{(12),(12)}+e_{+(12)}e_{(11)+}-1 \,,\\
w_{11;0}
&
=e_{(12),(11)}-e_{(11),(11)}(e_{(12),(12)}-1)+e_{+(12)}w_{1+}
\\
&+w_{+1}e_{(11)+}-e_{+(12)}W_{++}e_{(11)+}
\,,\\
w_{+1}
&
=e_{+(11)}-e_{(11),(11)}e_{+(12)}+e_{+(12)}W_{++}\,,\\
w_{1+}
&
=e_{(12)+}-e_{(11)+}(e_{(12),(12)}-1)+W_{++}e_{(11)+}
\,,\\
W_{++}
&
=E_{++}-e_{(11)+}e_{+(12)} \,.
\end{split}
\end{equation}
By Theorem \ref{thm:main1}, $w_{11;1}$ and $w_{11;0}$ lie in $W(\mf g,f)$. 
It is a straightforward computation
to check that also the entries of $w_{+1}$, $w_{1+}$ and $W_{++}$ 
lie in $W(\mf g,f)$.
Moreover,  equation \eqref{eq:minimal-gener} defines a map 
$w:\mf g^f\to W$ satisfying the conditions of Premet's Theorem,
confirming Conjecture \ref{conj1}.

We can also compute the commutator between generators by using 
their  explicit formulas \eqref{eq:minimal-gener}.
The non-zero commutators are:
\begin{equation}\label{eq:minimal-brackets}
\begin{split}
& [w_{11;1},w_{+1}]
=
-w_{+1}
\,,\,\,
[w_{11;1},w_{1+}]
=
w_{1+}
\,,\\
&[w_{11;0},w_{+1}]
=
-w_{+1}W_{++}+w_{+1}w_{11;1}
\,,\\
&[w_{11;0},w_{1+}]
=
W_{++}w_{1+}-w_{11;1}w_{1+}
\,,\\
&[w_{1+},w_{+1}]
=
-W_{++}^2+w_{11;1}W_{++}+w_{11;0}\id_{N-2}
\,,\\
&
[(W_{++})_{ij}, (w_{+1})_k]
=\delta_{ik}(w_{+1})_j
\,,\\
&
[(W_{++})_{ij}, (w_{1+})_k]
=-\delta_{jk}(w_{1+})_i
\,,
\\
&[(W_{++})_{ij},(W_{++})_{hk}]
=
\delta_{ik}(W_{++})_{hj}-\delta_{jh}(W_{++})_{ik}
\,.
\end{split}
\end{equation}
Note that the commutation relations in \eqref{eq:minimal-brackets} are the same relations 
defining the Yangian $Y_{N-1,2}(\sigma)$ (the shifted Yangian of $\mf{gl}_{N-1}$ of level $2$) 
as proved in \cite{BK06}. Equations \eqref{eq:minimal-brackets} also agree with \cite{Pre07}.
It is a long but straightforward computation to check that using the commutation relations 
in \eqref{eq:minimal-brackets}
we have $[L(z),L(w)]=0$, thus showing that $L(z)$ satisfies the Yangian identity as stated 
in Theorem \ref{thm:main2}.


\begin{thebibliography}{00} 


\bibitem[BG07]{BG07}
Brundan J., Goodwin S.M.,
\emph{Good grading polytopes},
Proc. Lond. Math. Soc. (3) {\bf 94} (2007), no. 1, 155-180.

\bibitem[BK06]{BK06}
Brundan J., Kleshchev A.,
\emph{Shifted Yangians and finite W-algebras},
Adv. Math. {\bf 200} (2006), no. 1, 136-195.

\bibitem[Cap1902]{Cap1902}
Capelli A.,
\emph{Lezioni sulla teoria delle forme algebriche},
Napoli 1902.

\bibitem[Che87]{Che87}
Cherednik I.,
\emph{A new interpretation of the Gelfand-Tzetlin bases},
Duke Math. J. {\bf 54} (1987), 563-577.


\bibitem[DSK06]{DSK06}
De Sole A., Kac V. G.,
\emph{Finite vs affine W-algebras},
Jpn. J. Math. {\bf 1} (2006), no. 1, 137-261.

\bibitem[DSKV13]{DSKV13}
De Sole A., Kac V. G., Valeri D.,
\emph{Classical $\mc W$-algebras and generalized Drinfeld-Sokolov bi-Hamiltonian systems 
within the theory of Poisson vertex algebras},
Comm. Math. Phys. {\bf 323} (2013), no. 2, 663-711.

\bibitem[DSKV15]{DSKV15}
De Sole A., Kac V. G., Valeri D.,
\emph{Adler-Gelfand-Dickey approach to classical W-algebras 
within the theory of Poisson vertex algebras},
Int. Math. Res. Not. 21 (2015), 11186-11235.

\bibitem[DSKVjems]{DSKVjems}
De Sole A., Kac V. G., Valeri D.,
\emph{Structure of classical (finite and affine) $\mc W$-algebras},
to appear in JEMS,
preprint arXiv:1404.0715

\bibitem[DSKV16a]{DSKV16a}
De Sole A., Kac V. G., Valeri D.,
\emph{A new scheme of integrability for (bi)-Hamiltonian PDE},
Comm. Math. Phys. (2016), to appear,
preprint arXiv:1508.02549.

\bibitem[DSKV16b]{DSKV16b}
De Sole A., Kac V. G., Valeri D.,
\emph{Classical affine $\mc W$-algebras for $\mf{gl}_N$ 
and associated integrable Hamiltonian hierarchies},
Comm. Math. Phys. (2016), to appear,
preprint arXiv:1509.06878.

\bibitem[Dr86]{Dr86}
Drinfeld V.G.,
\emph{Quantum groups},
Proceedings of the International Congress of Mathematicians, 
Vol. 1, 2 (Berkeley, Calif., 1986), 798--820, Amer. Math. Soc., Providence, RI, 1987.

\bibitem[EK05]{EK05}
Elashvili A.G., Kac V.G.,
\emph{Classification of good gradings of simple Lie algebras},
Lie groups and invariant theory, 85-104, 
Amer. Math. Soc. Transl. Ser. 2, 213, 
Amer. Math. Soc., Providence, RI, 2005.

\bibitem[Fed16]{Fed16}
Fedele L., PhD thesis.

\bibitem[GG02]{GG02}
Gan W.L., Ginzburg V.,
\emph{Quantization of Slodowy slices},
Int. Math. Res. Not. (2002), no.5, 243-255.

\bibitem[GGRW05]{GGRW05}
Gelfand I.M., Gelfand S.I., Retakh V. and Wilson R.L.,
\emph{Quasideterminants}
Adv. Math. {\bf 193} (2005), n.1, 56-141.


\bibitem[Kos78]{Kos78}
Kostant B., \emph{On Whittaker vectors and representation theory}, 
Invent. Math. {\bf 48} (1978), n.2, 101-184.


\bibitem[Mol07]{Mol07}
Molev A.,
\emph{Yangians and classical Lie algebras},
Mathematical Surveys and Monographs, 143. 
American Mathematical Society, Providence, RI, 2007.

\bibitem[Pre02]{Pre02}
Premet A.,
\emph{Special transverse slices and their enveloping algebras},
Adv. Math. {\bf170} (2002), 1-55.

\bibitem[Pre07]{Pre07}
Premet A.,
\emph{Enveloping algebras of Slodowy slices and the Joseph ideal},
J. Eur. Math. Soc. (JEMS) {\bf 9} (2007), no. 3, 487-543.

\bibitem[RS99]{RS99}
Ragoucy E., Sorba P.,
\emph{Yangian realisations from finite W-algebras},
Comm. Math. Phys. {\bf 203} (1999), 551-572.

\end{thebibliography}
\end{document}